\documentclass[aos]{imsart}
\usepackage{definitions}
\usepackage{mathrsfs}
\usepackage{verbatim}
\usepackage{bbm} 
\RequirePackage{amsthm,amsmath,amsfonts,amssymb}
\RequirePackage[numbers,sort&compress]{natbib}
\RequirePackage[colorlinks,citecolor=blue,urlcolor=blue,hypertexnames=false]{hyperref}
\RequirePackage{graphicx}
          
\startlocaldefs
\theoremstyle{plain}

\newtheorem{theorem}{Theorem}[section]
\newtheorem{lemma}[theorem]{Lemma}
\newtheorem{corollary}{Corollary}
\theoremstyle{remark}

\theoremstyle{remark}
\newtheorem{definition}[theorem]{Definition}

\usepackage{bookmark}  

\endlocaldefs

\newcommand{\bm}{\boldsymbol}
\begin{document}

\begin{frontmatter}
\title{Optimal Federated Learning for Nonparametric Regression with Heterogeneous Distributed Differential Privacy Constraints}
\runtitle{Optimal Federated Learning}

\begin{aug}
\author[A]{\fnms{T. Tony}~\snm{Cai}\ead[label=e1]{tcai@wharton.upenn.edu}},
\author[A]{\fnms{Abhinav}~\snm{Chakraborty}\ead[label=e2]{abch@wharton.upenn.edu}}
\and
\author[A]{\fnms{Lasse}~\snm{Vuursteen}\ead[label=e3]{lassev@wharton.upenn.edu}}
\address[A]{Department of Statistics and Data Science,
University of Pennsylvania\printead[presep={,\ }]{e1,e2,e3}}
\end{aug}

\begin{abstract}
This paper studies federated learning for nonparametric regression in the context of distributed samples across different servers, each adhering to distinct differential privacy constraints. The setting we consider is heterogeneous, encompassing both varying sample sizes and differential privacy constraints across servers. Within this framework, both global and pointwise estimation are considered, and optimal rates of convergence over the Besov spaces are established.

Distributed privacy-preserving estimators are proposed and their risk properties are investigated. Matching minimax lower bounds, up to a logarithmic factor, are established for both global and pointwise estimation. Together, these findings shed light on the tradeoff between statistical accuracy and privacy preservation. In particular, we characterize the compromise not only in terms of the privacy budget but also concerning the loss incurred by distributing data within the privacy framework as a whole. This insight captures the folklore wisdom that it is easier to retain privacy in larger samples, and explores the differences between pointwise and global estimation under distributed privacy constraints. 
\end{abstract}



\end{frontmatter}

\section{Introduction}
\label{sec:intro}

In today's data-driven world,  the proliferation of personal data and technological advancements has made the protection of privacy
a matter of paramount importance. Developing statistical methods with privacy guarantees is becoming increasingly important. Differential privacy (DP), one of the most widely adopted privacy frameworks, ensures that statistical analysis results do not divulge any sensitive information about the input data. DP was introduced in the seminal work by Dwork et al. \cite{dwork2006differential}. Since its inception, DP has garnered significant academic attention \cite{arachchige2019local,dwork2010differential,dwork2017exposed} and notable applications within industry leaders, including Google \cite{google_privacy}, Microsoft \cite{ding2017collecting}, and Apple \cite{apple2017}. It has also been embraced by governmental entities like the US Census Bureau \cite{abowd2020modernization}. 

A common setting in many real-life applications is the distributed nature of data collection and analysis. For example, medical data is spread across various hospitals in healthcare, customer data is stored in different branches or databases in financial institutions and various modern technologies rely on federated learning from networks of users, see, for example, \cite{li2020federated,konevcny2016federated,hard2018federated,beaufays2019federated,nguyen2022deep}. DP has found applications in many of these domains relating to, for example, healthcare, finance, tech and social sciences, where preserving individuals' data privacy is of utmost concern. In such scenarios, it is vital to develop efficient estimation techniques that respect privacy constraints while harnessing the collective potential of distributed data.

Federated learning is a machine learning paradigm designed to address the challenges of data governance and privacy. It enables organizations or groups, whether from diverse geographic regions or within the same organization, to collaboratively train and improve a shared global statistical model without external sharing of raw data. The learning process occurs locally at each participating entity, which we shall refer to as \emph{servers}. The servers exchange only characteristics of their data, such as parameter estimates or gradients, in a way that preserves privacy of the individuals comprising their data. Federated learning facilitates secure collaboration across industries like retail, manufacturing, healthcare, and financial services, allowing them to harness the power of data analysis while upholding data privacy and security.

Rigorous study of theoretical performance in federated learning settings with communication constraints has been conducted in, for example, bandwidth constraint parametric problems  \cite{duchi_optimality_2014,braverman2016communication,han2018geometric,barnes2020lower,cai2020distributed,pmlr-v125-acharya20b,szabo2022optimal} and bandwidth constraint nonparametric estimation and testing  \cite{szabo2020adaptive,pmlr-v80-zhu18a,szabo2022distributed,cai:2021:distributed,szabo2023optimal}. Under DP constraints, theoretical performance in federated learning settings have been studied for various parametric estimation and testing problems \cite{liu2020learning,pmlr-v206-acharya23a_user_level_LDP,levy2021learning,narayanan2022tight}. 
Federated learning settings where each server's sample consists of one individual observation (referred to as \emph{local} differential privacy settings) have been studied in many-normal-means model, discrete distributions and parametric models \cite{duchi2013local,duchi2018minimax,barnes2020fisher,pmlr-v119-acharya20a_context_aware_LDP,min_barg} and nonparametric density estimation \cite{sart_density_LDP,kroll_density_at_a_point_LDP,butucea_LDP_adaptation}.

This paper investigates the statistical optimality of federated learning under DP constraints in the context of nonparametric regression. We consider a setting where data is distributed among different entities, such as hospitals, that are concerned about sharing their data with other entities due to privacy concerns for their patients. Each entity communicates a transcript that fulfills a distinct DP requirement, and we assume a setting with $m$ servers, each with $n_j$ observations where $j = 1,\dots,m$. 

Our goals are two-fold: firstly, to establish optimal rates of convergence, measured in terms of minimax risk, for estimating the nonparametric regression function while adhering to DP constraints; secondly, to construct a rate-optimal estimator under these DP constraints. We explore both global and pointwise estimation, aiming to provide quantifiable measures of the trade-off between accuracy and privacy preservation. These convergence rates offer insights into the best achievable estimation performance in distributed settings while ensuring privacy. Recognizing that global estimation exhibits different characteristics compared to its pointwise counterpart in the classical setting \cite{cai2003rates}, we investigate how DP constraints impact global and pointwise estimation risks differently.

\subsection{Problem formulation}

We will begin by formally introducing the general framework of distributed estimation under privacy constraints. Consider a family of probability measures  $\{P_f\}_{f \in \cF}$ on the measurable space $(\mathcal{Z},\mathscr{Z})$, parameterized by $f \in \cF$. We consider a setting where $N = \sum_{j=1}^m n_j$ i.i.d. observations are drawn from a distribution $P_f$ and distributed across $m$ servers. Each server $j=1,\dots,m$ holds $n_j$ observations.

\begin{figure}[htb]
	\centering
	\includegraphics[width=0.9\textwidth]{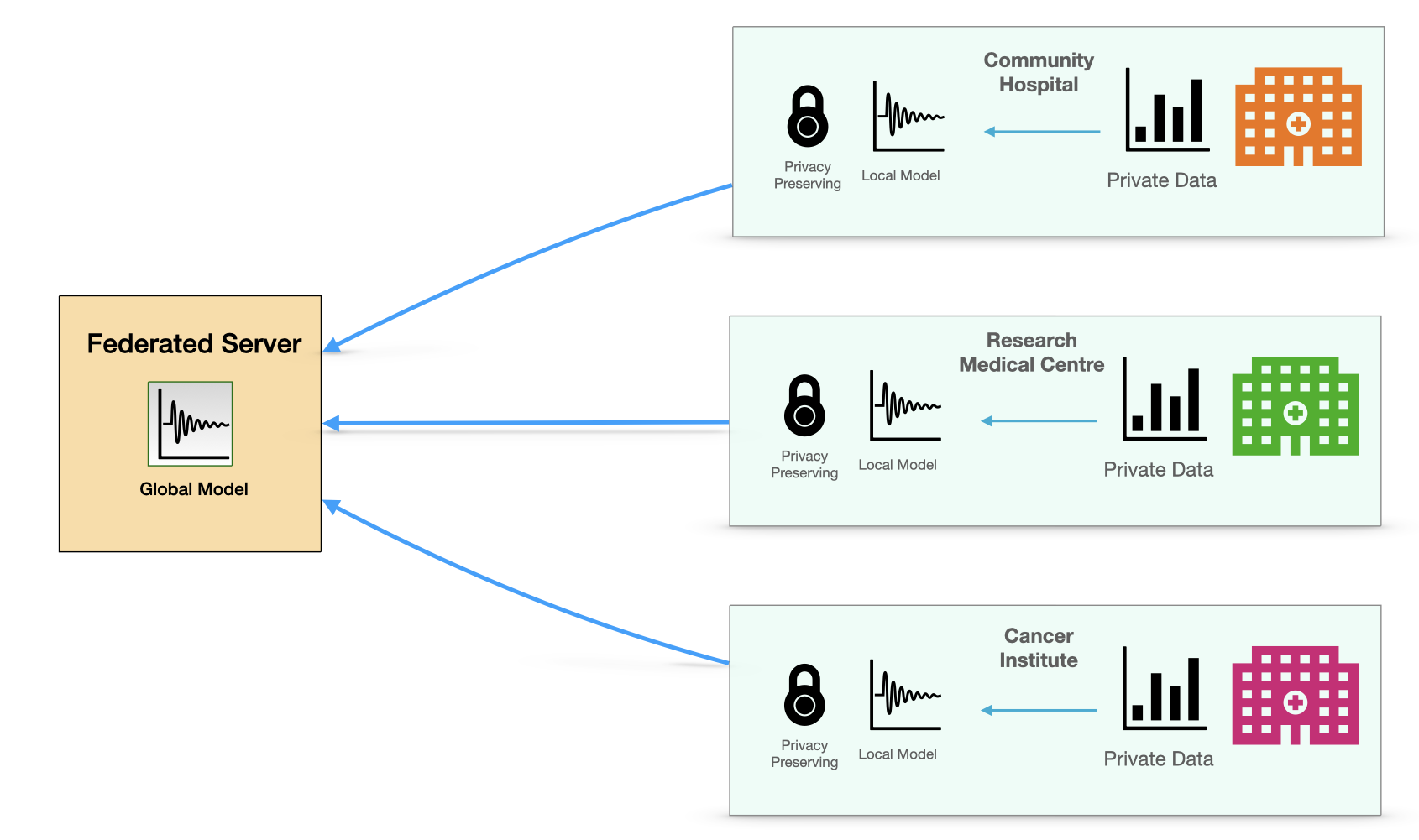} 
	\caption{An illustration of the federated learning framework.}
	\label{fig: FLplot}
\end{figure}

Let us denote by $Z^{(j)} = \{Z^{(j)}_i\}_{i=1}^{n_j}$ the $n_j$ realizations from $P_f$ on the $j$-th server. For each server, we output a (randomized) transcript $T^{(j)}$ based on $Z^{(j)}$, where the law of the transcript is given by a distribution conditional on $Z^{(j)}$, $\P( \cdot | Z^{(j)})$ on a measurable space $(\cT, \mathscr{T})$. The transcript $T^{(j)}$ has to satisfy a $(\varepsilon_j,\delta_j)$-differential privacy constraint, which is defined as follows.

\begin{definition}\label{def:differential_privacy}
	The transcript $T^{(j)}$ is $(\varepsilon_j,\delta_j)$-differentially private if for all $A \in \mathscr{T}$ and $z,z' \in \mathcal{Z}^{n_j}$ differing in one individual datum, it holds that
	\begin{equation*}
		\P \left( T^{(j)} \in A | Z^{(j)} = z \right) \leq  e^{\varepsilon_j} \P \left( T^{(j)} \in A | Z^{(j)} = z' \right)  + \delta_j.
	\end{equation*}
\end{definition}

In the above definition, ``differing in one datum" should be understood in terms of being Hamming distance ``neighbors." To clarify, the local datasets $Z^{(j)}$ and $\tilde{Z}^{(j)}$ are deemed \emph{neighboring} if their Hamming distance is at most $1$. The Hamming distance is calculated over $\mathcal{Z}^{n_j} \times \mathcal{Z}^{n_j}$. In other words, $\tilde{Z}^{(j)}$ can be derived from $Z^{(j)}$ by modifying at most one of the observations  $Z^{(j)}_1,\dots,Z^{(j)}_{n_j}$. The smaller the value of $\epsilon_j$ and $\delta_j$, the more stringent the privacy constraint. We shall consider $\varepsilon_j \leq C_{\varepsilon}$ for $j=1,\dots,m$ for a fixed but arbitrarily constant $C_{\varepsilon} > 0$, where the choice of the constant does not affect the rates in the results derived.

We focus on distributed protocols that apply to situations in which sensitive data is held by multiple parties, each generating an output while ensuring differential privacy. Within such a distributed protocol, the transcripts from each server only depend on its local data, and no information is exchanged between the servers. This occurs, for example, when multiple trials concerning the same population are conducted, but each location (e.g. hospital) does not wish to pool their original data because of privacy concerns. 

Each server transmits its transcript to the central server. The central server, utilizing all transcripts $T :=(T^{(1)}, \dots,T^{(m)})$, computes an estimator $\hat f: \cT^m \to \cF$. We refer to the pair $( \hat{f}, \{\left(\P( \cdot | z)\right)_{z \in \cZ} \}_{j=1}^m)$ as a \emph{distributed estimation protocol}, which we shall sometimes just denote as $\hat{f}$. We denote the vector of the differing DP levels by $(\bm \varepsilon,\bm\delta) = \{(\varepsilon_j,\delta_j)\}_{j=1}^m$ and denote the class of \emph{distributed estimation protocols}, i.e. $\left( \hat{f}, \{\left(\P( \cdot | z)\right)_{z \in \cZ} \}_{j=1}^m\right)$ satisfying Definition \ref{def:differential_privacy}, with $\cM(\bm \varepsilon,\bm\delta)$. We let $\P_f$ denote the joint law of transcripts and the $N = \sum_{j=1}^m n_j$ i.i.d. observations generated from $P_f$. We let $\E_f$ denote the expectation corresponding to $\P_f$.

In the context of nonparametric regression, the distributed estimation problem arises when data is distributed among multiple servers. Specifically, for each server $j$, the data $Z^{(j)} = \{(Y^{(j)}_i, X^{(j)}_i)\}_{i=1}^{n_j}$ consists of $n_j$ pairs of observations $(Y^{(j)}_i, X^{(j)}_i)$. Here, $X^{(j)}_i$ represents the input variable, and $Y^{(j)}_i$ represents the corresponding response variable.

We assume that under $P_f$, $X^{(j)}_i$ and $Y^{(j)}_i$ are generated by the relationship
\begin{equation}\label{eq:dynamics_generating_data}
	Y^{(j)}_i = f\left(X^{(j)}_i\right) + \xi^{(j)}_i, \quad X^{(j)}_i \sim U[0,1].
\end{equation}
Here, $f$ is an unknown function representing the underlying relationship between the input and response variables. The term $\xi^{(j)}_i$ represents random noise, assumed to be independent of $X^{(j)}_i$, and follows a Gaussian distribution with mean $0$ and known variance $\sigma^2$. Without loss of generality, we shall assume $\sigma = 1$ throughout the paper.

The aim is to estimate the function $f$ based on the distributed data. The difficulty of this estimation task arises from both the distributed nature of the data and privacy constraints that limit the sharing of information between servers.
As in the conventional decision-theoretical framework, for global estimation,  the estimation accuracy of a distributed estimator $\hat f \equiv \hat f(T)$ is measured by the integrated mean squared error (IMSE), $\E_{f} \|\hat f - f\|_2^2$, where the expectation is taken over the randomness in both the data (under $P_f$) and construction of the transcripts. As in the conventional  framework, a quantity of particular interest in federated learning is the \emph{global minimax risk} for the distributed private protocols over function class $\cF$,
\begin{equation}\label{eq:def-global-risk}
	\inf_{\hat f \in \cM(\bm \varepsilon,\bm\delta)} \sup_{f \in \cF} \E_{f} \|\hat f - f\|_2^2.
\end{equation}
The global risk characterizes the difficulty of the distributed learning problem over the function class $\cF$ when trying to infer the entire function underlying the data whilst adhering to the heterogeneous privacy constraints.

Besides global estimation, it is also of interest to estimate $f$ at a fixed point  $ x_0 \in (0,1)$ under the mean squared error (MSE). The \emph{pointwise minimax risk} in that case is given by 
\begin{equation}\label{eq:def-pointwise-risk}
	\inf_{\hat f \in \cM(\bm \varepsilon,\bm\delta)} \sup_{f \in \cF} \E_{f} (\hat f(x_0) - f(x_0))^2, \; \text{ for } x_0 \in (0,1),
\end{equation}
where $\hat f(x_0)$ denotes the estimated function value at $x_0 \in (0,1)$. The pointwise risk is particularly useful in understanding the behavior of estimators at specific points within the domain, which can be crucial in applications where certain regions are of particular interest or have higher consequences associated with estimation errors. It is known that in the classical setting, without privacy constraints, there are important differences  between the global risk and pointwise risk in terms of performance.  See, for example, \cite{cai2002block}.
 
We consider estimating $f$ over the Besov ball of radius $R>0$, denoted as  $\cB^{\alpha, R}_{p,q}[0,1]$ (defined in \eqref{Besov.ball}), where $p \geq 2$, $q \geq 1$ and ${\alpha} - 1/p > 1/2$.
This Besov space offers a suitable framework for analyzing functions with specific smoothness characteristics. Operating within this space allows us to encompass diverse function classes, accommodating varying levels of smoothness and complexity.

\subsection{Main contribution}

We quantify the cost of differential privacy for both the minimax global risk given by \eqref{eq:def-global-risk} and the pointwise risk as in \eqref{eq:def-pointwise-risk}. To achieve this, we introduce two differentially private estimators -- one for global and one for pointwise estimation. We obtain matching minimax lower bounds, up to logarithmic factors, thereby establishing their optimality. 

Our analysis reveals interesting phenomena, that go unobserved in settings where servers are assumed to have homogeneous privacy budgets. Further discussion on these broader findings is deferred to Section \ref{sec:main_results}. The results for the homogeneous case, where privacy budgets are equal among servers ($\varepsilon_j = \varepsilon$, $\delta_j = \delta$, and $n_j = n$ for $j=1,\dots,m$), yield novel insights. In this case, our results yield the following minimax rate for global estimation,
\begin{equation}\label{eq:minimax-rate-hom-intro-global}
	\inf_{\hat f \in \cM(\bm \varepsilon,\bm \delta)} \sup_{f \in \cB^{\alpha,R}_{p,q}} \E_{f} \|\hat f - f\|_2^2 \asymp \min \left\{M_{m,n} \cdot  \left((m n^2 \varepsilon^2)^{- \frac{2\alpha}{2\alpha+2}} + (m n )^{- \frac{2\alpha}{2\alpha+1}} \right) , 1 \right\},
\end{equation}
where $M_{m,n} \geq 1$ is a sequence at most of the order $\log(mn)\cdot\log(1/\delta)$. The rate $\left( m n  \right)^{- \frac{2\alpha}{2\alpha+1}}$ is the minimax rate for the global risk in the unconstrained problem, and is attained whenever $n\varepsilon^2 \gtrsim \left( mn\right)^{ \frac{1}{2\alpha+1}}$. The unconstrained optimal rate is attainable (up to a possibly poly-logarithmic factor) under DP constraints in the homogeneous setting as long as $n \varepsilon^2 \gtrsim \left( mn\right)^{ \frac{1}{2\alpha+1}}$. Whenever $n \varepsilon^2 \ll \left( mn\right)^{ \frac{1}{2\alpha+1}}$, the first term dominates and the minimax rate becomes $(m n^2 \varepsilon^2)^{- \frac{2\alpha}{2\alpha+2}}$. As expected in this regime, a smaller $\varepsilon$, which indicates a stronger privacy guarantee, results in an larger minimax estimation error. Whenever $\varepsilon \ll (\sqrt{m} n)^{-1}$, consistent estimation ceases to be possible altogether.

This result recover the known minimax rates under local DP constraints (i.e. $n=1$) that were derived for the problem of nonparametric density estimation for the $L_2$-risk in \cite{butucea_LDP_adaptation} and the squared Hellinger in \cite{sart_density_LDP}, up to logarithmic factor differences. When $n > 1$, the different powers with which $n$ and $m$ appear in the minimax rate reveal an important difference between the general distributed setting and local DP; if one distributes $N=mn$ observations across $m$ machines, the task becomes more challenging as the $N$ observations are spread over a greater number of machines, rather than having a large number of observations on a smaller number of machines. This phenomenon has an intuitive explanation; it is easier to retain privacy in larger samples, as each individual's data will have only a small influence on the aggregate statistics of interest.

For pointwise estimation, we establish the minimax rate in the homogeneous setting;
\begin{equation}\label{eq:minimax-rate-hom-intro-pointwise}
	\inf_{\hat f \in \cM(\bm \varepsilon,\bm\delta)} \sup_{f \in \mathcal{B}^{\alpha,R}_{p,q}} \E_{f} |\hat f(x_0) - f(x_0)|^2 \asymp  \min \left\{ M_{m,n} \cdot \left(\left( m n^2 \varepsilon^2  \right)^{- \frac{2\nu}{2\nu+2}} +\left( m n  \right)^{- \frac{2\nu}{2\nu+1}} \right) , 1\right\},
\end{equation}
where $M_{m,n} \geq 1$ is a sequence at most of the order $\log(mn)$. The rate reveals similar phenomena as the one for the global risk above, where for $n=1$ we recover the known minimax rate for the problem of nonparametric density estimation for the pointwise risk under local DP constraints as studied in \cite{kroll_density_at_a_point_LDP}. An important difference is the quantity $\nu = \alpha - 1/p$ appearing in the exponent instead of $\alpha$. This implies that privacy constraints impact pointwise estimation differently than global estimation, with the Besov parameter $p$ influencing both the relative privacy cost and the distribution of the $N=mn$ observations, as discussed further in Section \ref{ssec:homogeneous_budgets}.

Our findings have substantial implications for the development of federated learning algorithms that balance distributed privacy with accuracy. A clear understanding of the optimal convergence rate under distributed privacy constraints allows the design of algorithms that strike the right balance between accuracy and privacy trade-offs. This study contributes significantly to the growing knowledge on distributed settings for privacy-preserving machine learning, offering valuable insights for future research in this domain.

\subsection{Related Work}

The nonparametric regression setting considered in this work bears relationships with that of nonparametric density estimation as studied in the privacy setting for global risk  \cite{duchi2018minimax,sart_density_LDP,butucea_LDP_adaptation} and pointwise risk \cite{kroll_density_at_a_point_LDP}. The aforementioned papers consider the setting of local DP, in which the privacy protection is applied at the level of individual data entries or observations. This corresponds to the case wherein $n_j =1$ for $j=1,\dots,m$ in our setting. 

Distributed DP as considered in this paper, where DP applies at the level of the local sample consisting of multiple observations, has been studied for the homogeneous estimation setting for discrete distributions \cite{liu2020learning,pmlr-v206-acharya23a_user_level_LDP} and parametric mean  estimation \cite{levy2021learning,narayanan2022tight}. In the paper \cite{canonne2023private}, the authors consider discrete distribution testing in a two server setting ($m=2$) with differing DP constraints.

Settings in which the full data is assumed to be on a single server (i.e. $m=1$), where a single privacy constraint applies to all the observations, have also been studied for various parametric high-dimensional problems \cite{smith2011privacy,dwork2014analyze,bassily2014private,kamath2019privately,kamath2020private,cai2021cost,narayanan2022private}. The problem of mean estimation with a single server having heterogeneous privacy constraints for each individual observation have been studied in \cite{fallah2023optimal,chaudhuri2023mean}. 

\subsection{Organization of the paper}

The rest of the paper is organized as follows. We conclude this section with notation, definitions, and assumptions. Section \ref{sec:main_results} summarizes and discusses the minimax optimal convergence rates for global and pointwise risks under privacy constraints. In Section \ref{sec:upper-bounds}, we present distributed estimation procedures that achieve optimal global and pointwise risk while adhering to distributed privacy constraints, along with an upper bound on their statistical performance.  The matching minimax lower bounds for the global and pointwise risks are derived in Section \ref{sec:lower-bounds}. The proofs of the main results are given in the Supplementary Material \cite{Cai2023FL-NP-Regression-Supplement}.

\subsection{Notation, definitions and assumptions}

Throughout the article, we shall write $N := \sum_{j=1}^m n_j$ and consider asymptotics in $m$, the $n_j$'s and the privacy budget $(\bm \epsilon, \bm \delta) := \{\varepsilon_j,\delta_j\}_{j=1}^m$, where we assume that $N \to \infty$. For two positive sequences $a_k$, $b_k$ we write $a_k\lesssim b_k$ if the inequality $a_k \leq Cb_k$ holds for some universal positive constant $C$. Similarly, we write $a_k\asymp b_k$ if $a_k\lesssim b_k$ and $b_k\lesssim a_k$ hold simultaneously and let $a_k\ll b_k$ denote that $a_k/b_k = o(1)$. 

We use the notations $a\vee b$ and $a\wedge b$ for the maximum and minimum, respectively, between $a$ and $b$. For $k \in \N$, $[k]$ shall denote the set $\{1,\dots,k\}$. Throughout the paper $c$ and $C$ denote universal constants whose value can differ from line to line. The Euclidean norm of a vector $v \in \mathbb{R}^d$ is denoted by $\|v\|_2$. For a matrix $M \in \R^{d \times d}$, the norm $M \mapsto \| M \|$ is the spectral norm and $\text{Tr}(M)$ is its trace. Furthermore, we let $I_d$ denote the $d \times d$ identity matrix.

Throughout this paper, we shall let $\nu := \alpha - 1/p > 1/2$, which is a required assumption for estimation in Besov spaces (see e.g. \cite{Ibragimov1997}). We let $\mathcal{B}^{\alpha,R}_{p,q}$ denote the closed Besov ball of radius $R$, i.e. $\{ f \in \mathcal{B}^{\alpha}_{p,q}[0,1] : \| f\|_{\cB^{\alpha}_{p,q}} \leq R \}$, where $R > 0$ is taken to be a constant.

For random variables $U$ and $V$ with probability measures $P$ and $Q$ defined on the same measurable space, we let $D_{\mathrm{TV}}(U,V)$ denote the total variation norm between $P$ and $Q$, i.e. $\|P-Q\|_{\mathrm{TV}}$. Whenever $P \ll Q$, we write $D_{\mathrm{KL}}(U,V)$ for the Kullback-Leibler divergence between $P$ and $Q$: $D_{\mathrm{KL}}(P;Q) = \int \log \frac{dP}{dQ} dP$. Our lower bound results hold for transcripts taking values in standard Borel measure spaces. Different measure spaces or larger sigma-algebras can be considered (which only make the privacy constraint more stringent, see e.g. \cite{wasserman2010statistical}) as long as the quantities in the proofs are appropriately measurable.
\section{Minimax optimal rates of convergence}
\label{sec:main_results}

In this section, we present our primary findings regarding the minimax rate of convergence under DP constraints. Our results address both the global and pointwise risks.

For the global risk, the minimax rates are encapsulated in the upper bound of Theorem \ref{thm:minimax-upper-bound} and the lower bound of Theorem \ref{thm:minimax-lower-bound}, derived in Sections \ref{ssec:global_estimation_risk_estimator} and \ref{ssec:lowerbound}. Similarly, for the pointwise risk, our findings are summarized in Theorems \ref{thm:upper-bound-pointwise-risk} and \ref{thm:lower_bound-pointwise-risk}, in the form of an upper bound and lower bound respectively, in Sections \ref{ssec:pointwise_estimation_risk_estimator} and \ref{ssec:pointwise_lowerbound}. Together, these theorems are summarized by the following result.

\begin{theorem}\label{thm:rate_summary}
	For $\gamma > 0$, let $D > 0$ be the number solving the equation
	\begin{equation}\label{eq:dimension_determining_equation_global}
		D^{2\gamma+2} = \sum_{j=1}^m \left( n_j^2 \varepsilon_j^2 \right) \wedge \left(n_j D\right).
	\end{equation}
	Taking $\gamma = \alpha$, the minimax rate for the global risk is given by
	\begin{equation*}
		\inf_{\hat f \in \cM(\bm \varepsilon,\bm \delta)}\sup_{f \in \cB^{\alpha,R}_{p,q}} \E_{f} \|\hat f - f\|_2^2 \asymp  \left(M_{N} D^{-2\alpha} \wedge 1 \right),
	\end{equation*}
	whenever for all $j=1,\dots,m$ we have $ \delta_j \lesssim (n_j^{1/2}\varepsilon_j^2 (D\vee 1)^{-1})^{1+\kappa}$ for some $\kappa > 0$ and where $M_{N} \geq 1$ is a sequence of the order at most $\log (N) \log (1/ \min_{j\in[m]}\delta_j)$.

	For $\gamma = \nu$, the minimax rate for the pointwise risk is given by
	\begin{equation*}
		\inf_{\hat f \in \cM(\bm \varepsilon,\bm\delta)} \sup_{f \in \mathcal{B}^{\alpha,R}_{p,q}} \E_{f} \left|\hat f(x_0) - f(x_0)\right|^2 \asymp \left( M_{N} D^{-2\nu} \wedge 1\right),
	\end{equation*}
	whenever $\sum_{j} n_j \delta_j \to 0$, for a sequence $M_{N} \geq 1$ of the order at most $\log (N)$. 
\end{theorem}
We briefly comment on the derived result. First, we note that a unique positive solution to \eqref{eq:dimension_determining_equation_global} always exists. To see this, note that the exponent $2\gamma+2 > 2$ implies that the left-hand side is smaller than the right-hand side for $D > 0$ small enough, whilst the right-hand side grows linearly for small enough $D>0$. Furthermore, the right-hand side increases sublinearly in $D$, whilst the left-hand side increases superlinearly (strictly so). 

When the privacy budget is large enough (e.g. $\varepsilon_j = \infty$ for $j=1,\dots,m$), $D$ can be seen to correspond with the `effective resolution level' of the estimation problem. That is, $D$ would be proportional to the number of wavelet coefficients needed to obtain a wavelet estimator that attains the optimal estimation rate, see for example \cite{donoho1998minimax}. For $\alpha > 0$ smooth functions in a Besov space, the optimal resolution level of a wavelet estimator would correspond to $(1+2\alpha)^{-1} \lceil \log_2 N \rceil$ for the global risk. However, under privacy constraints, the effective resolution level changes to $(2+2\alpha)^{-1} \lceil \log_2 D \rceil$, which can be substantially different from the case without privacy constraints. We present several specific cases of Theorem \ref{thm:rate_summary} through corollaries that encapsulate its various implications, as discussed in Sections \ref{ssec:homogeneous_budgets} and \ref{ssec:heterogeneous}.

The upper bounds for both types of risk (as given in Theorems \ref{thm:minimax-upper-bound} and \ref{thm:upper-bound-pointwise-risk}) are derived by constructing two estimators. One is proven optimal for global risk, while the other is optimal for pointwise risk. The construction of these estimators is detailed in Section \ref{sec:upper-bounds}. Notably, the optimal estimators for each risk type take distinct forms and employ different privacy mechanisms.

Both of these lower bounds require a different technique. For the global risk, the lower bounding technique is reminiscent  of the score attack of \cite{cai2021cost,cai2023score}, which is a generalization of the tracing adversary argument of \cite{bun2014fingerprinting,dwork2015robust}. We describe the technique in detail in Section \ref{ssec:lowerbound}. In case of the pointwise risk, we employ a coupling argument akin to \cite{coupling_acharya,karwa2017finite} in conjunction with Le Cam's two point method (see e.g. \cite{lecam1973convergence,Yu1997}). The technique for the pointwise risk lower bound is described in Section \ref{ssec:pointwise_lowerbound}. Whilst the techniques differ, a similarity is that they both account for the differences in the required levels of privacy between the servers, with the quantity $D > 0$ in \eqref{eq:dimension_determining_equation_global} being the outcome of balancing a bias-variance trade-off, where the variance for each of the servers is either dominated by the (local) noise in the data itself or by the privacy requirement of the server.

\subsection{The homogeneous setting}
\label{ssec:homogeneous_budgets}

Let us start by studying the case where all machines have both an equal amount of observations, as well as privacy budgets. The following result describes the global risk behavior under DP constraints when the servers are homogeneous in both the number of observations, as well as the privacy constraints they adhere to.

\begin{corollary}\label{thm:equal_budget_global_risk}
	Suppose that $n_j = n$, $\varepsilon_j = \varepsilon$, $\delta_j = \delta$ for $j=1,\dots,m$ and assume that $\delta \lesssim  (\varepsilon^2/ \sqrt{m})^{1+\kappa}$ for some $\kappa >0$. Then, the global minimax risk over $\cM(\bm \varepsilon,\bm \delta)$ satisfies \eqref{eq:minimax-rate-hom-intro-global}.
\end{corollary}
Whenever $  n \varepsilon^2 \ll \left( mn\right)^{ \frac{1}{2\alpha+1}}$, we have that
$$
\inf_{\hat f \in \cM(\bm \varepsilon,\bm \delta)}\sup_{f \in \cB^{\alpha,R}_{p,q}} \E_{f} \|\hat f - f\|_2^2 \asymp M_{m,n} \left( m n  \right)^{- \frac{2\alpha}{2\alpha+1}} \left( m^{\frac{1}{2\alpha+1}}  n^{- \frac{2\alpha}{2\alpha+1}} \varepsilon^{-2} \right)^{\frac{2\alpha}{2\alpha+2}}, 
$$
which indicates that the minimax estimation error becomes larger than the unconstrained minimax rate ($( m n )^{- \frac{2\alpha}{2\alpha+1}}$) by a factor of $( m^{\frac{1}{(2\alpha+1)}}  n^{- \frac{2\alpha}{2\alpha+1}} \varepsilon^{-2} )^{\frac{2\alpha}{2\alpha+2}}$ (ignoring the logarithmic factor). This factor can be seen to capture the cost of privacy in terms of the global risk. A smaller $\varepsilon$ results in an increase in minimax estimation error, where larger smoothness exacerbates the increase. 

A second observation that can be made on the basis of the privacy cost factor, is the cost of distributing observations in a privacy setting. That is to say, if one distributes $N=mn$ observations across $m$ machines, the task becomes more challenging as the $N$ observations are spread over a greater number of machines, rather than having a large number of observations on a smaller number of machines. The relative cost of distributing observations is also revealed to be related to the smoothness, where a larger smoothness again exacerbates the relative cost of distributing data. This observation confirms a folklore understanding that it is easier to retain privacy within a larger crowd. Distributing data across more machines means that each machine needs to add additional noise, to compensate for an overall lack of observations. It also affirms that local differentially private methods perform relatively poorly in multiple observation settings and that applying a privacy constraint at an observation level is comparatively costly.
 
Classically, the pointwise risk is known to be subject to different phenomena than the global risk over the Besov spaces \cite{cai2003rates}. Writing $\nu = \alpha - 1/p$ and assuming $\alpha > 1/p$, it is known that the unconstrained pointwise minimax risk satisfies
\begin{equation}\label{eq:pointwise_risk_optimal_rate}
	\inf_{\hat f } \sup_{f \in \mathcal{B}^{\alpha,R}_{p,q}} \E_{f} |\hat f(x_0) - f(x_0)|^2 \asymp \left( m n  \right)^{- \frac{2\nu}{2\nu+1}}.
\end{equation}
Compared to the unconstrained global risk, this indicates that the estimation error at a point is subject to a fundamentally slower convergence rate than the global estimation minimax rate, where the $\ell_p$-norm used to measure the smoothness of the Besov ellipsoid influences the minimax estimation performance. Roughly speaking, the ``pointwise'' integrability of the derivatives of the function underlying the data impacts the problem of estimation at a point, whilst the global risk remains unaffected. This effect disappears for H\"older alternatives, where $p = \infty$ and the minimax rate for the global risk and the pointwise risk coincide.

The main theorem on the minimax risk for pointwise estimation leads to the following result for the homogeneous setting.
\begin{corollary}\label{thm:pointwise_risk_homogeneous_setting}
	Suppose that $n_j = n$, $\varepsilon_j = \varepsilon$, $\delta_j = \delta$ for $j=1,\dots,m$ and $\delta \ll (mn)^{-1}$. Then, for $x_0 \in [0,1]$, the pointwise minimax risk at $x_0$ over the class $\cM(\bm \varepsilon,\bm \delta)$ satisfies \eqref{eq:minimax-rate-hom-intro-pointwise}.
	
\end{corollary} 
The minimax rate for the pointwise risk seemingly takes on a similar form as that of the global risk and it coincides with the global risk whenever $p = \infty$. However, for finite values of $p$, the cost of privacy can be seen to differ. In particular, to attain the unconstrained optimal pointwise minimax rate \eqref{eq:pointwise_risk_optimal_rate}, it can be seen that a relatively larger $\varepsilon$ is needed, where a smaller value of $p$ in fact exacerbates the demand. More precisely, whenever $\left( mn\right)^{ \frac{1}{2\alpha+1}} \lesssim n \varepsilon^2 \ll \left( mn\right)^{ \frac{1}{2\nu+1}}$, the pointwise risk suffers from the DP constraints, whereas the global risk performance is the same as in the problem without the DP constraints.

Whenever $n\varepsilon^2 \ll \left( mn\right)^{ \frac{1}{2\nu+1}}$, comparing \eqref{eq:minimax-rate-hom-intro-pointwise} to \eqref{eq:pointwise_risk_optimal_rate} shows that the minimax rate of the classical (unconstrained) pointwise risk increases by a factor of $( m^{\frac{1}{2\nu+1}}  n^{- \frac{2\nu}{2\nu+1}} \varepsilon^{-2} )^{\frac{2\nu}{2\nu+2}}$ (ignoring the logarithmic factor). This shows that the pointwise risk is subject to a similar cost-relationship as the global risk. What is similar is that more stringent privacy demands in terms of a smaller $\varepsilon$ translate to an increased cost in terms of the pointwise risk. However, the relative increase in privacy cost resulting from a decrease in $\varepsilon$ for the case of pointwise risk, is smaller than the relative increase in privacy cost of the global risk, where this discrepancy is further exacarbated for smaller values of $p$. This shows that stringent privacy demands are comparatively less costly for the pointwise risk. 

On the other hand, the cost of distributing observations (i.e. increasing $m$ when distributing $N = nm$ observations) is relatively larger for smaller values of $p$. That is to say, differentially private estimation in pointwise risk suffers less from stringent per machine privacy demands, while it suffers more from the fact that data is distributed before privacy preservation is applied. This surprising phenomenon shows that in a distributed setting with privacy constraints, the distribution of the data across servers impacts the rate differently depending on the inferential task at hand.

\subsection{The heterogeneous setting}
\label{ssec:heterogeneous}

While the homogeneous setting described in the introduction serves to illustrate fundamental phenomena, real-world scenarios often involve heterogeneous data and privacy constraints across various data silos. In applications, data may not be uniformly distributed among different sources. For instance, consider cases where data is observed and processed locally, as in the context of hospitals. The results presented here highlight the optimal estimation under differential privacy in such a heterogeneous setting.

Theorems \ref{thm:minimax-lower-bound} and \ref{thm:minimax-upper-bound} describe the minimax rate for the global risk for the full spectrum of possibilities in terms of heterogeneous constraints. Similarly, Theorems \ref{thm:lower_bound-pointwise-risk} and \ref{thm:upper-bound-pointwise-risk} describe the minimax rate for the local risk in the heterogeneous setting. Here, for the sake of clarity of interpretation, we will focus on two different regimes of privacy budgets. For both regimes, whenever $\min_j \varepsilon \gtrsim 1/(mn)$, we require that $\min_j \delta_j \ll 1/(mn)^2$,  which translates to $\delta$ having no further impact on the minimax performance except for incurring a logarithmic factor in case of the global risk. For the first regime, we shall consider privacy budgets where the no single server has much more data than the other servers, comparatively to the stringency in terms of the DP parameter $\varepsilon$. This amounts to
\begin{equation}\label{eq:no_dominant_budget}
	\left( \sum_{j=1}^m n_j^2 \varepsilon_j^2 \right)^{\frac{1}{2\gamma+2}} \geq \max_j \, n_j \varepsilon_j^2,
\end{equation}
where $N = \sum^m_{j=1} n_j$ and $\gamma = \alpha$ or $\gamma = \nu$ for the global and pointwise risk respectively. The following result describes the minimax rate in this regime for the global and pointwise risks. 

\begin{corollary}\label{thm:heterogeneous_no_dominant_budget_both_risks}
	Suppose that $(\bm{\epsilon},\bm{\delta})$ is such that $\sum_{j=1}^m n_j^2 \varepsilon_j^2 \to \infty$, for $j=1,\dots,m$ we have $\delta_j \lesssim (\varepsilon_j^2/\sqrt{m})^{1+\kappa}$ for some $\kappa > 0$ and \eqref{eq:no_dominant_budget} holds with $\gamma = \alpha$. Then, it holds that
	\begin{equation*}
		\inf_{\hat f \in \cM(\bm \varepsilon,\bm \delta)} \sup_{f \in \cB^{\alpha,R}_{p,q}} \E_{f} \left\|\hat f - f\right\|_2^2 \asymp M_{N} \left( \sum_{j=1}^m n_j^2 \varepsilon_j^2 \right)^{- \frac{2\alpha}{2\alpha+2}}.
	\end{equation*}
	for some $M_{N} \geq 1$ of the order at most $\log(N) \cdot \log(1/\min_j \delta_j)$.

	If \eqref{eq:no_dominant_budget} holds for $\gamma = \nu$ and $\sum_{j=1}^m n_j \delta_j \to 0$, it holds that
	\begin{equation*}
		\inf_{\hat f \in \cM(\bm \varepsilon,\bm\delta)} \sup_{f \in \mathcal{B}^{\alpha,R}_{p,q}} \E_{f} \left|\hat f(x_0) - f(x_0)\right|^2 \asymp M_{N} \left( \sum_{j=1}^m n_j^2 \varepsilon_j^2 \right)^{- \frac{2
				\nu}{2\nu+2}}
	\end{equation*}
	for some $M_{N} \geq 1$ of the order at most $\log(N)$.
\end{corollary}
In such a setting, the behaviour in terms of the privacy cost is similar to that described by Corollaries \ref{thm:equal_budget_global_risk} and \ref{thm:pointwise_risk_homogeneous_setting}. A first glance shows that in the distributed privacy setup, the problem is much more difficult compared to the problem without privacy constraints: the rate when no privacy constraints are in place, which is $N^{-\frac{2\alpha}{2\alpha+1}}$. Furthermore, the minimax rate shows that, when $N$ observations are divided over $m$ machines somewhat equally, there is benefit in dividing over as few machines as possible and there is an additional benefit to having machines with a relatively large amount of data. The explanation for this is the same as described in the homogeneous case: it is easier to retain privacy within large local samples. When the samples are ``spread thinly'' across the servers, the cost of DP is larger. Between the pointwise risk and the global risk, the phenomenon of pointwise risk incurring relatively less cost when $\varepsilon_j$'s are decreased compared to the global risk is also still observed when $p < \infty$, whilst the cost of distributing is relatively higher.

In the regime of \eqref{eq:no_dominant_budget}, even though the privacy budgets vary between the servers, all the servers can be seen to provide a non-negligible contribution to the central estimator. Another regime which we highlight, is the case where some $j^* \in [m]$ the privacy budget satisfies
\begin{align}\label{eq:dominant_budget}
	\left( n_{j^*}^2 \varepsilon_{j^*}^2 \right) \wedge \left( n_{j^*}^{\frac{2\gamma + 4}{2\gamma+2}} \varepsilon_{j^*}^{\frac{2}{2\gamma+2}} \right) \wedge n_{j^*}^{\frac{2\gamma+2}{2\gamma+1}}  \ge \underset{[m] \setminus \{j^*\}}{\overset{}{\sum}} n_j^2 \varepsilon_j^2  ,
\end{align}
where we consider $\gamma = \alpha$ or $\gamma = \nu$ for the global and pointwise risk respectively. This regime is in a sense the juxtaposition of \eqref{eq:no_dominant_budget}. Where in \eqref{eq:no_dominant_budget}, no server has a substantially better privacy budget compared to its number of observations, in the case of \eqref{eq:dominant_budget}, there is (at least) one server with a substantially larger sample and/or a relatively better privacy budget than the other servers. The following result captures the minimax rate for the global and pointwise risks in such a regime.

\begin{corollary}\label{thm:heterogeneous_dominant_budget_both_risks} 
	Suppose that $(\bm{\epsilon},\bm{\delta})$ satisfies $ \delta_j \lesssim (\sqrt{n_j}\varepsilon_j^2/(n_{j^*}^{2/3} \varepsilon_{j^*}^{2/3}))^{1+\kappa}$ for some $\kappa>0$ and all $j =1,\dots,m$, \eqref{eq:dominant_budget} holds for $\gamma=\alpha$ and $\varepsilon_{j^*} > (n_{j^*})^{-1}$. Then, it holds that
	\begin{equation*}
		\inf_{\hat f \in \cM(\bm \varepsilon,\bm \delta)}\sup_{f \in \cB^{\alpha,R}_{p,q}} \E_{f} \|\hat f - f\|_2^2 \asymp M_{m,n} \left(\left( n_{j^*}^2 \varepsilon_{j^*}^2 \right)^{- \frac{2\alpha}{2\alpha+2}} + \left( n_{j^*} \right)^{- \frac{2\alpha}{2\alpha+1}}\right)
	\end{equation*}
	for some $M_{N} \geq 1$ of the order at most $\log(N) \cdot \log(1/\min_j \delta_j)$. If \eqref{eq:dominant_budget} holds for $\gamma = \nu$, it holds that
	\begin{equation*}
		\inf_{\hat f \in \cM(\bm \varepsilon,\bm\delta)} \sup_{f \in \mathcal{B}^{\alpha,R}_{p,q}} \E_{f} |\hat f(x_0) - f(x_0)|^2 \asymp M_{m,n} \left( \left( n_{j^*}^2 \varepsilon_{j^*}^2 \right)^{- \frac{2\nu}{2\nu+2}}  +  \left( n_{j^*} \right)^{- \frac{2\nu}{2\nu+1}}\right) 
	\end{equation*}
	for some $M_{N} \geq 1$ of the order at most $\log(N)$. 
\end{corollary}

In the regime described by the theorem, certain servers have a large sample and relatively large privacy budget, compared to the ``majority'' of the other servers in the sense of \eqref{eq:dominant_budget}. The minimax rate derived describes that in such settings these large sample/budget servers dictate the statistical accuracy of estimation. This is true both for the global, as well as the pointwise risk. In terms of optimal estimation procedures, the minimax rate can be achieved by only using estimators based on the data of the server(s) with relatively large samples and privacy budget, as the benefit of the servers with smaller samples and privacy budgets have an asymptotically negligible benefit.

\section{Optimal Distributed $(\bm \varepsilon,\bm \delta)$-DP Estimators}
\label{sec:upper-bounds}

In this section, we present two estimators that attain the optimal rates as described by the theorems of the previous section. One estimator specifically targets the global risk, the other is constructed specifically to perform well in terms of the pointwise risk. Whilst it perhaps natural to estimate $f(x_0)$ using the global risk DP-estimator evaluated at $x_0$, the specific pointwise  estimator we propose combines DP-estimates of $f(x_0)$ computed locally (i.e. estimators of $f(x_0)$ computed at each of the servers). This approach offers several benefits, such as an improved performance regardless of the value of $\delta_j \geq 0$. Both estimators are constructed using a wavelet basis. 

Wavelets are known to have many favourable properties when using them for function estimation in classical settings, see for example \cite{donoho1998minimax,hall1999minimax,cai1999adaptive}. Under DP constraints, wavelet constructions have other desirable properties: they allow for exact control of the estimator's \emph{sensitivity} to changes in the data. Loosely speaking, this allows us to control the ``influence'' each individual observation has on the outcome of the estimator, whilst retaining the information the full sample has to a large extent.

\subsection{Wavelets and Besov spaces}

In the context of nonparametric regression, we aim to construct an optimal estimator for an unknown function $f$ based on the distributed data. Here, we assume that $f$ belongs to the Besov space $\cB^\alpha_{p,q}$. Roughly stated, the Besov space $\cB^\alpha_{p,q}$ containts functions having $\alpha$ bounded derivatives in $L_p$-space, with $q$ giving a finer control of the degree of smoothness. We refer the reader to \cite{triebel1992theory} for a detailed description. 
 
Wavelet bases allow characterization of the Besov spaces, where $\alpha$, $p$ and $q$ are parameters that capture the decay rate of wavelet basis coefficients. Before presenting the two optimal estimators for global and pointwise risk in Sections \ref{ssec:global_estimation_risk_estimator} and \ref{ssec:pointwise_estimation_risk_estimator} respectively, we first briefly introduce wavelets and collect some properties used to define the Besov space. For a more detailed and elaborate introduction of wavelets in the context of Besov spaces, we refer to \cite{hardle2012wavelets, johnstone2019manuscript}.

In our work we consider the Cohen, Daubechies and Vial construction of compactly supported, orthonormal, $A$-regular wavelet basis of $L_2[0,1]$, see for instance \cite{cohen1993wavelets}. First for any $A \in\mathbb{N}$ one can follow Daubechies' construction of the father $\phi(\cdot)$ and mother $\psi(\cdot)$ wavelets with $A$ vanishing moments and bounded support on $[0,2A-1]$ and $[-A+1,A]$, respectively, for which we refer to \cite{daubechies1992ten}. The basis functions are then obtained as
\begin{align*}
	\big\{ \phi_{l_0+1,m},\psi_{lk}:\, m\in\{0,...,2^{l_0+1}-1\},\quad l \geq l_0 +1,\quad k\in\{0,...,2^{l}-1\} \big\},
\end{align*}
with $\psi_{lk}(x)=2^{l/2}\psi(2^lx-k)$, for $k\in [A-1,2^l-A]$, and $\phi_{l_0+1, k}(x)=2^{l_0+1}\phi(2^{l_0 + 1}x-m)$, for $m\in [0,2^{l_0 + 1}-2A]$, while for other values of $k$ and $m$, the functions are specially constructed, to form a basis with the required smoothness property. In a slight abuse of notation, we shall denote the father wavelet by $\psi_{l_0  k} = \phi_{l_0 +1, k}$ and represent any function $f\in L_2[0,1]$ in the form
\begin{align}\label{eq:wavelet_transform}
	f=\sum_{l=l_0}^{\infty}\sum_{k=0}^{2^{l}-1} f_{lk} \psi_{lk},
\end{align}
where the $f_{lk}=\langle f,\psi_{lk}\rangle$ are called the \emph{wavelet coefficients}. Note that in view of the orthonormality of the wavelet basis the {$L_2$-norm} of the function $f$ is equal to
\begin{align*}
	\|f\|_2^2=\sum_{l=l_0}^{\infty}\sum_{k=0}^{2^{l}-1}f_{lk}^2.
\end{align*}
Next we give definition of Besov spaces using wavelets. Let us define the norms 
\begin{equation*}
	\| f\|_{\cB^{\alpha}_{p,q}} :=  \begin{cases}
		\left( \underset{l=l_0}{\overset{\infty}{\sum}} \left(2^{l(\alpha+1/2-1/p)} \left\| (f_{lk})_{k=0}^{2^l-1} \right\|_p\right)^{q}\right)^{1/q} &\text{ for } 1 \leq q < \infty, \\
		\; \underset{l\geq l_0}{\sup} \; 2^{l(\alpha+1/2-1/p)} \left\| (f_{lk})_{k=0}^{2^l-1} \right\|_p   &\text{ for } q = \infty,
	\end{cases}
\end{equation*}
for $\alpha\in(0,A)$, $1\leq q \leq \infty$, $2 \leq p \leq \infty$. Then, the Besov space $\mathcal{B}^{\alpha}_{p,q}[0,1]$ and Besov ball $\mathcal{B}^{\alpha,R}_{p,q}[0,1]$ of radius $R>0$ can be defined as
\begin{equation}
	\label{Besov.ball}
	\mathcal{B}^{\alpha}_{p,q}[0,1]=\{f\in L_2[0,1]: \| f\|_{\cB^{\alpha}_{p,q}}<\infty \} \;\text{and}\; \mathcal{B}^{\alpha,R}_{p,q}[0,1]=\{f\in L_2[0,1]: \| f\|_{\cB^{\alpha}_{p,q}}\leq R \},
\end{equation}
respectively. The above definition of the Besov space and norm is equivalent to the classical one based on the weak derivatives of the function (see e.g. Chapter 4 in \cite{gine_mathematical_2016}).

For the construction of our estimators, we consider a $A$-smooth wavelet basis ($A>\alpha$) with a compactly supported mother wavelet $\psi$ such that wavelets $\psi_{lk}(x) = 2^{l/2}  \psi(2^lx -k)$ for $l \geq l_0$ and $k =0,\dots,2^l-1$ form an orthonormal basis for $\cB^\alpha_{p,q}[0,1]$.

We will briefly describe in broad terms the idea behind using the wavelet transform \eqref{eq:wavelet_transform} to construct the global and pointwise optimal estimators. Both estimators are based on wavelet approximations up until a limited resolution level. Besides the excellent approximation properties of wavelets in Besov spaces (see e.g. \cite{gine_mathematical_2016}), the first property ensures that a change in the data in terms of $X_i^{(j)}$ has a limited change in terms of the ``size'' the wavelet estimator. The second and third property yield a limited support that shrinks at higher ``resolution levels'' of the wavelet functions, which controls the number of wavelet coefficients affected by a change in $X_i^{(j)}$. Making sure that changes in individual datums have a limited effect on the shared transcript is crucial for assuring privacy. A further, more detailed description of how these properties interlink is given in the Sections \ref{ssec:global_estimation_risk_estimator} and \ref{ssec:pointwise_estimation_risk_estimator} below.

\subsection{Constructing an optimal global estimator}
\label{ssec:global_estimation_risk_estimator}

We now proceed to construct the estimator, utilizing the wavelet transform of \eqref{eq:wavelet_transform}, allowing the representation of a function $f$ in $L_2$ as a linear combination of wavelet basis functions. We first introduce some notation. For $\tau > 0, x \in \R$, let $[ x ]_{\tau}$ denote $x$ clipped at the threshold $\tau$: 
\begin{equation*}
	[ x ]_{\tau} := \begin{cases}
		\tau \; \text{ if } x > \tau, \\ 
		x \; \text{ if } -\tau \leq x \leq \tau \\ 
		-\tau \; \text{ otherwise.}
	\end{cases}
\end{equation*}
Given $L \in \N$ and $\tau > 0$, each machine $j=1,\dots,m$ computes the real numbers
\begin{equation}\label{eq:estimated_wavelet_coefficient}
	\hat f^{(j)}_{lk;\tau} = \frac{1}{n_j}\sum_{i=1}^{n_j} \left[Y_i^{(j)}\right]_{\tau} \psi_{lk}\left(X_i^{(j)}\right),
\end{equation}
for $l,k \in \N$ such that $l_0 \leq l \leq L$, $0 \leq k  \leq 2^l-1$. We will specify the exact choice of $\tau$ and $L$ later. These numbers, which we denote as the vector
\begin{equation*}
	\bm{\hat{f}}^{(j)}_{L,\tau} := \left\{\hat{f}^{(j)}_{lk;\tau}: k = 0, \ldots, 2^l-1, l = l_0, \ldots, L \right\},
\end{equation*}
will form the statistic underlying our transcript. To assure privacy, we aim to communicate a noisy version of this vector. Adding additional noise leads to an estimator that is necessarily worse, adding noise of a large enough magnitude yields a final transcript satisfies the privacy guarantee of Definition \ref{def:differential_privacy}. To control the magnitude of the noise that needs to be added, it is important to have a statistic that does not change too drastically when the underlying data is changed in one data point. We formalize this in terms of the \emph{sensitivity} of the statistic $\bm{\hat{f}}^{(j)}_{L,\tau}$.

The following lemma controls the \emph{$L_2$-sensitivity} of the statistic $\bm{\hat{f}}^{(j)}_{L,\tau}$, i.e. the difference in Euclidian distance when applied to two neighboring data sets.
\begin{lemma}\label{lemma:l2-sensitivity-wavelets}
	Let $Z^{(j)}$ and $\tilde{Z}^{(j)}$ any realizations of neighboring data sets. It holds that
	$$
	\left\|\bm{\hat{f}}^{(j)}_{L,\tau}(Z^{(j)}) - \bm{\hat{f}}^{(j)}_{L,\tau}(\tilde{Z}^{(j)})\right\|_2 \leq c_\psi \frac{\tau \sqrt{2^L}}{n_j}
	$$
	where $c_\psi$ is a constant depending only on the choice of wavelet basis.
\end{lemma}
We provide a proof for the lemma in Section \ref{sec:proof-of-l2-sensitivity-for-wavelts}. The limited {$L_2$-sensitivity} of the $\bm{\hat{f}}^{(j)}_{L,\tau}$ is a consequence of merging two elements in its construction. First of, clipping limits the change in \eqref{eq:estimated_wavelet_coefficient} when $Y_i^{(j)}$ is exchanged for another data point $\tilde{Y}_i^{(j)}$. In the vector as a whole, the coordinate wise change is limited by the compact support of the wavelet basis. The essential feature of the wavelet basis here is that, even though the basis elements increase exponentially as the resolution levels $l$ increases, their support shrinks proportionally, ensuring that each $X_i^{(j)}$ is in the support of only finitely many wavelets at each resolution level. That is, there are at most $c_A > 0$ number of basis functions $\psi_{lk}$ with overlapping support at each resolution level $l$, where $c_A > 0$ depends on $A > \alpha$. This means that changing one datum in $Z^{(j)}$, say $(Y_i^{(j)},X_i^{(j)})$ to $(\tilde{Y}_i^{(j)},\tilde{X}_i^{(j)})$, has only a limited impact at the level of the transcript $\bm{\hat{f}}^{(j)}_{L,\tau}(Z^{(j)})$.

The bounded $L_2$-sensitivity means that the statistic $\bm{\hat{f}}^{(j)}_{L,\tau}(Z^{(j)})$, combined with additive, appropriately scaled Gaussian noise satisfies $(\varepsilon_j,\delta_j)$-differentially privacy. This result for mappings with bounded $L_2$-sensitivity is well known and a proof can be found in e.g. Appendix A of \cite{dwork2014algorithmic}. To be precise, the $j$-th server outputs
$
\tilde T^{(j)}_{lk;\tau} = \hat f^{(j)}_{lk;\tau} + W^{(j)}_{lk}
$
for $\,k=0,\dots,2^l-1,\,l=l_0,\dots,L$, where the coordinates of $\bm W^{(j)}:= (W^{(j)}_{lk} : k = 0, \ldots, 2^l-1, l = l_0, \ldots, L )$ are i.i.d. mean zero gaussian with variance  $\frac{4 \tau^2 2^{L}c_\psi^2\log(2/\delta_j)}{n_j^2\varepsilon_j^2}$. The constant $c_\psi := 2\sqrt{2}\sqrt{c_A} \| \psi \|_\infty$ matches the constant in Lemma \ref{lemma:l2-sensitivity-wavelets}. The addition of the gaussian noise ensures that the transcript
$${T}^{(j)}_{L,\tau} := \left\{ T^{(j)}_{lk;\tau} \, : \, k=0,\dots,2^l-1,\,l=l_0,\dots,L \right\} = \bm{\hat{f}}^{(j)}_{L,\tau}(Z^{(j)}) + \bm W^{(j)}$$
is $(\varepsilon_j,\delta_j)$-differentially private.

The final estimator of $f$ is then obtained via a post-processing step in which each of the transcripts is reweighted, taking the heterogeneity between the servers into account. The choice of weight depends crucially on the local number of observations $n_j$ and the local privacy constraint $\varepsilon_j$. Given the transcripts $T = ({T}^{(1)}_{L,\tau},\dots,{T}^{(m)}_{L,\tau})$, the final estimator takes the form of
$$
\hat f_{L,\tau}(x) = \sum_{l=l_0}^L \sum_{k=0}^{2^l-1} \left( \sum_{j=1}^m u_j T^{(j)}_{lk;\tau} \right) \psi_{k,l}(x),
$$
where the weights are given by 
\begin{equation}\label{eq:estimator_weights_global}
	u_j  = \frac{v_j}{\sum_jv_j} \; \text{ with } \; v_j = \left(n_j^2\varepsilon_j^2 \right) \wedge  \left( n_j 2^L \right).
\end{equation}
The following theorem captures the global risk attained by the estimator $\hat f_{L,\tau}$ resulting from the distributed $(\bm{\epsilon},\bm{\delta})$-DP procedure outlined above, with optimal selection of $L$ and a sufficiently large choice of $\tau$. For the latter, a choice of {$ C_{\alpha,R} + \sqrt{(2\alpha+1)L}$} is adequate, where $C_{\alpha,R} > 0$ is a constant, as specified by Lemma \ref{lem:uniform_bound_on_f} or a larger constant.

The variance of the Gaussian noise vectors $\bm W^{(j)}$, which yield the privacy guarantee, increases with $L$. Consequently, the optimal choice of $L$ is not just governed by the classical bias variance trade-off, but also by the trade-off in the additional noise required to guarantee privacy. 

The optimal choice of $L$ is taken as follows. Let $D > 0$ be the number solving the equation
\begin{equation}\label{eq:dimension_determining_equation_global}
	D^{2\alpha+2} = \sum_{j=1}^m \left( n_j^2 \varepsilon_j^2 \right) \wedge  \left( n_j D \right).
\end{equation}
Setting $L = (l_0 + 1) \vee \lceil \log_2(D) \rceil$ yields the optimal performance as described by the theorem below, in terms of a bias-variance-sensitivity trade-off. Furthermore, this performance turns out to be the theoretically best possible performance in a minimax sense, as established by the lower bound of Theorem \ref{thm:minimax-lower-bound} in Section \ref{sec:lower-bounds}.

\begin{theorem}\label{thm:minimax-upper-bound}
	Set {$\tau = C_{\alpha,R} + \sqrt{(2\alpha+1)L}$} and take $L = (l_0 + 1) \vee \lceil \log_2(D) \rceil$, where $D > 0$ is the solution to \eqref{eq:dimension_determining_equation_global}. 

	Then, the $L_2$-risk of the distributed $(\bm \varepsilon,\bm \delta)$-DP protocol $\hat f_{L,\tau}$ satisfies
	$$
	\sup_{f \in \cB^{\alpha,R}_{p,q}} \E_{f} \left\|\hat f_{L,\tau} - f \right\|_2^2 \leq C_\psi \log(N) 2^{-2L\alpha}\log(2/\delta'),
	$$
	where $\delta' = \min_{i\in[m]}\delta_i$ and $C_\psi$ denotes a constant depending on $\psi$.
\end{theorem}
We briefly comment on the derived result. We first note that the choice of wavelet basis (in particular the father wavelet $\psi$) influences the constants in the theorem, but not the convergence rate. The rate attained by the choice of $L$ as directed by \eqref{eq:dimension_determining_equation_global} yields optimal rate as given in Corollary \ref{thm:equal_budget_global_risk} in case of homogeneous servers, the optimal rate of Corollary \ref{thm:heterogeneous_no_dominant_budget_both_risks} in case the privacy budgets satisfy \eqref{eq:no_dominant_budget} or the optimal rates of Corollary \ref{thm:heterogeneous_dominant_budget_both_risks} in case the budget satisfies \eqref{eq:dominant_budget}.

\subsection{Constructing an optimal estimator of $f$ at a point}
\label{ssec:pointwise_estimation_risk_estimator}

We now turn to the task of estimating the unknown function $f \in \cB^{\alpha,R}_{p,q}$ at a given point $x_0 \in(0,1)$. That is to say, we will construct an estimator $\hat f$ such that $\E_{f} (\hat f(x_0) - f(x_0))^2$ achieves the optimal rates as predicated by Corollaries \ref{thm:pointwise_risk_homogeneous_setting}, \ref{thm:heterogeneous_no_dominant_budget_both_risks} and \ref{thm:heterogeneous_dominant_budget_both_risks}.

A natural estimator of $f(x_0)$ is to use the global plug-in estimator of the previous section, $\hat f_{L,\tau}(x_0)$, with $\hat f_{L,\tau}$ as constructed in the previous section. However, for estimation at a point as goal of inference, we instead opt for estimating $f(x_0)$ locally. Similarly to the pre-processing step in Section \ref{ssec:global_estimation_risk_estimator}, the preliminary local estimator is to be perturbed with noise to ensure it satisfies the DP constraint of Definition \ref{def:differential_privacy}. Adding Laplacian noise turns out to suffice to ensure $(\varepsilon_j,0)$-DP, which is a stronger guarantee than $(\varepsilon_j, \delta_j)$-DP. Another advantage to this approach compared to the plug-in estimator $\hat f_{L,\tau}(x_0)$ is that the procedure derived below has a $\log(N)$-factor improved rate compared to the plug-in estimator. 

As a first step in constructing the estimator of $f(x_0)$, we consider for $L \in \N$ and $\tau > 0$ the first wavelet coefficients $\bm{\hat{f}}^{(j)}_{L,\tau}$ as computed in \eqref{eq:estimated_wavelet_coefficient}. On the $j$-th server, we construct the estimator corresponding to $\bm{\hat{f}}^{(j)}_{L,\tau}$, which we then evaluate in the point $x_0$,
\begin{equation}\label{eq:pointwise_preliminary_est}
	\hat f_{L,\tau}^{(j)}(x_0) \equiv \hat f_{L,\tau}^{(j)}(x_0 | Z^{(j)}) := \sum_{l=l_0}^L \sum_{k=0}^{2^l -1} \hat{f}^{(j)}_{lk;\tau} \psi_{lk}(x_0).
\end{equation}
In order to create a $(\varepsilon_j,0)$-DP transcript, we add Laplacian noise to $\hat f_{L,\tau}^{(j)}(x_0)$ directly. Laplacian noise performs well for statistics with small $L_1$-\emph{sensitivity}, i.e. the change in $L_1$-norm when one datum is changed underlying the statistic. The $L_1$-sensitivity scales poorly in the dimension of the statistic compared to the $L_2$-sensitivity dictating the noise of the Gaussian mechanism of Section \ref{ssec:global_estimation_risk_estimator}. Essentially, the estimator proposed in \eqref{eq:pointwise_preliminary_est} has large sensitivity only for observations that are ``close'' to $x_0$. To see this, note that $\psi_{lk} \lesssim 2^{l/2}$ and that the latter estimator can be written as 
\begin{equation*}
\hat f_{L,\tau}^{(j)}(x_0 | Z^{(j)}) = \sum_{l=l_0}^L \sum_{k \in K_l(x_0)}^{} \hat{f}^{(j)}_{lk;\tau} \psi_{lk}(x_0),
\end{equation*}
where $K_l(x_0) := \{ k : \psi_{lk}(x_0) \neq 0 \}$. Using this fact, the lemma below gives an exact bound on the $L_1$-{sensitivity} of the functional $\hat f_{L,\tau}^{(j)}(x_0)$.

\begin{lemma}\label{lemma:l1-sensitivity-wavelets-at-a-point} 
	Let $Z^{(j)}$ and $\tilde{Z}^{(j)}$ any realizations of neighboring data sets. It holds that
	\begin{equation}\label{eq:l1_sensitivity}
		\left\|\hat f_{L,\tau}^{(j)}(x_0 | Z^{(j)})  - \hat f_{L,\tau}^{(j)}(x_0 | \tilde{Z}^{(j)}) \right\|_1 \leq c'_\psi \frac{\tau 2^L}{n_j},
	\end{equation}
	where $c''_\psi := 2c_A \|\psi\|_\infty^2$ is a constant depending only on the choice of wavelet basis.
\end{lemma}  
This bound on the $L_1$-sensitivity yields that as a privacy mechanism it suffices to add Laplace noise with variance $\frac{c'_\psi \tau 2^L}{n_j\varepsilon_j}$ to the functional $\hat f_{L,\tau}^{(j)}(x_0)$. That is, the transcript $T^{(j)}_{lk,\tau}$ for $j \in[m]$ given by
$$
T^{(j)}_{L,\tau} = \hat f_{L,\tau}^{(j)}(x_0) + W^{(j)}, \quad \mathrm{where} \quad W^{(j)}\overset{i.i.d}{\sim} \mathrm{Lap}\left(0,\frac{c'_\psi \tau 2^L}{n_j \varepsilon_j}\right)  
$$ 
is $(\varepsilon_j,0)$-DP (see e.g. \cite{dwork2014algorithmic}). With each $j$-th server transmitting $T^{(j)}_{L,\tau}$, the estimator computed in the central server is given by $\hat f(x_0) = \sum_{j=1}^m u_j  T^{(j)}_{L,\tau}$. 

It remains to determine the optimal choice of $L$. Similarly as in in the case of the estimator of global risk as presented in Section \ref{ssec:global_estimation_risk_estimator}, there is a trade-off between bias, variance and sensitivity, where the $L_1$-sensitivity can be seen to have a dependence on $L$. The explanation for this, is that even though the functional $\hat f_{L,\tau}^{(j)}(x_0)$ is unidimensional, the wavelet resolution level $L$ still determines how a change in an individual datum can potentially change the value of a local estimator defined in \eqref{eq:pointwise_preliminary_est}.

Here, the choice of $L$ is governed by $L = (l_0 + 1) \vee \lceil \log_2(D) \rceil$, where $D > 0$ be the number solving the equation
\begin{equation}\label{eq:dimension_determining_equation_pointwise}
	D^{2\nu+2} = \sum_{j=1}^m \left( n_j^2 \varepsilon_j^2 \right) \wedge \left( n_j D \right).
\end{equation}
The following theorem describes the performance of the pointwise estimator on the basis of the $(\bm \varepsilon,\bm \delta)$-DP transcript $T = ({T}^{(1)}_{L,\tau},\dots,{T}^{(m)}_{L,\tau})$ for a sufficiently large choice of $\tau$, such as $\tau = C_{\nu,R} +  \sqrt{2(2\nu+1)L}$, with $C_{\nu,R} > 0$ as given by Lemma \ref{lem:uniform_bound_on_f}.

\begin{theorem}\label{thm:upper-bound-pointwise-risk}
	Set $\tau = C_{\nu,R} +  \sqrt{2(2\nu+1)L}$ and take $L = (l_0 + 1) \vee \lceil \log_2(D) \rceil$, where $D$ is governed by \eqref{eq:dimension_determining_equation_pointwise}. 

	Then, the pointwise $\ell_2$-risk of the distributed $(\bm \varepsilon,\bm \delta)$-DP protocol $\hat f_{L,\tau}$ satisfies
	\[
	\sup_{f \in \cB^{\alpha,R}_{p,q}} \E_{f} (\hat f(x_0) - f(x_0))^2 \leq  C_\psi \log (N)  2^{-2L\nu}.
	\]
	for a constant $C_\psi > 0$ depending only on the choice of wavelet basis.
\end{theorem}

The rate attained by the choice of $L$ as directed by \eqref{eq:dimension_determining_equation_pointwise} does not just yield the best possible bias-variance-sensitivity trade-off for the estimator class under consideration, but it turns out to be minimax optimal (up to a log factor) as established in the lower bound of Theorem \ref{thm:lower_bound-pointwise-risk}. It consequently yields the optimal rate as given in Corollary \ref{thm:pointwise_risk_homogeneous_setting} in case of homogeneous servers, the optimal rate of Corollary \ref{thm:heterogeneous_no_dominant_budget_both_risks} in case the privacy budgets satisfy \eqref{eq:no_dominant_budget} or the optimal rates of Corollary \ref{thm:heterogeneous_dominant_budget_both_risks} in case the budget satisfies \eqref{eq:dominant_budget}. As is the case with Theorem \ref{thm:minimax-upper-bound}, the choice of wavelet basis influences the constants in the theorem, but not the convergence rate. 

\section{Minimax Lower Bounds}
\label{sec:lower-bounds}

Theorems \ref{thm:minimax-upper-bound} and \ref{thm:upper-bound-pointwise-risk} provide the rates of convergence for the proposed estimators of $f$ and $f(x_0)$, respectively. In this section we shall show that these rates of convergence are indeed optimal among all estimators by establishing two matching minimax lower bounds,  up to logarithmic factors, for global and pointwise estimation. These results affirm the optimality of the estimators presented in Section \ref{sec:upper-bounds}. 

The lower bounds are presented in Theorems \ref{thm:minimax-lower-bound} and \ref{thm:lower_bound-pointwise-risk} for the global and pointwise risks, respectively. The derivation of each lower bound relies on entirely distinct techniques, elaborated in Section \ref{ssec:lowerbound} and Section \ref{ssec:pointwise_lowerbound}.

\subsection{Minimax lower bounds under heterogeneous distributed DP constraints}
\label{ssec:lowerbound}

The following theorem states a lower bound on the minimax risk for global estimation. 
\begin{theorem}\label{thm:minimax-lower-bound}
	Let $D>0$ be the solution to \eqref{eq:dimension_determining_equation_global} and assume that $\delta_j < (n_j^{1/2}\varepsilon_j^2 (D \vee 1)^{-1})$ for some $\kappa>0$ and all $j\in[m]$. Then, we have the following lower bound on the minimax risk:
	\begin{equation}\label{eq:global_risk_lower_bound}
		\inf_{\hat f \in \cM(\bm \varepsilon,\bm \delta)} \, \sup_{f \in \mathcal{B}^{\alpha,R}_{p,q}} \mathbb{E}_{f} \|\hat f - f\|_2^2 \gtrsim D^{-2\alpha} \wedge 1.
	\end{equation}
\end{theorem}

Before outlining the proof, we briefly note that this lower bound matches that of the upper bound of Theorem \ref{thm:minimax-upper-bound} (up to a log-factor), for the choice $L = (l_0 + 1) \vee \lceil \log_2(D) \rceil$ in the estimator under consideration in Section \ref{ssec:global_estimation_risk_estimator}. The theorem affirms that, up to a log factor, the proposed estimator attains the best rate among all privacy constrained estimators.

Next, we discuss the most important steps in the proof here, whilst leaving the technical details to the appendix. To lower bound the global risk, we first restrict to a finite-dimensional sub-model of the Besov space $\mathcal{B}^{\alpha,R}_{p,q}$. To align notation with the previous section, we shall use the wavelet basis from before to do so. Given $L \in \N$, we consider the finite-dimensional subspace 
\[
\mathcal{B}^{\alpha,R,L}_{p,q} := \left\{f \in \mathcal{B}^{\alpha,R}_{p,q}:  f=\sum_{k=0}^{2^{L}-1}f_{Lk}\psi_{Lk}, \; f_{Lk} \in [-2^{-L(\alpha+1/2)}R,2^{-L(\alpha+1/2)}R] \right\}.
\]
Let $\bm \psi(X)$ denote the $2^L$ dimensional vector $\{\psi_{Lk}(X)\}_{k=1}^{2^L}$ and define
\begin{equation}\label{eq:influence_function}
	\bm S_f\left(Z_i^{(j)}\right) := \sigma^{-1}\left(Y^{(j)}_i - \sum_{k=0}^{2^L-1} f_{Lk}\psi_{Lk}\left(X^{(j)}_i\right) \right) \bm\psi \left(X^{(j)}_i \right).
\end{equation}
The random vector $\bm S_f(Z_i^{(j)})$ can be seen as an ``score function'' of the $i$-th observation on the $j-$th server, within the finite dimensional sub-model. Similarly, consider the ``score function'' for local data $Z^{(j)}$ on the $j$th server; $ \bm S_f(Z^{(j)}) := \sum_{i=1}^{n_j} \bm S_f(Z_i^{(j)})$. Furthermore, let $\bm C_f(T^{(j)})$ denote the $2^L$ dimensional matrix 
\begin{equation}\label{eq:conditional_covariance_of_influence_function}
	\mathbb{E} \; \mathbb{E}\left[\bm S_f\left(Z^{(j)}\right)\mid T^{(j)}\right]\mathbb{E}\left[\bm S_f\left(Z^{(j)} \right)\mid T^{(j)}\right]^T. 
\end{equation}
We shall write $\bm C_f(Z^{(j)})$ for the unconditional version covariance matrix of $\bm S_f(Z^{(j)})$and let $\bm C_{f,i}^{{(j)}} = \mathbb{E}\left[\bm S_f(Z_i^{(j)})\bm S_f(Z_i^{(j)})^T\right]  $, such that $\bm C_f = \sum_{i=1}^{n_j} \bm{C}_{f,i}^{{(j)}}$.

Using the Van-Trees inequality (with a prior as specified later on in the section), we obtain an expression in terms of the sum-of-traces of the matrices in display \eqref{eq:conditional_covariance_of_influence_function}, i.e. the covariance of the score function $\bm S_f(Z^{(j)})$, conditionally on the released transcripts. 

As the conditional expectation contracts the $L_2$-norm, we have the ``data processing'' bound $\bm C_f(T^{(j)}) \leq \bm C_f(Z^{(j)})$, which in turn implies that 
\begin{equation}\label{eq:regular_trace_bound}
	\text{Tr}(\bm C_f(T^{(j)})) \leq \text{Tr}(\bm C_f(Z^{(j)})).
\end{equation}
The right-hand side can be bounded by $2^L n_j$ by direct calculation, which we defer to Section \ref{sec:lower-bound-proofs}. These standard bounds do not take the privacy constraints into account and would lead to the unconstrained minimax rate.

To capture the loss of information stemming from the DP constraint of Definition \ref{def:differential_privacy}, a more sophisticated data processing argument is required. This bring us to one of the key technical innovations of the paper, which comes in the form of a data-processing inequality (Lemma \ref{lem:trace-processing} below) for the conditional covariance given a $(\varepsilon_j,\delta_j)$-differentially private transcripts of linear functionals of the data such as the score $\bm S_f(Z_i^{(j)})$. The lemma can be seen as a geometric version of the ``score attack'' lower bound of \cite{cai2023score}. Combining this data processing step with the linearity of the trace accommodates for the heterogeneity between the servers. 

\begin{lemma}\label{lem:trace-processing} Let $\delta_j \log(1/\delta_j)< n_j^{1/2}\varepsilon_j^2 (D \vee 1)^{-1}$ for $j=1,\dots,m$. There exists a universal constant $C>0$ such that
	\begin{align*}
		\mathbb{E} \left[\mathrm{Tr}(\bm C_f(T^{(j)}))\right] \leq  C n_j\varepsilon_j \sqrt{\mathbb{E} \left[\mathrm{Tr}(\bm C_f(T^{(j)}))\right]} &\sqrt{\lambda_{\max}(\bm C_{f,i})} \\ &+ C \delta_j\left( 2^Ln_j^{1/2} \log(1/\delta_j) + n_j \right).
	\end{align*}
\end{lemma} 
In Section \ref{sec:lower-bound-proofs} of the appendix, we show that the largest eigenvalue of $\bm C_{f,i}$; $\lambda_{\max}(\bm C_{f,i})$, is bounded, from which it follows from the  $\mathbb{E} \left[\mathrm{Tr}(\bm C_f(T^{(j)}))\right] \lesssim n_j^2\varepsilon_j^2$ uniformly for $f \in \mathcal{B}^{\alpha,R,L}_{p,q}$ whenever $\delta_j$ is of smaller than $n_j^{1/2}\varepsilon_j^2 D^{-1}$.

With the two bounds on the trace of $\bm C_f(T^{(j)})$ in hand, we now lower bound global estimation risk using the Van-Trees inequality. The Van-Trees inequality provides an expression in terms of the trace of a certain covariance matrix, which is the conditional covariance of a linear functional of the data. Combined with the data processing inequalities, the linearity of the trace accommodates for the heterogeneity between the servers.

In order to apply the Van-Trees inequality, we first define a prior such that the worst-case global risk is lower bounded by the corresponding Bayes risk. To that extent, we define a prior $\Pi$ that is supported on $\mathcal{B}^{\alpha,R,L}_{p,q}$. Given the resolution level $L \in \N$, we draw $f_{Lk}$ independently from the probability distribution $\Pi_{Lk}$, defined through an appropriately rescaled version of the density $t \mapsto  \cos^2(\pi t/2) \mathbbm{1}_{|t| \leq 1)}$ such that has its support equal to $[-2^{-L(\alpha+1/2)}R,2^{-L(\alpha+1/2)}R]$ for $k=0\dots,2^L-1$ and set $f_{lk} =0$ otherwise. For this choice of prior, the Van-Trees inequality of \cite{gill1995applications} yields the following lemma, for which we defer the details of the proof to Section \ref{sec:lower-bound-proofs} in the appendix. 

\begin{lemma}\label{lemma:van-trees-nonparametric}
	It holds that $\sup_{f \in \mathcal{B}^{\alpha,R}_{p,q}} \mathbb{E}_{f} \|\hat f - f\|_2^2$ is lower bounded by the Bayes risk $\int \mathbb{E}_{f} \|\hat f - f\|_2^2 d\Pi(f)$, which is further lower bounded as follows
	\[
	\int \mathbb{E}_{f} \|\hat f - f\|_2^2 d\Pi(f) \geq \frac{2^{2L}}{\sup_{f \in \mathcal{B}^{\alpha,R,L}_{p,q}}\sum_{j=1}^m \mathrm{Tr}(\bm C_f(T^{(j)})) + \pi^2 2^{L(2\alpha+2)}}.
	\]
\end{lemma}

Combining the upper bound on the trace of $\bm C_f(T^{(j)})$ of \eqref{eq:regular_trace_bound} and Lemma \ref{lem:trace-processing}, we have, by Lemma \ref{lemma:van-trees-nonparametric}, that
\[
\sup_{f \in \mathcal{B}^{\alpha,R}_{p,q}} \mathbb{E}_{f} \|\hat f - f\|_2^2 \gtrsim \frac{2^{2L}}{\sum_{j=1}^m n_j^2\varepsilon_j^2 \wedge n_j 2^L + \pi^2 2^{L(2\alpha+2)}}.
\]
We obtain the desired lower bound by choosing an $L$ that maximizes the lower bound. Setting $L = (l_0 + 1) \vee \lceil \log_2(D) \rceil$ can be seen to do so by the relationship \eqref{eq:dimension_determining_equation_global}, which proves Theorem \ref{thm:minimax-lower-bound}.

\subsection{Lower bound for the pointwise risk}
\label{ssec:pointwise_lowerbound}

In this section, we derive the minimax lower bound for the pointwise risk. We first present the lower bound as the main result of the section in the form of Theorem \ref{thm:lower_bound-pointwise-risk}, after which we discuss its proof. The theorem tells us that the pointwise risk estimator proposed in Section \ref{ssec:pointwise_estimation_risk_estimator} performs optimally in terms of achieving the minimax privacy constrained rate up to a logarithmic factor.

\begin{theorem}\label{thm:lower_bound-pointwise-risk}
	Let $D > 0$ be the number solving the equation
	\begin{equation}\label{eq:dimension_determining_equation}
		D^{2\nu+2} = \sum_{j=1}^m \left( n_j^2 \varepsilon_j^2 \right) \wedge \left(n_j D \right).
	\end{equation}
	Assume furthermore that $\sum_{j} n_j \delta_j \to 0$. Then, for any $x_0 \in (0,1)$, the minimax pointwise risk is lower bounded as follows:
	\[
	\sup_{f \in \mathcal{B}^{\alpha,R}_{p,q}} \mathbb{E}_{f} \left(\hat f(x_0) - f(x_0)\right)^2 \gtrsim D^{-2\nu} \wedge 1.
	\]
	Whenever $\left( \sum_{j=1}^m n_j^2 \varepsilon_j^2 \right)^{\frac{1}{2\nu+2}} \geq \max_j n_j \epsilon_j^2$, the right hand side is further bounded from below by $\left( \sum_{j=1}^m n_j^2 \varepsilon_j^2 \right)^{-\frac{2\nu}{2\nu+2}} \wedge 1$.
\end{theorem}

The proof of the theorem is based around the Le Cam two point method, which is a common approach to lower bounding the pointwise risk, see for example \cite{Yu1997}. However, to capture the effect of the transcripts satisfying the DP constraint of Definition \ref{def:differential_privacy}, we introduce a coupling argument in conjunction.

We briefly sketch the two point method and coupling argument here, leaving the technical details to the appendix. Take any function $f \in \cB^\alpha_{p,q}$ such that $\|f\|_{\cB^\alpha_{p,q}} = R' < R$ and a compactly supported function $g \in \cB^\alpha_{p,q}$ such that $\| g\|_{\cB^{\alpha}_{p,q}} \leq R - R'$ and $g(0) >0$. Define a third function
\begin{equation*}
	\tilde{f}(t) := \gamma_D^{-1} g(\beta_D(t-x_0)) + f(t),
\end{equation*}
where $\gamma_D := c_0^{-1} D^\nu$ and $\beta_D = \gamma_D^{1/\nu}$, where we recall that $\nu = \alpha - \frac{1}{p}$. By e.g. Lemma 1 from \cite{cai2003rates}, $\|f\|_{\cB^\alpha_{p,q}} \leq R$.

Let $(Y_i^{(j)},X_i^{(j)}) \sim P_{f}$ and $(\tilde{Y}_i^{(j)},\tilde{X}_i^{(j)}) \sim P_{\tilde{f}}$ for individual observations generated according to \eqref{eq:dynamics_generating_data} with either $f$ or $\tilde{f}$ the true underlying regression function respectively. We construct a coupling between $P_{f}$ and $P_{\tilde{f}}$ such that $(Y_i^{(j)},X_i^{(j)})$ and $(\tilde{Y}_i^{(j)},\tilde{X}_i^{(j)})$ are equal with probability proportional to $\sigma^{-1} \| \tilde{f} - f \|_1$, which forms the content of the following lemma.

\begin{lemma}\label{lem:coupling_lemma}
	There exists a joint distribution $P_{f,\tilde{f}}$ of $\left((Y_i^{(j)},X_i^{(j)}),(\tilde{Y}_i^{(j)},\tilde{X}_i^{(j)})\right)$ such that
	\begin{equation}\label{eq:coupling_probability_bound}
		\rho := P_{f,\tilde{f}}\left( \left(Y_i^{(j)},X_i^{(j)}\right) \neq \left(\tilde{Y}_i^{(j)},\tilde{X}_i^{(j)}\right) \right) \leq \frac{c}{\sigma} \| \tilde{f} - f \|_1,
	\end{equation}
	for a universal constant $c>0$.
\end{lemma}
We prove the above lemma in Section \ref{sec:pointwise_lb_proofs}. Loosely speaking, the quantity $\rho$ captures the difficulty of distinguishing individual observations from $P_f$ of those generated from $P_{\tilde{f}}$.

Consider now transcripts $T = (T^{(1)},\dots,T^{(m)})$ each satisfying the DP constraint of Definition \ref{def:differential_privacy} with a privacy budget $(\bm \varepsilon, \bm \delta)$, and let $\P_f$ denote the joint law of transcripts and the $N = \sum_{j=1}^m n_j$ observations generated from $P_f$. Let $\P_{f}^T$ denote the push-forward measure of the transcript, i.e. its marginal distribution given that the data is generated by $P_f$. Similarly, let $\P_{\tilde{f}}$ denote the joint law of $T$ with the data generated from $P_{\tilde{f}}$ and let $\P_{\tilde{f}}^T$ denote the corresponding marginal distribution of $T$. With the coupling of Lemma \ref{lem:coupling_lemma} in hand, we derive the following lemma.

\begin{lemma}\label{lemma:total-variation-two-point_maintext}
	For any subset $S \subseteq [m]$,
	\begin{equation}\label{eq:total_variation_split}
		\left\| \P_{f}^T - \P_{\tilde{f}}^T  \right\|_{\mathrm{TV}} \leq \sqrt 2\sqrt{\sum_{j \in S} \bar{\varepsilon}_j \left( e^{\bar{\varepsilon}_j} - 1 \right) + \sum_{j \in S^c}n_j D_{\mathrm{KL}}(P_f; P_{\tilde{f}})} + 4 \underset{j \in S}{\overset{}{\sum}} e^{\bar{\varepsilon}_j} n_j\delta_j \rho,
	\end{equation}
	where $\bar{\varepsilon}_j = 6n_j\varepsilon_j \rho$, $\rho$ as defined in \eqref{eq:coupling_probability_bound}.
\end{lemma}

We defer a proof of the lemma to Section \ref{sec:pointwise_lb_proofs} of the appendix. The lemma allows analysis of the contributions of the separate the servers, accounting for the heterogeneity in the privacy budgets $(\epsilon_j,\delta_j)$ and the differing number of observations. Roughly speaking, for servers with relatively large privacy budgets, their contribution to the estimator is to be captured by $n_j D_{\mathrm{KL}}(P_f; P_{\tilde{f}}^T)$, which does not involve the privacy budget all together. Servers for which the privacy budget is more stringent, contribute with the (potentially) smaller quantity $\bar{\varepsilon}_j$, where $\rho$ corresponds to the probability in \eqref{eq:coupling_probability_bound}, established in the coupling relationship of Lemma \ref{lem:coupling_lemma}. 

The optimal division into these stringent and non-stringent privacy budgets is made by taking 
\[
S = \left\{j \in [m]: \varepsilon_j \leq \sqrt {{D}/{n_j}} \right\},
\]
in the sense that this choice of $S$ minimizes the right-hand side of \eqref{eq:total_variation_split}. With this choice of $S$, the bound on $\rho$ established in Lemma \ref{lem:coupling_lemma} and the fact that by the construction of $\tilde{f}$ we have $\|\tilde{f} - f\|_1 = \gamma_D^{-1}\beta_D^{-1} \|g\|_1$ with $\gamma_D := c_0^{-1} D^\nu$ and $\beta_D^{} = \gamma_D^{ 1/\nu}$, we obtain that
\begin{align*}
	\sum_{j \in S} \bar{\varepsilon}_j \left( e^{\bar{\varepsilon}_j} - 1 \right)\lesssim \sum_{j \in S} n_j^2 \varepsilon_j^2 \rho^2 \leq  \gamma_D^{-2}\beta_D^{-2} \| g \|_1^2 \sum_{j \in S} n_j^2 \varepsilon_j^2 = c_0^{2+2/\nu} D^{-2-2\nu} \| g \|_1^2 \sum_{j \in S} n^2_j \varepsilon_j^2.
\end{align*}
The bound $D_{\mathrm{KL}}(P_{\tilde{f}};P_{f}) \lesssim \sigma^{-2}\|\tilde{f}-f\|_2^2$, which is obtained through standard calculations, combined with the fact that $\|\tilde{f} - f\|_2^2 \lesssim \gamma_{D}^{-2} \beta_{D}^{-1} \| g \|_2^2$ by construction,
\begin{equation*}
	\sum_{j \in S^c} n_j D_{\mathrm{KL}}(P_f; P_{\tilde{f}}) \lesssim c_0^{1+2/\nu} D^{- 2 - 2\nu} \| g \|_2^2 \sum_{j \in S^c} n_j D.
\end{equation*}
In Section \ref{sec:pointwise_lb_proofs} of the appendix, it is shown that if $\sum_{j \in S} n_j \delta_j = o(1)$ as is assumed in the theorem, it holds that $\underset{j \in S}{\overset{}{\sum}} e^{\bar{\varepsilon}_j} n_j\delta_j \rho = o(1)$ for the particular choice of $S$ as well. Per the choice of the set $S$ and $D$ satisfying \eqref{eq:dimension_determining_equation_pointwise}, we obtain that for $c_0 < 1$,
\begin{equation*}\label{eq:total_variation_split2}
	\left\| \P_{f}^T - \P_{\tilde{f}}^T  \right\|_{\mathrm{TV}} \leq  C c_0^{1+2/\nu} + o(1).
\end{equation*}
By taking the constant $c_0$ sufficiently small, the conclusion of Theorem \ref{thm:lower_bound-pointwise-risk} follows by Le Cam's two point method (see e.g. Lemma 1 in \cite{Yu1997}), which then yields that
\[
\sup_{f \in \mathcal{B}^{\alpha,R}_{p,q}} \mathbb{E}_{f} \left(\hat f(x_0) - f(x_0)\right)^2 \gtrsim \left(\tilde{f}(x_0) - f_0(x_0)\right)^2 \gtrsim \gamma_n^{-2}g^2(0) \gtrsim D^{-2\nu},
\]
proving the theorem.

\section{Discussion}\label{sec:discussion}
 
The findings in the present paper  highlight the trade-off between statistical accuracy and privacy preservation within the context of federated nonparametric regression. The results under the heterogeneous setting quantify the degree to which the individual DP constraints, as well as the degree to which observations are distributed among the different servers, impact the statistical performance as measured by the minimax risk. Furthermore, we find that the influence of the privacy constraints on the optimal performance depends on the inferential task at hand, with global estimation of an unknown function being subject to a different performance impact than estimation of a function at a point. For each of these inferential tasks, we provide an estimation procedure that attains the optimal statistical performance up to a logarithmic factor.

One promising direction for future research is the exploration of adaptive estimation in the federated learning framework. While our paper characterizes the statistical performance for nonparametric regression in a heterogeneous setting, the estimation procedures in this paper assume knowledge about the regularity of the underlying function. However, in many real-world applications, the regularity is unknown and estimators that can adapt to the true underlying regularity are required to attain the best possible performance. while such adaptive techniques exist for problems without privacy constraints (see e.g. \cite{donoho1998minimax,cai2003rates}), the theoretical (im)possibilities of adaptation under privacy constraints are relatively understudied. Such adaptive techniques might serve dual purposes: they could potentially refine the statistical accuracy while also optimizing the DP constraints for individual servers, especially when the constraints are dynamic or based on real-time needs. We leave this for future work.

Another promising avenue for future research is nonparametric hypothesis testing. It is well known that testing in nonparametric settings is subject to different phenomena than estimation, see for example \cite{ingster_nonparametric_2003}. However, under privacy constraints, the theoretical best possible performance in nonparametric hypothesis testing is not well understood. Addressing this question complements our understanding of the estimation problem.

\begin{funding}
The research of Tony Cai was supported in part by  NIH grants R01-GM123056 and R01-GM129781.
\end{funding}

\begin{supplement}
\stitle{Supplement to ``Optimal Federated Learning for Nonparametric Regression with Heterogenous Distributed Differential Privacy Constraints"}
\sdescription{In the supplement to this paper, we present the detailed proofs for the main theorems in the paper ``Optimal Federated Learning for Nonparametric Regression with Heterogenous Distributed Differential Privacy Constraints''.}
\end{supplement}


\bibliographystyle{imsart-number} 
\bibliography{references}       

\pagebreak 
\setcounter{page}{1}

\begin{frontmatter}
\title{Supplementary Material to ``Optimal Federated Learning for Nonparametric Regression with Heterogenous Distributed Differential Privacy Constraints''}
\runtitle{Federated Nonparametric Private Estimation}

\begin{aug}
\author[A]{\fnms{T. Tony}~\snm{Cai}\ead[label=e1]{tcai@wharton.upenn.edu}},
\author[A]{\fnms{Abhinav}~\snm{Chakraborty}\ead[label=e2]{abch@wharton.upenn.edu}}
\and
\author[A]{\fnms{Lasse}~\snm{Vuursteen}\ead[label=e3]{lassev@wharton.upenn.edu}}
\address[A]{Department of Statistics and Data Science,
University of Pennsylvania\printead[presep={,\ }]{e1,e2,e3}}
\end{aug}

\begin{abstract}

In this supplement, we present the detailed proofs for the  main results in the paper ``Optimal Federated Learning for Nonparametric Regression with Heterogenous Distributed Differential Privacy Constraints''.

\end{abstract}

\begin{keyword}[class=MSC2020]
\kwd[Primary ]{62G08}
\kwd{62C20}
\kwd[; secondary ]{68P27}
\end{keyword}

\begin{keyword}
\kwd{Distributed methods}
\kwd{Nonparametric}
\kwd{Minimax optimal}
\end{keyword}

\end{frontmatter}

\setcounter{section}{0}
\setcounter{equation}{0}
\renewcommand{\theequation}{S.\arabic{equation}}
\renewcommand{\thesection}{\Alph{section}}

\section{Proof of the main results in Section \ref{sec:main_results}}\label{sec:appendix}

\subsection{Proof of Theorem \ref{thm:rate_summary}}

\begin{proof}
The statement of the theorem follows from Theorems \ref{thm:minimax-upper-bound} and \ref{thm:minimax-lower-bound} for the global risk and Theorems \ref{thm:upper-bound-pointwise-risk} and \ref{thm:lower_bound-pointwise-risk} for the pointwise risk. Depending on which problem we consider, the global or the pointwise estimation problem, the minimax rate upto logarithmic factors of the order at most $\log^2(N)$ and $\log(N)$ respectively, is given by $2^{-2L\gamma}$ where $\gamma = \alpha$ for the former problem and $\gamma = \nu$ for the latter, where $L$ is given by $(l_0 + 1) \vee \lceil \log_2(D) \rceil$. Furthermore, each of the minimax risks is trivially bounded over the class $\cB^{\alpha,R}_{p,q}$ since $f \in \cB^{\alpha,R}_{p,q}$ are bounded by a uniform constant (see Lemma \ref{lem:uniform_bound_on_f}) and the estimator $\hat{f} \equiv 0$ does not require any information to be shared, so it is trivially $(\bm\epsilon,\bm\delta)$-DP.
\end{proof}

\subsection{Proof of Corollaries \ref{thm:equal_budget_global_risk} and \ref{thm:pointwise_risk_homogeneous_setting}}

\begin{proof}
For these choices of $\varepsilon_j$ and $n_j$, $D$ is defined by
\[
D^{2\gamma+2} = m( (n^2 \varepsilon^2) \wedge (n D)).
\]
The solution to this equation is $D = (mn)^{\frac{1}{2\gamma + 1}} \wedge (mn^2\varepsilon^2)^{\frac{1}{2\gamma+ 2}}$. In the regime where $\varepsilon > (\sqrt mn)^{-1}$, we have that $D >1$ and consequently $2^L \asymp D$ for $L = (l_0 + 1) \vee \lceil \log_2 D \rceil$. The minimax rate $ (mn)^{-\frac{2\gamma}{2\gamma + 1}} \wedge (mn^2\varepsilon^2)^{-\frac{2\gamma}{2\gamma+ 2}}$ consequently follows from Theorems \ref{thm:minimax-upper-bound} and \ref{thm:minimax-lower-bound}. Whenever $\varepsilon \leq (\sqrt mn)^{-1}$, we find that $D \leq 1$, which consequently yields $L=1$.

\end{proof}

\subsection{Proof of Corollary \ref{thm:heterogeneous_no_dominant_budget_both_risks}}

\begin{proof}
	Let $L$ be given by $(l_0 + 1) \vee \lceil \log_2(D) \rceil$, where $D$ is  given by
	\[
	D^{2\gamma+2} = \left(\sum_{j=1}^m n_j^2 \varepsilon_j^2\right) \wedge \left( n_j D \right).
	\]
	Depending on which problem we consider, the global or the pointwise estimation problem, Theorem \ref{thm:rate_summary} implies that the minimax rate follow upto logarithmic factors is given by $D^{-2L\gamma} \wedge 1$ where $\gamma = \alpha$ for the former problem and $\gamma = \nu$ for the latter. Using~\eqref{eq:no_dominant_budget} the equation in the display above reduces to $D^{2\gamma+2} = \sum_{j=1}^m n_j^2 \varepsilon_j^2$. To see that observe that if we set $D$ such that $D^{2\gamma+2} = \sum_{j=1}^m n_j^2 \varepsilon_j^2$ then we have by using~\eqref{eq:no_dominant_budget} that $n_j^2\varepsilon_j^2 \wedge n_jD = n_j^2\varepsilon_j^2 $ for $j= 1,\ldots,m$,  and hence the above display simplifies.
	
	If $\sum_{j=1}^m n_j^2 \varepsilon_j^2 \to \infty $ we have that $D \to \infty$, in which case we obtain the rate $ \left(\sum_{j=1}^m n_j^2 \varepsilon_j^2\right)^{-\frac{2\gamma}{2\gamma + 2}}$, upto logarithmic factors. Whenever $\sum_{j=1}^m n_j^2 \varepsilon_j^2 \lesssim 1$, we have $D \lesssim 1$ which in turn implies that the risks are bounded away from zero by a constant.
\end{proof}

\subsection{Proof of Corollary \ref{thm:heterogeneous_dominant_budget_both_risks}}

\begin{proof}
	As in Corollary \ref{thm:heterogeneous_no_dominant_budget_both_risks} we tackle both the global and pointwise problem together.
	Let us denote by $D^* =  \left( n^2_{j^*} \varepsilon_{j^*}^2 \right)^{\frac{1}{2\gamma+2}} \wedge n_{j^*}^{\frac{1}{2\gamma+1}}$, condition~\eqref{eq:dominant_budget} implies that 
\begin{align*}
	\left( n_{j^*}^2 \varepsilon_{j^*}^2 \right) \wedge \left(n_{j^*} \left( \left( n^2_{j^*} \varepsilon_{j^*}^2 \right)^{\frac{1}{2\gamma+2}} \wedge n_{j^*}^{\frac{1}{2\gamma+1}} \right) \right)\ge \underset{[m] \setminus \{j^*\}}{\overset{}{\sum}} n_j^2 \varepsilon_j^2  ,
\end{align*}
	$n^2_{j^*}\varepsilon_{j^*}^2 \wedge n_{j^*} D^* \geq \sum_{j \in [m] \setminus \{j^*\}}n_j^2 \varepsilon_j^2  $. We show below that under the previously mentioned condition, 
	\[
	D^{2\gamma+2} = \sum_{j=1}^m \left( n_j^2 \varepsilon_j^2 \right) \wedge \left( n_j D \right) \quad \implies \quad D^* \leq D \leq c D^*
	\]
	for some fixed constant $c > 1$. 

	Using the fact that $D$ solves the equation above, we have that $D^{2\gamma+2} \geq n_{j^*}^2 \varepsilon_{j^*}^2 \wedge n_{j^*} D$, as all the summands are positive. Hence, $D \geq D^*$. For obtaining an upper bound we proceed as follows:
	\begin{align*}
		D^{2\gamma + 2} &\leq n_{j^*}^2 \varepsilon_{j^*}^2 \wedge n_{j^*} D + \sum_{j \in [m] \setminus \{j^*\}}n_j^2 \varepsilon_j^2 \\
		& \leq n_{j^*}^2 \varepsilon_{j^*}^2 \wedge n_{j^*} D + n_{j^*}^2 \varepsilon_{j^*}^2 \wedge n_{j^*} D^* \quad(\text{by}~\eqref{eq:dominant_budget} )\\
		&\leq  2 n_{j^*}^2 \varepsilon_{j^*}^2 \wedge n_{j^*} D \quad(\text{by } D^* \leq D) 
	\end{align*}
	The last line can be shown to imply $D \leq c \left( n^2_{j^*} \varepsilon_{j^*}^2 \right)^{\frac{1}{2\gamma+2}} \wedge n_{j^*}^{\frac{1}{2\gamma+1}} = cD^*$, for some $c >1$. Hence we can conclude that $D \asymp D^*$, which gives us the proposed rate $(D^*)^{-2\nu}$.
\end{proof}
\section{Proofs of the global risk estimation results}
\label{sec:proofs}

We would divide the proofs of the global estimation problem into two sections. Section \ref{sec:upper-bound-proofs} presents the proof the lemmas and theorem pertaining to the proving the minimax upper bounds, whereas Section \ref{sec:lower-bound-proofs} states the proofs of the results pertaining to the lower bound. 
\subsection{Upper Bound Proofs}\label{sec:upper-bound-proofs}
\subsubsection{Proof of Lemma \ref{lemma:l2-sensitivity-wavelets}}\label{sec:proof-of-l2-sensitivity-for-wavelts}
\begin{proof}
	Consider two neighboring data sets $Z^{(j)}$ and $\tilde{Z}^{(j)}$, meaning that at most for one $i \in [n]$ it holds that
\begin{equation*}
Z^{(j)}_i = (Y_i^{(j)},X_i^{(j)}) \neq (\tilde Y_i^{(j)},\tilde X_i^{(j)}) = \tilde Z^{(j)}_i.
\end{equation*}
Let $\hat f^{(j)}_{lk;\tau}$ be as in \eqref{eq:estimated_wavelet_coefficient} and consider
\begin{equation*}
	\tilde f^{(j)}_{lk;\tau} := \frac{1}{n_j}\sum_{i=1}^{n_j} \left[\tilde Y_i^{(j)}\right]_{\tau} \psi_{lk}\left(\tilde X_i^{(j)}\right).
\end{equation*}
	By Plancharel's theorem, the $L_2$-sensitivity of the statistic $\bm{\hat{f}}^{(j)}_{L,\tau}$ is given by
	\begin{align*}
		\left\|\bm{\hat{f}}^{(j)}_{L,\tau}(Z^{(j)}) - \bm{\hat{f}}^{(j)}_{L,\tau}((\tilde{Z}^{(j)}))\right\|^2_2 &= \sum_{l = l_0}^L \sum_{k=0}^{2^l -1 } (\hat f_{lk;\tau}^{(j)} - \tilde{f}_{lk;\tau}^{(j)})^2 \\
		&= \sum_{l = l_0}^L \sum_{k=0}^{2^l -1 } \left(\frac{1}{n_j} [Y_i^{(j)}]_\tau \psi_{lk}(X_i^{(j)}) - \frac{1}{n_j} [\tilde Y_i^{(j)}]_\tau \psi_{lk}(\tilde{X}_i^{(j)})\right)^2.
	\end{align*}
	Using the bound $\|\psi_{lk}\|_\infty = 2^{l/2}\|\psi\|_\infty$ for all $k = 0,\dots, 2^{l} -1$, we have that
	\begin{align*}	
		\frac{1}{n_j} [Y_i^{(j)}]_\tau \psi_{lk}(X_i^{(j)}) - \frac{1}{n_j} [\tilde Y_i^{(j)}]_\tau \psi_{lk}(\tilde{X}_i^{(j)}) \leq \frac{\tau}{n_j}\|\psi\|_\infty 2^{l/2+1}.
	\end{align*}
Since the wavelet $\psi$ is compactly supported, there is a constant number of basis functions $\psi_{lk}$ with overlapping support at each resolution level $l$ (the constant depends on the support size of the chosen wavelet $\psi$). Consequently, the latter can be further bounded 
	\begin{align*}
		 \sum_{l = l_0}^L \sum_{k=0}^{2^l -1 } \left(\frac{\tau}{n_j}\|\psi\|_\infty2^{l/2+1}\right)^2 \mathbbm{1}\left\{\{X_i^{(j)},\tilde{X}_i^{(j)}\} \cap \mathrm{supp}(\psi_{lk}) \neq \emptyset\right\} &\leq 
		 \left(\frac{\tau}{n_j}\|\psi\|_\infty \right)^2 c_A 2^{L+3}.
	\end{align*}
Representing the constants that depend on $\psi$ as a single constant, we have shown that for some $c_\psi > 0$ it holds that
\[
\left\|\bm{\hat{f}}^{(j)}_{L,\tau}(Z^{(j)}) - \bm{\hat{f}}^{(j)}_{L,\tau}(\tilde{Z}^{(j)}) \right\|_2 \leq c_\psi \frac{\tau \sqrt{2^L}}{n_j}.
\]

\end{proof}
\subsubsection{Proof of Theorem \ref{thm:minimax-upper-bound}}
\begin{proof}
	Let us denote the vector  $\bm{f}_{L} := \{ f_{lk}: \,k=1\dots,2^l,\,l=1,\dots,L\}$. We decompose the expected $L_2$-error $\E_f\|\hat f_{L,\tau} -f\|_2^2$ as follows 
	\begin{align}
		\E_f \left\|\hat f_{L,\tau} -f\right\|_2^2 &=\E_f \left\|\sum_{j=1}^m u_j \bm{T}^{(j)}_{L,\tau} - \bm f_L \right\|_2^2 +\sum_{l > L} \sum_{k=0}^{2^l-1} f_{lk}^2 \nonumber \\
		&\leq 2\E \left\|\sum_{j=1}^m u_j \bm{\hat{f}}^{(j)}_{L,\tau} - \bm f_L \right\|_2^2+ 2\E\left\|\sum_{j=1}^m u_j \bm W^{(j)}\right\|_2^2 +\sum_{l > L} \sum_{k=0}^{2^l-1} f_{lk}^2. \label{eq:global_risk_break_down_ub}
	\end{align}
	The equality follows from the orthonormality of the basis functions, the inequality from the triangle inequality in conjunction with the parabolic inequality $2ab \leq a^2 + b^2$. Lemma \ref{lemma:proof-of-bound-on-tail-sum} in Section \ref{ssec:proofs-upper-bound-global-risk-auxiliary-lemmas} yields the following bound on the third term;
	\begin{equation}\label{eq:bound-on-tail-sum}
		\sum_{l > L} \sum_{k=0}^{2^l-1} f_{lk}^2 \leq c_\alpha 2^{-2L\alpha} R^{2},
	\end{equation}
for a universal constant $c_\alpha > 0$ depending only on $\alpha$. Next we bound the error term due to the noise added for privacy purposes as follows,
	\begin{equation}\label{eq:bound-on-error-due-to-privacy}
	\E_f\left\|\sum_{j=1}^m u_j \bm W^{(j)}\right\|_2^2 \leq  \frac{4 c_\psi^2 \tau^2 2^{2L}\log(2/\delta')}{\sum_j n^2_j\varepsilon^2_j \wedge n_j2^L},
	\end{equation}
which we prove in Lemma \ref{lemma:proof-of-bound-on-error-due-to-privacy} below.

Next, we consider the first term in \eqref{eq:global_risk_break_down_ub}, which by orthogonality of the wavelet basis can be written as
\begin{align*}
	\E_f \left\|\sum_{j=1}^m u_j \bm{\hat{f}}^{(j)}_{L,\tau} - \bm f_L \right\|_2^2 &= \E \sum_{l=l_0}^L\sum_{k=0}^{2^l-1} \left(\sum_{j=1}^m u_j \hat f_{lk}^{(j)} - f_{lk}\right)^2\\
	 &= \sum_{l=l_0}^L\sum_{k=0}^{2^l-1} \E_f \left(\sum_{j=1}^m \frac{u_j}{n_j}\sum_{i=1}^{n_j}\left[Y^{(j)}_i\right]_\tau \psi_{lk}(X^{(j)}_i)  - f_{lk}\right)^2.
\end{align*}
We have that
\[
\left[Y^{(j)}_i\right]_\tau = \left[f(X^{(j)}_i) + \xi_i^{(j)}\right]_\tau = \left[ \xi_i^{(j)}\right]^{\tau-f(X^{(j)}_i)}_{-\tau-f(X^{(j)}_i)} + f(X^{(j)}_i).
\]
Writing $\eta^{(j)}_i := \left[ \xi_i^{(j)}\right]^{\tau-f(X^{(j)}_i)}_{-\tau-f(X^{(j)}_i)}$, we obtain
\begin{align}
	 \sum_{l=l_0}^L\sum_{k=0}^{2^l-1} \E_f \left(\sum_{j=1}^m \frac{u_j}{n_j}\sum_{i=1}^{n_j}\left[Y^{(j)}_i\right]_\tau \psi_{lk}(X^{(j)}_i)  - f_{lk}\right)^2 &= \nonumber \\
	 \sum_{l=l_0}^L\sum_{k=0}^{2^l-1} \E_f \left(\sum_{j=1}^m \frac{u_j}{n_j}\sum_{i=1}^{n_j}f(X_i^{(j)}) \psi_{lk}(X^{(j)}_i) + \sum_{j=1}^m \frac{u_j}{n_j}\sum_{i=1}^{n_j}\eta_i^{(j)}\psi_{lk}(X^{(j)}_i)  - f_{lk}\right)^2 &\leq \nonumber \\
	  2\sum_{l=l_0}^L\sum_{k=0}^{2^l-1} \E_f \left(\sum_{j=1}^m \frac{u_j}{n_j}\sum_{i=1}^{n_j}f(X_i^{(j)}) \psi_{lk}(X^{(j)}_i)  - f_{lk}\right)^2 +& \nonumber \\ 2\sum_{l=l_0}^L\sum_{k=0}^{2^l-1} \E_f \left( \sum_{j=1}^m \frac{u_j}{n_j}\sum_{i=1}^{n_j}\eta_i^{(j)}\psi_{lk}(X^{(j)}_i)  \right)^2.&	 \label{eq:last_line_decomp_1}
\end{align}
We show in Lemma \ref{lemma:proof-of-bounding-non-private-part-1} below that
\begin{equation}\label{eq:bounding-non-private-part-1}
	\sum_{l=l_0}^L\sum_{k=0}^{2^l-1} \E_f \left(\sum_{j=1}^m \frac{u_j}{n_j}\sum_{i=1}^{n_j}f(X_i^{(j)}) \psi_{lk}(X^{(j)}_i)  - f_{lk}\right)^2 \leq C_{R,\alpha}^2c_3 2^L \sum_{j=1}^m \frac{u_j^2}{n_j}.
\end{equation}
The second term in \eqref{eq:last_line_decomp_1} can be decomposed as follows
\begin{align*}
	\sum_{l=l_0}^L\sum_{k=0}^{2^l-1} &\E_f \left( \sum_{j=1}^m \frac{u_j}{n_j}\sum_{i=1}^{n_j}\eta_i^{(j)}\psi_{lk}(X^{(j)}_i)  \right)^2 = \overbrace{\sum_{l=l_0}^L\sum_{k=0}^{2^l-1} \E_f \left[\sum_{j=1}^m \sum_{i=1}^{n_j}\frac{u^2_j}{n^2_j}\left(\eta_i^{(j)}\psi_{lk}(X^{(j)}_i)\right)^2\right]}^{\textcircled{1}}\\
	&+   \underbrace{\sum_{l=l_0}^L\sum_{k=0}^{2^l-1} \E_f \left[\sum_{j=1}^m \sum_{i=1}^{n_j}\sum_{(i',j')\neq (i,j)}\frac{u_j}{n_j}\frac{u_{j'}}{n_{j'}}\left(\eta_i^{(j)}\psi_{lk}(X^{(j)}_i)\eta_{i'}^{(j')}\psi_{lk}(X^{(j')}_{i'})\right)\right]}_{\textcircled{2}}.
\end{align*}
We show in Lemma \ref{lemma:proof-of-bounding-non-private-part-2} and Lemma \ref{lemma:proof-of-bounding-non-private-part-3}  (Section \ref{ssec:proofs-upper-bound-global-risk-auxiliary-lemmas}) respectively that
\begin{equation}\label{eq:bounding-non-private-part-2}
	\textcircled{1} \leq c_4 2^L \sum_{j=1}^m \frac{u_j^2}{n_j},
\end{equation}
and
\begin{equation}\label{eq:bounding-non-private-part-3}
	\textcircled{2} \leq c_5 R^2 2^{-2\alpha L }
\end{equation}
for constants $c_4, c_5 > 0$. Combining~\eqref{eq:bounding-non-private-part-1},~\eqref{eq:bounding-non-private-part-2} and~\eqref{eq:bounding-non-private-part-3} we have for some constant $c_6 > 0$
\begin{equation}\label{eq:bound-on-clipped-non-private-part}
	\E \left\|\sum_{j=1}^m u_j \bm{\hat{f}}^{(j)}_{L,\tau} - \bm f_L \right\|_2^2 \leq c_6 \left( 2^L \sum_{j=1}^m \frac{u_j^2}{n_j} +   2^{-2\alpha L } \right).
\end{equation}	
Observe that the choice $u_j  = \frac{v_j}{\sum_jv_j}$ with $v_j = n_j^2\varepsilon_j^2 \wedge n_j 2^L$ and the above gives
\[
\sum_{j=1}^m \frac{u_j^2}{n_j} \leq \frac{1}{\sum_j n^2_j\varepsilon^2_j \wedge n_j2^L}\sum_{j=1}^m u_j (n^2_j\varepsilon^2_j \wedge n_j2^L) \frac{1}{n_j}\leq \frac{2^{L}}{\sum_j n^2_j\varepsilon^2_j \wedge n_j2^L} \lesssim 2^{(-2\alpha - 1)L}.
\]
Next we combine~\eqref{eq:bound-on-tail-sum},~\eqref{eq:bound-on-error-due-to-privacy} and~\eqref{eq:bound-on-clipped-non-private-part} with the above to obtain that 
\begin{equation*}
	\E \left\|\hat f_{L,\tau} -f\right\|_2^2 \leq c_\psi' \tau^2 \log(2/\delta')2^{-2L\alpha},
\end{equation*}
for a constant $c_\psi' > 0$. The choice $L = (l_0 + 1) \vee \lceil \log_2(D) \rceil$ where $D > 0$ is the solution to \eqref{eq:dimension_determining_equation_global} gives
\[
	2^{L(2\alpha+2)} \asymp \sum_{j=1}^m n_j^2\varepsilon_j^2 \wedge n_j2^{L} \leq N 2^L,
\]
which implies $ L = O(\log N)$ and hence $\tau^2 = O(\log N)$.
\end{proof}

\subsubsection{Auxilliary Lemmas \ref{lemma:proof-of-bound-on-error-due-to-privacy},  \ref{lemma:proof-of-bounding-non-private-part-1}, \ref{lemma:proof-of-bounding-non-private-part-2}, \ref{lemma:proof-of-bounding-non-private-part-3}, \ref{lemma:proof-of-bound-on-tail-sum} and \ref{lem:uniform_bound_on_f}}\label{ssec:proofs-upper-bound-global-risk-auxiliary-lemmas}


\begin{lemma}\label{lemma:proof-of-bound-on-error-due-to-privacy}
	We have that
	\[
	\E\left\|\sum_{j=1}^m u_j \bm W^{(j)}\right\|_2^2 \leq  \frac{2c_\psi^2 \tau^2 2^{2L}\log(2/\delta')}{\sum_j n^2_j\varepsilon^2_j \wedge n_j2^L},
	\]
	where $\bm W^{(j)} \sim N\left(0, \frac{4 \tau^2 2^{L}c_\psi\log(2/\delta_j)}{n_j^2\varepsilon_j^2} I_{2^L}\right)$, 
\begin{equation*}
	u_j  = \frac{v_j}{\sum_jv_j} \; \text{ with } \; v_j = n_j^2\varepsilon_j^2 \wedge n_j 2^L
\end{equation*}
	 and $c'_\psi > 0$ a constant depending on the choice of wavelet basis, $R$ and $\alpha > 1/p$.
\end{lemma}
\begin{proof}
Using the independence of the noise generating mechanisms across the servers we have
\begin{align*}
	\E\left[\left\|\sum_{j=1}^m  u_jW^{(j)}\right\|^2\right] &=\sum_{i=1}^m u^2_j\E\left[\left\|  W^{(j)} \right\|_2^2\right]\\
	&= \sum_{j=1}^m u_j^2 \frac{(c_\psi)^2 \tau^2 \log(2/\delta_j)}{n_j^2\varepsilon_j^2}2^{2L+2}.
\end{align*}
Using the fact that $v_j \leq n_j^2\varepsilon_j^2$ and $\sum_j u_j =1$, the latter is bounded by from above by
\begin{align*}
	 \frac{4c_\psi}{\sum_j n^2_j\varepsilon^2_j \wedge n_j2^L}\sum_{j=1}^m u_j v_j\frac{ \tau^2 2^{2L}\log(2/\delta')}{n_j^2\varepsilon_j^2} &\leq\\
	 \frac{4c_\psi}{\sum_j n^2_j\varepsilon^2_j \wedge n_j2^{L}}\sum_{j=1}^m u_j \tau^2 2^{2L}\log(2/\delta') &\leq\\
	 \frac{4c_\psi \tau^2 2^{2L}\log(2/\delta')}{\sum_j n^2_j\varepsilon^2_j \wedge n_j2^L}.&
\end{align*}
\end{proof}
\begin{lemma}\label{lemma:proof-of-bounding-non-private-part-1}
	It holds that
	\[
	\sum_{l=l_0}^L\sum_{k=0}^{2^l-1} \E_f \left(\sum_{j=1}^m \frac{u_j}{n_j}\sum_{i=1}^{n_j}f(X_i^{(j)}) \psi_{lk}(X^{(j)}_i)  - f_{lk}\right)^2 \leq C_{R,\alpha}^2c_3 2^L \sum_{j=1}^m \frac{u_j^2}{n_j},
	\]
	where $f_{lk} = \langle f,\psi_{lk}\rangle$  are the wavelet basis coefficents of the function $f \in B^\alpha_{p,q}(R)$ and $u_j$ is as defined in~\eqref{eq:estimator_weights_global}.
\end{lemma}
\begin{proof}
Using the independence of samples on a given server and using that $\E f(V) \psi_{lk}(V) = f_{lk}$ for $V \sim U[0,1]$, we have
\begin{align*}
	\sum_{l=l_0}^L\sum_{k=0}^{2^l-1} \E_f \left(\sum_{j=1}^m \frac{u_j}{n_j}\sum_{i=1}^{n_j}f(X_i^{(j)}) \psi_{lk}(X^{(j)}_i)  - f_{lk}\right)^2 
	&=  \\	\sum_{l=l_0}^L\sum_{k=0}^{2^l-1} \sum_{j=1}^m \sum_{i=1}^{n_j} \frac{u^2_j}{n^2_j}\mathrm{Var}\left(f(X_i^{(j)}) \psi_{lk}(X^{(j)}_i) \right) &\leq \\
	 \sum_{l=l_0}^L\sum_{k=0}^{2^l-1} \sum_{j=1}^m \sum_{i=1}^{n_j} \frac{u^2_j}{n^2_j}\E_f \left(f(X_i^{(j)}) \psi_{lk}(X^{(j)}_i) \right)^2.
\end{align*}
By Lemma \ref{lem:uniform_bound_on_f}, $\|f\|_\infty \leq C_{R,\alpha}$ for a universal constant $C_{R,\alpha}>0$ depending on $R > 0$ and $s > 1/p$ only. Hence, each term in the display above satisfies
\begin{align*}
	\E_f \left(f(X_i^{(j)}) \psi_{lk}(X^{(j)}_i) \right)^2 \leq C_{R,\alpha}^2 \E_f \left( \psi_{lk}(X^{(j)}_i)^2\right) = C_{R,\alpha}^2.
\end{align*}
We obtain that
\begin{align*}
		\sum_{l=l_0}^L\sum_{k=0}^{2^l-1} \E_f \left(\sum_{j=1}^m \frac{u_j}{n_j}\sum_{i=1}^{n_j}f(X_i^{(j)}) \psi_{lk}(X^{(j)}_i)  - f_{lk}\right)^2 
		&\leq \sum_{l=l_0}^L\sum_{k=0}^{2^l-1} \sum_{j=1}^m \sum_{i=1}^{n_j} \frac{u^2_j}{n^2_j}C_{R,\alpha}^2 \\
		&\leq 2^{L+1}C_{R,\alpha}^2 \sum_{j=1}^m \frac{u_j^2}{n_j}.
\end{align*}
\end{proof}
\begin{lemma}\label{lemma:proof-of-bounding-non-private-part-2}
It holds that
	\[
	\sum_{l=l_0}^L\sum_{k=0}^{2^l-1} \E \left[\sum_{j=1}^m \sum_{i=1}^{n_j}\frac{u^2_j}{n^2_j}\left(\eta_i^{(j)}\psi_{lk}(X^{(j)}_i)\right)^2\right] \leq c_4 2^L \sum_{j=1}^m \frac{u_j^2}{n_j},
	\]
	where $\eta^{(j)}_i := \left[ \xi_i^{(j)}\right]^{\tau-f(X^{(j)}_i)}_{-\tau-f(X^{(j)}_i)}$ with $\tau = C_{\alpha,R} +\sqrt{(2\alpha + 1)L}$ and $u_j$ is as defined in~\eqref{eq:estimator_weights_global}.
\end{lemma}
\begin{proof}
We first observe that since $\tau > C_{\alpha,R} > \|f\|_\infty$, Lemma \ref{lem:folklore_clipping_reduces_variance} yields that
\begin{align*}
	\E\left[(\eta_i^{(j)})^2 \mid X^{(j)}_i\right] &= \E \left[\left(\left[ \xi_i^{(j)}\right]^{\tau-f(X^{(j)}_i)}_{-\tau-f(X^{(j)}_i)}\right)^2 \mid X^{(j)}_i\right]\\
	&\leq \E\left[(\xi_i^{(j)})^2 \mid X^{(j)}_i\right] = \sigma^2,
\end{align*}
which in turn implies that
\begin{align*}
	\E \left[(\eta_i^{(j)})^2 \psi^2_{lk}(X^{(j)}_i)\right] &= \E \left[\psi^2_{lk}(X^{(j)}_i)\E \left[(\eta_i^{(j)})^2\mid X_i^{(j)}\right] \right]\\
	&\leq \sigma^2 \E \psi^2_{lk}(X^{(j)}_i) = \sigma^2.
\end{align*}
Hence, 
\[
\sum_{l=l_0}^L\sum_{k=0}^{2^l-1} \E \left[\sum_{j=1}^m \sum_{i=1}^{n_j}\frac{u^2_j}{n^2_j}\left(\eta_i^{(j)}\psi_{lk}(X^{(j)}_i)\right)^2\right] \leq  \sum_{l=l_0}^L\sum_{k=0}^{2^l-1} \sum_{j=1}^m \sum_{i=1}^{n_j} \frac{u^2_j}{n^2_j}\sigma^2 \leq 2^{L+1} \sum_{j=1}^m \frac{u_j^2}{n_j}.
\]
\end{proof}
\begin{lemma}\label{lemma:proof-of-bounding-non-private-part-3}
It holds that
	\[
	\sum_{l=l_0}^L\sum_{k=0}^{2^l-1} \E \left[\sum_{j=1}^m \sum_{i=1}^{n_j}\sum_{(i',j')\neq (i,j)}\frac{u_j}{n_j}\frac{u_{j'}}{n_{j'}}\left(\eta_i^{(j)}\psi_{lk}(X^{(j)}_i)\eta_{i'}^{(j')}\psi_{lk}(X^{(j')}_{i'})\right)\right] \leq 2R^2 2^{-2\alpha L },
	\]
	where $\eta^{(j)}_i := \left[ \xi_i^{(j)}\right]^{\tau-f(X^{(j)}_i)}_{-\tau-f(X^{(j)}_i)}$ with $\tau = C_{\alpha,R} +\sqrt{(2\alpha + 1)L}$ and $u_j$ is as defined in~\eqref{eq:estimator_weights_global}.
\end{lemma}
\begin{proof}
Let $Z\sim N(0,1)$. Then, for $a>0$ and $\tau > a$,
\begin{align*}
	0< \E \left[Z\right]^{\tau+a}_{-\tau+a} &\leq \E \left[Z\right]^{\tau-a}_{-\tau+a} + (2a) \P(\tau-a \leq Z \leq \tau+a)\\
	&\leq 4a e^{-\frac{1}{2}(\tau-a)^2}.
\end{align*}
Similarly, for $a<0$ and $\tau > a$ we obtain that $0> \E \left[Z\right]^{\tau+a}_{-\tau+a} \geq 4a e^{-\frac{1}{2}(\tau+a)^2}$. Combining these two cases together we have for $a \in \R$ and $\tau$ large that
\begin{equation}\label{eq:truncated-gaussian-expectation-bound}
\left|\E \left[Z\right]^{\tau+a}_{-\tau+a}\right| \leq 4|a| e^{-\frac{1}{2}(\tau-|a|)^2}.
\end{equation}
Using the tower property consequently obtain
\begin{align*}
	\left|\E\eta_i^{(j)}\psi_{lk}(X^{(j)}_i)\eta_{i'}^{(j')}\psi_{lk}(X^{(j')}_{i'})\right| &= \left|\E\left[\psi_{lk}(X^{(j)}_i)\psi_{lk}(X^{(j')}_{i'})\E\left[\eta_i^{(j)}\eta_{i'}^{(j')}\mid X^{(j)}_i,X^{(j')}_{i'} \right]\right]\right|\\
	&\leq \E\left[\left|\psi_{lk}(X^{(j)}_i)\psi_{lk}(X^{(j')}_{i'})\right| \left|\E\left[\eta_i^{(j)}\eta_{i'}^{(j')}\mid X^{(j)}_i,X^{(j')}_{i'} \right]\right|\right].
\end{align*}
Next we focus on the conditional expectation and bound it using~\eqref{eq:truncated-gaussian-expectation-bound} as follows (observe that $(i,j) \neq (i',j')$)
\begin{align*}
	\left|\E\left[\eta_i^{(j)}\eta_{i'}^{(j')}\mid X^{(j)}_i,X^{(j')}_{i'} \right]\right| &\leq \left|\E\left[\eta_i^{(j)}\mid X^{(j)}_i \right]\right|^2\\
	&\leq \left|\E\left[\left[ \xi_i^{(j)}\right]^{\tau-f(X^{(j)}_i)}_{-\tau-f(X^{(j)}_i)}\mid X^{(j)}_i \right]\right|^2\\
	&\leq 16|f(X^{(j)}_i)|^2 e^{-(\tau-|f(X^{(j)}_i)|)^2} \leq 16R^2 e^{-(\tau-C_{\alpha,R})^2}.
\end{align*}
This implies that
\begin{align*}
	\left|\E\eta_i^{(j)}\psi_{lk}(X^{(j)}_i)\eta_{i'}^{(j')}\psi_{lk}(X^{(j')}_{i'})\right| 
	&\leq \left(\E|\psi_{lk}(X^{(j)}_i)|\right)^216R^2e^{-(\tau-C_{\alpha,R})^2}\\ 
	&\leq16 R^2e^{-(\tau-C_{\alpha,R})^2}.
\end{align*} 
Plugging this inequality in~\eqref{eq:bounding-non-private-part-3} and using the prescribed choice of $\tau = C_{\alpha,R} +\sqrt{(2\alpha + 1)L}$ yields
\begin{align*}
	\sum_{l=l_0}^L\sum_{k=0}^{2^l-1} \E \left[\sum_{j=1}^m \sum_{i=1}^{n_j}\sum_{(i',j')\neq (i,j)}\frac{u_j}{n_j}\frac{u_{j'}}{n_{j'}}\left(\eta_i^{(j)}\psi_{lk}(X^{(j)}_i)\eta_{i'}^{(j')}\psi_{lk}(X^{(j')}_{i'})\right)\right]  &\leq  \\ \sum_{l=l_0}^L\sum_{k=0}^{2^l-1} \sum_{j=1}^m \sum_{i=1}^{n_j}\sum_{(i',j')\neq (i,j)} \frac{u_j}{n_j}\frac{u_{j'}}{n_{j'}} 16 R^2e^{-(\tau-C_{\alpha,R})^2} &\leq\\
	 \sum_{l=l_0}^L\sum_{k=0}^{2^l-1}\left(\sum_{j=1}^m \sum_{i=1}^{n_j} \frac{u_j}{n_j}\right)^2  16 R^2e^{-(\tau-C_{\alpha,R})^2} &\leq\\
	 2R^2 2^L 2^{(-2\alpha - 1)L }.
\end{align*}

\end{proof}

The following two lemmas are standard results, see e.g. Chapter 9 in \cite{johnstone2019manuscript}.

\begin{lemma}\label{lemma:proof-of-bound-on-tail-sum}
	Let $f_{lk}$ are the wavelet coefficents of the function $f \in B^\alpha_{p,q}(R)$. For any $1\leq q \leq \infty$, $2 \leq p \leq \infty$, $L >0$, we have
	\[
	\sum_{l > L} \sum_{k=0}^{2^l-1} f_{lk}^2 \leq c_\alpha 2^{-2L\alpha} R^{2},
	\]
	where $c_\alpha > 0$ is a universal constant depending only on $\alpha$.
\end{lemma}

\begin{lemma}\label{lem:uniform_bound_on_f}
There exists a constant $C_{\alpha,R} > 0$ such that $\|f\|_\infty \leq C_{\alpha,R}$ for all $f \in B^\alpha_{p,q}(R)$ with $\alpha - 1/2 - 1/p > 0$, $1 \leq q \leq \infty$ and $2 \leq p \leq \infty$.
\end{lemma}

The last lemma of the section is a standard result too, which states that clipping a random variable reduces its variance.

\begin{lemma}\label{lem:folklore_clipping_reduces_variance}
For any $\tau > 0$ and random variable $V$,
\begin{equation*}
 \text{Var}\left( \left[ V \right]_{\tau}\right) \leq \text{Var}\left( V \right).
\end{equation*}
\end{lemma}
\begin{proof}
Let $\mu = \E V$. Since the expectation of a random variable is the constant minimizing the $L_2$-distance to that random variable,
\begin{equation*}
\text{Var}\left( \left[ V \right]_{\tau}\right) \leq \E  \left( \left[ V  \right]_{\tau} - \mu \right)^2 .
\end{equation*}
The latter expectation can be written as
\begin{align*}
\E \mathbbm{1}_{V \in [-\tau,\tau]} \left( V  - \mu \right)^2 + \E \mathbbm{1}_{V > \tau} \left( \tau - \mu \right)^2 + \E \mathbbm{1}_{V < -\tau} \left( - \tau  - \mu \right)^2.
\end{align*}
Assuming $\mu \geq 0$, $V < -\tau$ implies that $|- \tau  - \mu| \leq |- V  - \mu|$. Since $\mu \leq \tau$, $V > \tau$ implies that $|\tau - \mu| \leq |V  - \mu|$. Consequently, the above display is bounded from above by
\begin{align*}
 \E \mathbbm{1}_{V \in [-\tau,\tau]} \left( V  - \mu \right)^2 + \E \mathbbm{1}_{V > \tau} \left( V - \mu \right)^2 + \E \mathbbm{1}_{V < -\tau} \left( - V  - \mu \right)^2 = \E  \left(  V   - \mu \right)^2.
\end{align*}
The case where $\mu < 0$ follows by the same reasoning.
\end{proof}

\subsection{Lower Bound Proofs}\label{sec:lower-bound-proofs}

\subsubsection{Proof of Lemma \ref{lemma:van-trees-nonparametric}}
\begin{proof}
	Note that $\E_f\|\hat f -f\|_2^2 \geq \E\left(\sum_{k=0}^{2^L-1} (\hat f_{Lk}-f_{Lk})^2\right)$. Denote $f_L$ by the vector $\{f_{Lk}: k=0,\dots,2^L-1\}.$ 
	For getting a lower bound on the minimax risk we will use a multivariate version of the Van-Trees inequality due to \cite{gill1995applications} (Theorem 1), which bounds the average $\ell^2$ risk by
	\begin{equation*}
		\int_{f_L}\E\left(\sum_{k=0}^{2^L-1} (\hat f_{Lk}-f_{Lk})^2\right) \lambda(f_L) df_L \geq \frac{(2^L)^2}{\int\mathrm{Tr}(I_{T^{(1)},\ldots,T^{(m)}}(f_L))\pi(f_L)df_L + J(\pi)},
	\end{equation*}
	where $I_{T^{(1)},\ldots,T^{(m)}}(f_L)$ is the Fisher information asscociated with $\bm T =\left(T^{(1)},\ldots,T^{(m)}\right)$ and  $\pi(f_L) = \prod_{k=0}^{2^L-1} \pi_k(f_{Lk})$ is a prior for the parameter $f_L$ and $J(\pi)$ is the Fisher information associated with the prior $\pi$:
	\[
	J(\pi) = \sum_{k=0}^{2^L-1} \int \frac{\pi'_k(f_{Lk})^2}{\pi_k(f_{Lk})} df_{Lk}.
	\]
	By taking as prior $\pi_k$ a rescaled version of the density $t \mapsto \cos^2(\pi t/2) \mathbbm{1}\{|t| \leq 1\}$ such that it is supported on $[-2^{-L(\alpha+1/2)}R,2^{-L(\alpha+1/2)}R]$, we obtain that prior is supported on $\cB^{\alpha,R}_{p,q}$. By a straightforward calculation for the Fisher information associated with the prior, we have that the minimax risk is lower bounded by
	\[
	\sup_{f \in \cB^{\alpha,R}_{p,q}} \E_f\|\hat f - f\|_2^2 \geq \frac{(2^L)^2}{\sup_{f_L}\mathrm{Tr}(I_{T^{(1)},\ldots,T^{(m)}}(f_L)) + \frac{2^L\pi^2}{(2^{-L(\alpha+1/2)}R)^2}}.
	\]
	Next using the fact that $T^{(1)},\dots, T^{(m)}$ are independent trancripts we have that
	$I_{T^{(1)},\ldots,T^{(m)}}(f_L)  = \sum_{j=1}^m I_{T^{(j)}}(f_L)$. It is a straighforward calculation to show that $ I_{T^{(j)}}(f_L) = \E[\bm C_f(T^{(j)})]$ which proves the lemma.
\end{proof}
\subsubsection{Proof of Lemma \ref{lem:trace-processing}}
\label{sec:proof-of-trace-processing}
\begin{proof}
	For random variables $V,W$, it holds that
	\begin{equation*}
		\E W \E[ W | V] = \E \E[ W | V] \E[ W | V],
	\end{equation*}
	since $W - \E[ W | V]$ is orthogonal to $\E[ W | V]$. Using this fact and linearity of the inner product and conditional expectation,
	\begin{equation}\label{eq:trace_to_innerproduct_projection}
		\E \left[\text{Tr}(C_f(T^{(j)}))\right] = \E\|\E[S_f(Z^{(j)}) \mid T^{(j)}]\|_2^2 = \underset{i=1}{\overset{n_j}{\sum}}  \E G^{(j)}_i.
	\end{equation}
	Define also
	\begin{equation*}
		\breve{G}^{(j)}_i = \langle \E[ S_f(Z^{(j)})| T^{(j)} ], S_f(\breve Z^{(j)}_i) \rangle,
	\end{equation*}
	where $\breve{Z}_i^{(j)}$ is and independent copy of ${Z}_i^{(j)}$ and note that $\E \breve{G}^{(j)}_i = 0$. Write $(G_i^{(j)})^+ := 0 \vee G^{(j)}_i$ and $(G_i^{(j)})^- = - 0 \wedge G^{(j)}_i$. We have that $\E (G^{(j)}_i)^+$ equals
	\begin{align*}
		   \int_0^\infty \P \left( (G_i^{(j)})^+ \geq t \right) dt = \int_0^M \P \left((G_i^{(j)})^+ \geq t \right) + \int_M^\infty \P \left( (G_i^{(j)})^+ \geq t \right) &\leq \\
		 e^{\epsilon} \int_0^M \P \left( (\breve{G}_i^{(j)})^+ \geq t \right) dt +M \delta + \int_M^\infty \P \left( (G_i^{(j)})^+ \geq t \right) &\leq \\ 
		  \int_0^\infty \P \left( (\breve{G}_i^{(j)})^+ \geq t \right)dt + C' \epsilon \int_0^\infty \P \left( (\breve{G}_i^{(j)})^+ \geq t \right)dt +M \delta + \int_M^\infty \P \left( (G_i^{(j)})^+ \geq t \right),
	\end{align*}
	where in the last inequality the constant $C'$ depends only on $C_\varepsilon$. In the same way, we obtain that $\E (G_i^{(j)})^- $ is lower bounded by
	\begin{align*}
		  e^{-\epsilon} \int_0^M \P \left( (\breve{G}^{(j)}_i)^- \geq t \right) dt - M \delta + \int_M^\infty \P \left( (G_i^{(j)})^- \geq t \right) &\geq \\ 
		  \int_0^M \P \left( (\breve{G}^{(j)}_i)^- \geq t \right)dt - C' \epsilon \int_0^\infty \P \left( (\breve{G}^{(j)}_i)^- \geq t \right)dt - M \delta + \int_M^\infty \P \left( (G_i^{(j)})^- \geq t \right) dt &= \\
		 \int_0^\infty \P \left( (\breve{G}^{(j)}_i)^- \geq t \right)dt - C' \epsilon \int_0^\infty \P \left( (\breve{G}^{(j)}_i)^- \geq t \right)dt - M \delta - \int_M^\infty \P \left( (\breve{G}^{(j)}_i)^- \geq t \right)dt.
	\end{align*}
	Putting these together with $G_i = (G_i^{(j)})^+ -( G_i^{(j)})^-$, we get
	\begin{align*}
		\E G^{(j)}_i &\leq \int_0^\infty \P \left((\breve{G}^{(j)}_i)^+ \geq t \right)dt + \int_0^\infty \P \left( (\breve{G}^{(j)}_i)^- \geq t \right)dt + 2 C' \epsilon \int_0^\infty \P \left( |\breve{G}^{(j)}_i| \geq t \right)dt \\ &+ 2 M \delta + \int_M^\infty \P \left( (G_i^{(j)})^+ \geq t \right) dt - \int_M^\infty \P \left( (\breve{G}^{(j)}_i)^- \geq t \right)dt\\
		&= \E \breve{G}^{(j)}_i + 2 C' \varepsilon_j \E | \breve{G}^{(j)}_i | + 2 M \delta_j + \int_M^\infty \P \left( (G^{(j)}_i)^+ \geq t \right) dt + \int_M^\infty \P \left( (\breve{G}^{(j)}_i)^- \geq t \right)dt.
	\end{align*}
	The first term in the last display equals zero. For the second term, observe that 
	
	\begin{align*}
		\E | \breve{G}^{(j)}_i | &\leq \sqrt{\E\left[\E \langle \E(S_f(Z^{(j)}_i) |T^{(j)}), S_f(\breve Z^{(j)}_i)\rangle^2\right]} \\
		&\leq \sqrt{\E \left[\E(S_f(Z^{(j)}_i) |T^{(j)})^T \mathrm{Var}(S_f(\breve Z^{(j)}_i))\E(S_f(Z^{(j)}_i) |T^{(j)}) \right]} \\
		&\leq  \sqrt{\E \|\E(S_f(Z^{(j)}_i) |T^{(j)})\|_2^2  } \sqrt{\lambda_{\max}(\bm C_f)}\\
		&\leq \sqrt{\E\mathrm{Tr}(\bm C_f(T^{(j)}))} \sqrt{\lambda_{\max}(\bm C_f)}.
	\end{align*}
	A concentration argument bounding the tails of $(G^{(j)}_i)^+$ and $\breve{G}^{(j)}_i$, i.e. Lemma \ref{lemma:bounding-tail-term-trace-attack} with $M = (2^L\sqrt{n_j})  \log(1/\delta_j)$, in combination with the identity \eqref{eq:trace_to_innerproduct_projection}, yields the desired result.
\end{proof}
\subsubsection{Proof of Theorem \ref{thm:minimax-lower-bound}}
\begin{proof}
The crux of the lower bound argument involves showing that 
\begin{equation}\label{eq:bound-on-expected-trace-of-conditional-score-function}
	\mathbb{E} \left[\mathrm{Tr}(\bm C_f(T^{(j)}))\right] \lesssim n_j^2\varepsilon_j^2,
\end{equation}
which we prove in Lemma \ref{lemma:proof-of-bound-on-expected-trace-of-conditional-score-function} in Section \ref{sec:proof-of-bound-on-expected-trace-of-conditional-score-function}.

It is easy to observe that $\bm C^{(j)}_{f,i} = I_{2^L}$, i.e. a $2^L \times 2^L$ identity matrix. 
Using the properties of the conditional expectations we have that for random vector $U$ and $V$
\[
\mathrm{Cov}\left[\E(U |V)\right] \preceq \mathrm{Cov}(U)
\]
where $\preceq$ denotes the positive definite inequality. This implies that $\E \bm C_f(T^{(j)}) \preceq \bm C^{(j)}_f$ using the fact that $\E\bm S_f(Z^{(j)}) = 0$. Hence we have that $\mathrm{Tr}\left[\mathbb{E} (\bm C_f(T^{(j)}))\right] \leq \mathrm{Tr}(\bm C^{(j)}_f) = n_j 2^L$.

Combining these two upper bounds on the expected trace of $\bm C_f(T^{(j)})$, we have, by Lemma \ref{lemma:van-trees-nonparametric}, that
\[
\sup_{f \in \mathcal{B}^{\alpha,R}_{p,q}} \mathbb{E}_{f} \|\hat f - f\|_2^2 \gtrsim \frac{2^{2L}}{\sum_{j=1}^m n_j^2\varepsilon^2 \wedge n_j 2^L + \pi^2 2^{L(2\alpha+2)}}
\]
Now employing our choice of $L$ we have the desired minimax lower bound.
\end{proof}

\subsubsection{Lemmas \ref{lemma:proof-of-bound-on-expected-trace-of-conditional-score-function} and \ref{lemma:bounding-tail-term-trace-attack}}
\label{sec:proof-of-bound-on-expected-trace-of-conditional-score-function}
\begin{lemma}\label{lemma:proof-of-bound-on-expected-trace-of-conditional-score-function}
Let $\bm C_f(T^{(j)}) = \mathbb{E}\left[\bm S_f(Z^{(j)})\mid T^{(j)}\right]\mathbb{E}\left[\bm S_f(Z^{(j)})\mid T^{(j)}\right]^T$ and $\bm S_f(Z^{(j)})$ is defined as in~\eqref{eq:influence_function}. Assuming $\delta_j$ is such that $\delta_j \log(1/\delta_j) \lesssim n_j^{1/2}\varepsilon_j^22^{-L}$, it holds that
	\[
	\mathbb{E} \left[\mathrm{Tr}(\bm C_f(T^{(j)}))\right] \lesssim n_j^2\varepsilon_j^2.
	\]
\end{lemma}
\begin{proof}
We wish to employ Lemma \ref{lem:trace-processing} for the proof. Using Lemma \ref{lemma:bounding-tail-term-trace-attack}, we have that
\begin{align*}
	\underset{i=1}{\overset{n}{\sum}}  \E G^{(j)}_i &\lesssim 2 n_j \varepsilon_j\sqrt{\text{Tr}(\E(\bm C_f(T^{(j)})))} + \delta_j (2^Ln_j^{1/2}) \log(1/\delta_j) + n_j \delta_j.
\end{align*}
If $\sqrt{\text{Tr}(\E(\bm C_f(T^{(j)})))} \leq n_j \varepsilon_j$, then~\eqref{eq:bound-on-expected-trace-of-conditional-score-function} holds (there is nothing to prove). So assume instead that $\sqrt{\text{Tr}(\E(\bm C_f(T^{(j)})))} \geq n_j \varepsilon_j$. Combining the above display with \eqref{eq:trace_to_innerproduct_projection}, we get
\begin{align*}
	\sqrt{\text{Tr}(\E(\bm C_f(T^{(j)})))}  &\lesssim {2 n_j \varepsilon_j} +  \delta_j \frac{2^L  n_j^{1/2}}{n_j\varepsilon_j} \log(1/\delta_j) +  \frac{1}{\varepsilon_j}  \delta_j.
\end{align*}
The assumption of small enough $\delta_j$ ($\delta_j \log(1/\delta_j)\lesssim n_j^{1/2}\varepsilon_j^22^{-L}$) implies that the last two terms are $o(n_j\varepsilon_j)$ which yields the result.
\end{proof}
\begin{lemma}\label{lemma:bounding-tail-term-trace-attack}
	For $M \asymp (2^L\sqrt{n_j})  \log(1/\delta_j)$ we have that
	\[
	2 M \delta_j + \int_M^\infty \P \left( (G^{(j)}_i)^+ \geq t \right) dt + \int_M^\infty \P \left( (\breve{G}^{(j)}_i)^- \geq t \right)dt \lesssim  \delta_j 2^L  \sqrt{n_j}\log(1/\delta_j) +   \delta_j.
	\]
	where $(G^{(j)}_i)^+$ and $(\breve{G}^{(j)}_i)^-$ is as defined in the proof of Lemma \ref{lem:trace-processing} proved in Section \ref{sec:proof-of-trace-processing}.
\end{lemma}
\begin{proof}
	 We first analyze the term concerning $G^{(j)}_i =\langle \mathbb{E}[\bm S_f(Z^{(j)})|T^{(j)}],\bm S_f(Z_i^{(j)}) \rangle$. Using Jensen's and law of total probability we have that for $i=1,\dots,n$,
	\begin{align*}
		\E e^{t |G^{(j)}_i|} &= \E e^{t |\langle \mathbb{E}[\bm S_f(Z^{(j)})|T^{(j)}],\bm S_f(Z_i^{(j)}) \rangle|} \\ &= \int \int e^{t |\langle \mathbb{E}[\bm S_f(Z^{(j)})|T^{(j)} = u],\bm S_f(Z_i^{(j)}) \rangle|} dP^{T^{(j)}|Z_i^{(j)}=z}(u) dP^{Z^{(j)}_i}(z) \\
		&\leq \int \int \E[ e^{t |\langle \bm S_f(Z^{(j)}),\bm S_f(Z_i^{(j)}) \rangle|} | T^{(j)} = u ] dP^{T^{(j)}|Z^{(j)}_i=z}(u) dP^{Z^{(j)}_i}(z) \\
		&=\E e^{t|\langle \bm S_f(Z^{(j)}),\bm S_f(Z_i^{(j)}) \rangle|}.
	\end{align*}
	Without loss of generality, consider the index $i$ to be equal to $1$. Note that we have that $\langle \bm S_f(Z^{(j)}),\bm S_f(Z_i^{(j)}) \rangle \overset{d}{=} \sum_{i=1}^{n_j} \xi_i\xi_1\langle \bm \psi(X^{(j)}_i), \bm\psi(X^{(j)}_1)\rangle$ where $\xi_i \overset{i.i.d}{\sim} N(0,1)$ with $i\in[n_j]$. Also set $U_i := \langle \bm \psi(X^{(j)}_i), \bm\psi(X^{(j)}_1)\rangle$. We would need a bound on $U_i$'s for our subsequent analysis which we show now using
	\[
	\|\psi(X_i^{(j)})\|_2^2 \leq \sum_{k=0}^{2^L-1} \psi^2_{Lk}(X_i^{(j)}) \leq c_\psi 2^L,
	\]
	for some constant $c_\psi$ where we have used the bounded support property of $\psi_{Lk}$. We set $t = \frac{1}{4c_\psi} \frac{2^{-L}}{\sqrt n_j}$.
	Using the tower property of conditional expectation we have that
	\begin{align*}
		\E e^{t|\langle \bm S_f(Z^{(j)}),\bm S_f(Z_i^{(j)}) \rangle|} &= \E e^{t\left|\sum_{i=1}^{n_j}\xi_i\xi_1U_i\right|}\\	
		&= \E \E \left[ e^{t\left|\sum_{i=1}^{n_j}\xi_i\xi_1U_i\right|} \mid \{U_i\}_{i=1}^{n_j},\xi_1\right]\\
		&\leq \E \E \left[ e^{t|\xi_1^2U_1| + t|\xi_1|\left|\sum_{i=2}^{n_j}\xi_iU_i\right|} \mid \{U_i\}_{i=1}^{n_j},\xi_1\right]\\
		&\leq \E \left[e^{t|\xi_1^2U_1|}\E \left[ e^{ t|\xi_1||\sum_{i=2}^{n_j}\xi_iU_i|} \mid \{U_i\}_{i=1}^{n_j},\xi_1\right]\right].
	\end{align*}
	Next we use the fact that $\sum_{i=2}^{n_j} \xi_iU_i \mid \{U_i\}_{i=2}^{n_j} \sim N (0, \sum_{i=2}^{n_j} U^2_i)$ and the m.g.f of gaussian distribution we have 
	\[
	\E \left[ e^{ t|\xi_1||\sum_{i=2}^{n_j}\xi_iU_i|} \mid \{U_i\}_{i=1}^{n_j},\xi_1\right] \leq 2e^{\frac{1}{2}t^2\sum_{i=2}^{n_j} U_i^2 \xi_1^2},
	\]
	using this and the fact that $|U_i| \lesssim 2^L $ we have 
	\begin{align*}
		\E e^{t|\langle \bm S_f(Z^{(j)}),\bm S_f(Z_i^{(j)}) \rangle|} 
		&\leq 2\E \left[e^{\left(t2^{L+1} + \frac{4(n_j-1)2^{2L}}{2}t^2\right) \xi_1^2}\right]\\
		&\leq 2e^{(2t2^L + 2n_j2^{2L}t^2)}\E \left[e^{(2t2^L + 2n_j2^{2L}t^2) (\xi_1^2-1)}\right]\\
		&\leq 2e^{(2t2^L + 2n_j2^{2L}t^2)} e^{2(2t2^L + 2n_j2^{2L}t^2)^2} \quad (\because (2t2^L + 2n_j2^{2L}t^2) \leq 1/4)\\
		&\leq  2e^{3(2t2^L + 2n_j2^{2L}t^2)} \quad (\because (2t2^L + 2n_j2^{2L}t^2) \leq 1/4)\\
		&\leq 2 e^{12},
	\end{align*}
	where we have used the fact that $\xi_1^2$ is chi-square distributed with $1$ degree of freedom. It follows that
	\begin{align*}
		\P(|G^{(j)}_i| > u) &\leq \P(e^{t|G^{(j)}_i| }\geq e^{tu})\\
		&\leq e^{-tu}\E e^{t|G^{(j)}_i|}\\
		&\leq 2e^6 e^{-u\frac{2^{-L}}{4\sqrt n_j} }.
	\end{align*}
	This means that for $M \gtrsim (2^L\sqrt{n_j})  \log(1/\delta_j)$, we obtain
	\begin{align*}
		\int_M^\infty \P \left( (G^{(j)}_i)^+ \geq t \right) dt \leq \int_M^\infty \P \left( |G^{(j)}_i| \geq t \right) dt \lesssim 2e^{- \log(1/\delta_j)}.
	\end{align*}
	
	Next we consider $\int_M^\infty \P \left( |\breve G^{(j)}_i| \geq t \right) dt$, which we bound with a similar argument as before. Using Jensen's inequality and the law of total probability, we obtain that
	\[
	\E e^{t |\breve G^{(j)}_i|} \leq \E e^{t|\langle \bm S_f(Z^{(j)}),\bm S_f(\breve Z_i^{(j)}) \rangle|}.
	\]
	Note that we have that $\langle \bm S_f(Z^{(j)}),\bm S_f(\breve Z_i^{(j)}) \rangle \overset{d}{=} \sum_{i=1}^{n_j} \xi_i\xi\langle \bm \psi(X^{(j)}_i), \bm\psi(\breve X^{(j)}_1)\rangle$ where $\xi, \xi_i \overset{i.i.d}{\sim} N(0,1)$ with $i\in[n_j]$.
	Define $U_i' = \langle \bm \psi(X^{(j)}_i), \bm\psi(\breve X^{(j)}_1)\rangle$, where by the earlier argument it follows that $|U_i'| \leq c_\psi 2^L$.
	We have
	\begin{align*}
		\E e^{t|\langle \bm S_f(Z^{(j)}),\bm S_f(\breve Z_i^{(j)}) \rangle|} &= \E e^{t\left|\sum_{i=1}^{n_j}\xi_i\xi U'_i\right|}\\	
		&= \E \E \left[ e^{t|\xi|\left|\sum_{i=1}^{n_j}\xi_iU'_i\right|} \mid \{U_i\}_{i=1}^{n_j},\xi_1\right]\\
		&\leq \E \left[2e^{\frac{1}{2}t^2\sum_{i=1}^{n_j} (U'_i)^2 \xi^2}\right]\\
		&\leq 2\E \left[e^{\frac{1}{32} \xi^2}\right] \leq 2 e.
	\end{align*}
	By the same reasoning, it follows that for $M \gtrsim (2^L\sqrt{n_j})  \log(1/\delta_j)$,
	\[
	\int_M^\infty \P \left( |\breve G^{(j)}_i| \geq t \right) dt \lesssim 2 e^{-\log (1/\delta_j)}.
	\]
	Putting everything together we obtain the desired result.
\end{proof}
\section{Proofs for the point wise estimation results}
In this section we outline the proofs pertaiining to the pointwise risk problem. This section is further divided into three subsections: Section \ref{sec:sensitivity-bound-for-pointwise-estimator} contains the proof of the $\ell_1$-sensitivity of the pointwise estimator and estabishes the order of magnitude of noise that needs to be added to ensure the statistic satisfies DP guarantees. Section \ref{sec:pointwise-risk-upper-bounds} outlines the proof of the results which quantify the minimax risk of our privatised pointwise estimator, and lastly Section \ref{sec:pointwise_lb_proofs} establishes the optimality of our proposed estimator by showing matching lower bounds on the minimax risk. 
\subsection{Proof of Lemma \ref{lemma:l1-sensitivity-wavelets-at-a-point}}
\label{sec:sensitivity-bound-for-pointwise-estimator}
\begin{proof}
	Consider two neighboring data sets $Z^{(j)}$ and $\tilde{Z}^{(j)}$, meaning that at most for one $i \in [n]$ it holds that
\begin{equation*}
Z^{(j)}_i = (Y_i^{(j)},X_i^{(j)}) \neq (\tilde Y_i^{(j)},\tilde X_i^{(j)}) = \tilde Z^{(j)}_i.
\end{equation*}
Define for $l \geq l_0$ the set $K_l(x_0) := \{k:\psi_{lk}(x_0) \neq 0\}$. It holds that
	\begin{align*}
		\| \hat f_{L,\tau}^{(j)}(x_0 | Z^{(j)})  -\hat f_{L,\tau}^{(j)}(x_0 | (\tilde{Z}^{(j)})) \|_1 &= \\ \left| \sum_{l =l_0}^L \underset{k=0}{\overset{2^l - 1}{\sum}} \left(\frac{1}{n_j} [Y_i^{(j)}]_\tau \psi_{lk}(X_i^{(j)}) - \frac{1}{n_j} [\tilde Y_i^{(j)}]_\tau \psi_{lk}(\tilde X_i^{(j)}) \right)\psi_{lk}(x_0)\right| &= \\
		 \left| \sum_{l =l_0}^L \sum_{k \in K_l(x_0)}\left(\frac{1}{n_j} [Y_i^{(j)}]_\tau \psi_{lk}(X_i^{(j)}) - \frac{1}{n_j} [\tilde Y_i^{(j)}]_\tau \psi_{lk}(\tilde X_i^{(j)}) \right)\psi_{lk}(x_0)\right|. &\\
		\end{align*}
Using the $L_\infty$-norm bound on $\psi_{lk}$, which states that $\|\psi_{lk}\|_\infty = 2^{l/2}\|\psi\|_\infty $ for all $l \leq L$ and $k = 0,\dots, 2^{l} -1$, the latter is bounded by 
\begin{align*}
		\left| \sum_{l =l_0}^L \sum_{k \in K_l(x_0)}\left(\frac{1}{n_j} \tau 2^{l/2+1} \|\psi\|_\infty \right)\psi_{lk}(x_0)\right|
		\leq \left| \sum_{l =l_0}^L  \sum_{k \in K_l(x_0)}\left(\frac{1}{n_j} \tau 2^{l+1} \|\psi\|^2_\infty \right)\right|
		\leq c''_\psi \frac{\tau2^L}{n_j},
	\end{align*}
	where the final step uses that $|K_l(x_0)| \leq c_A$ since the wavelet basis is compactly supported.
\end{proof}

\subsection{Proof of Theorem \ref{thm:upper-bound-pointwise-risk}}
\label{sec:pointwise-risk-upper-bounds}

\begin{proof}
	Fix $f \in \cB^{\alpha,R}_{p,q}$. We decompose the point wise risk as follows:
	\begin{align*}
		\E(\hat f(x_0) - f(x_0))^2 &= \E\left(\sum_{j=1}^m u_j T^{(j)}_{L,\tau} - f(x_0)\right)^2\\
		&\leq 2\E \left(\sum_{j=1}^m u_j \hat f^{(j)}_{L,\tau}(x_0) - f(x_0)\right)^2 + 2 \E\left(\sum_{j=1}^m u_j W^{(j)}\right)^2,
	\end{align*}
	where the first part is the non-private part and the second term accounts for the noise added due to privacy constrains.
	In Lemma \ref{lemma:proof-of-bound-on-error-due-to-privacy-pointwise-risk}, Section \ref{sec:proof-of-bound-on-error-due-to-privacy-pointwise-risk}, we bound the second term as follows:
	\begin{equation}\label{eq:bound-on-error-due-to-privacy-pointwise-risk}
		\E\left(\sum_{j=1}^m u_j W^{(j)}\right)^2 \leq  \frac{2(c''_\psi)^2 \tau^2 2^{2L}}{\sum_j n^2_j\varepsilon^2_j \wedge n_j2^L} = 2(c''_\psi)^2 \tau^2 2^{-2L\nu}.
	\end{equation}
	The first term can be decomposed as
	\begin{equation*}
		\E \left(\sum_{j=1}^m u_j \hat f^{(j)}_{L,\tau}(x_0) - f(x_0)\right)^2 \leq 2 \E \left(\sum_{j=1}^m u_j \hat f^{(j)}_{L,\tau}(x_0) - f_L(x_0)\right)^2 + 2 (f_L(x_0) - f(x_0))^2,
	\end{equation*}
	where $f_L(x_0) = \sum_{l=l_0}^L\sum_{k=0}^{2^l - 1} f_{lk} \psi_{lk}(x_0)$. Define $K_l(x_0) = \{k:\psi_{lk}(x_0) \neq 0\}$. The second term is now bounded as follows:
	\begin{align*}
		(f_L(x_0) - f(x_0))^2 &\leq \left(\sum_{l >L} \sum_{k=0}^{2^l-1} |f_{lk}| |\psi_{lk}(x_0)|\right)^2\\
		&\leq  \left(\sum_{l >L} \sum_{k\in K_l(x_0)} c_* 2^{-l\left(\alpha + \frac{1}{2}-\frac{1}{p}\right)}2^{l/2}\|\psi\|_\infty \right)^2\\
		&\leq \left(c \sum_{l >L} 2^{-l\left(\alpha - \frac{1}{p}\right)} \right)^2\\
		&\leq c_\psi 2^{-2L\nu} \quad (\text{recalling that }\nu = \alpha-1/p),
	\end{align*}
	where we have used the fact that $|f_{lk}| \leq c_* 2^{-l\left(\alpha + \frac{1}{2}-\frac{1}{p}\right)}$ for all $l,k$ and $\|\psi_{lk}\|_\infty \leq 2^{l/2} \|\psi\|_\infty$.
	Next we focus on the first term which we bound as follows:
	\begin{align*}
		\E \left(\sum_{j=1}^m u_j \hat f^{(j)}_{L,\tau}(x_0) - f_L(x_0)\right)^2 &\leq \\
		 \E \left(\sum_{j=1}^mu_j \sum_{l=l_0}^L \sum_{k=0}^{2^l -1} (\hat f^{(j)}_{lk;\tau} - f_{lk}) \psi_{lk}(x_0)\right)^2 &\leq \\
		 2\E\left( \sum_{j=1}^mu_j \sum_{l=l_0}^L \sum_{k=0}^{2^l -1} \left(\frac{1}{n_j} \sum_{i=1}^{n_j} f(X^{(j)}_i)\psi_{lk}(X_i^{(j)})- f_{lk}\right) \psi_{lk}(x_0)\right)^2&	\\
		+2\E\left( \sum_{j=1}^mu_j \sum_{l=l_0}^L \sum_{k=0}^{2^l -1} \left(\frac{1}{n_j} \sum_{i=1}^{n_j} \eta^{(j)}_i\psi_{lk}(X_i^{(j)})\right) \psi_{lk}(x_0)\right)^2,&
	\end{align*}
	where $\eta_i^{(j)} = [\xi_i^{(j)}]^{\tau-f(X_i^{(j)})}_{-\tau-f(X_i^{(j)})}$.
	
We upper bound the first term (see Lemma \ref{lemma:proof-of-bounding-non-private-part-1-pointwise-risk} below) as follows:
\begin{align}\label{eq:bounding-non-private-part-1-pointwise-risk}
\E\left( \sum_{j=1}^mu_j \sum_{l=l_0}^L \sum_{k=0}^{2^l -1} \left(\frac{1}{n_j} \sum_{i=1}^{n_j} f(X^{(j)}_i)\psi_{lk}(X_i^{(j)})- f_{lk}\right) \psi_{lk}(x_0)\right)^2 \leq  c_\psi R^2 2^L\sum_{j=1}^m \frac{u_j^2}{n_j}	.
\end{align}

For the remaining term, write
\begin{align*}
	\E&\left( \sum_{j=1}^mu_j \sum_{l=l_0}^L \sum_{k=0}^{2^l -1} \left(\frac{1}{n_j} \sum_{i=1}^{n_j} \eta^{(j)}_i\psi_{lk}(X_i^{(j)})\right) \psi_{lk}(x_0)\right)^2\\
	&= \E\left( \sum_{j=1}^m \sum_{l=l_0}^L \sum_{k=0}^{2^l -1}  \sum_{i=1}^{n_j} \frac{u^2_j}{n^2_j} (\eta^{(j)}_i)^2\psi^2_{lk}(X_i^{(j)}) \psi^2_{lk}(x_0)\right)\\
	&+ \E\left(  \sum_{(i,j,l,k)\neq (i',j',l',k')}  \frac{u_j}{n_j} \frac{u_{j'}}{n_{j'}}\eta^{(j)}_i \eta^{(j')}_{i'}\psi_{lk}(X_i^{(j)}) \psi_{l'k'}(X_{i'}^{(j')})\psi_{lk}(x_0)\psi_{l'k'}(x_0)\right) .
\end{align*}
Here, the first term can be written as
\begin{align*}
	\sum_{j=1}^m \sum_{l=l_0}^L \sum_{k\in K_l(x_0)}  \sum_{i=1}^{n_j} \frac{u^2_j}{n^2_j}  \psi^2_{lk}(x_0) \E(\eta^{(j)}_i)^2\psi^2_{lk}(X_i^{(j)}) &\leq c \sum_{j=1}^{m}\sum_{i=1}^{n_j} \sum_{l=l_0}^L 2^l \frac{u_j^2}{n_j^2}\\
	&\leq c_2 2^L \sum_{j=1}^m \frac{u_j^2}{n_j},
\end{align*}
for a unversival constant $c_2 > 0$, where we have used the fact that $\E(\eta^{(j)}_i)^2\psi^2_{lk}(X_i^{(j)})$ is bounded above by a constant, since
\begin{align*}
\E \underset{x \in [-C_{\alpha,R},C_{\alpha,R}]}{\sup} \left([\xi_i^{(j)}]^{\tau-x}_{-\tau-x}\right)^2 \psi^2_{lk}(X_i^{(j)}) &= \E \underset{x \in [-C_{\alpha,R},C_{\alpha,R}]}{\sup} \left([\xi_i^{(j)}]^{\tau-x}_{-\tau-x}\right))^2 \E \psi^2_{lk}(X_i^{(j)})  \\ 
&= \E \underset{x \in [-C_{\alpha,R},C_{\alpha,R}]}{\sup} \left([\xi_i^{(j)}]^{\tau-x}_{-\tau-x}\right)^2 = O(1),
\end{align*}
by similar arguments as in Lemma \ref{lemma:proof-of-bounding-non-private-part-3}. We show in Lemma \ref{lemma:proof-of-bounding-non-private-part-2-pointwise-risk} below that the second term can be bounded above by
\begin{align}\label{eq:bounding-non-private-part-2-pointwise-risk}
	\sum_{(i,j,l,k)\neq (i',j',l',k')} \frac{u_j}{n_j}\frac{u_{j'}}{n_{j'}}\psi_{lk}(x_0)\psi_{l'k'}(x_0) \E \eta^{(j)}_i \eta^{(j')}_{i'}\psi_{lk}(X_i^{(j)}) \psi_{l'k'}(X_{i'}^{(j')})
	\leq cR^2 2^{-2\nu L}.
\end{align}
Hence we have that 
\begin{align*}
	\E \left(\sum_{j=1}^m u_j \hat f^{(j)}_{L,\tau}(x_0) - f_L(x_0)\right)^2 &\leq c'R^2 2^L \sum_{j=1}^m \frac{u_j^2}{n_j} + c'R^22^{-2L\nu}\\
	&= c'R^2\frac{2^L}{\sum v_j} \sum_j u_j \frac{v_j }{n_j} + c'R^22^{-2L\nu}\\
	&\leq c'R^2\frac{2^L}{\sum v_j} \sum_j u_j \frac{n_j 2^L }{n_j} +c'R^22^{-2L\nu}\\
	&= c'R^2\frac{2^{2L}}{\sum_j n_j^2\varepsilon_j^2 \wedge n_j 2^L} + c'R^22^{-2L\nu} = c'R^2 2^{-2L\nu},
\end{align*}	
where we have used the fact that $v_j \leq n_j 2^L$. Using similar ideas as in the global risk proof we have that $\tau^2 = O(\log N)$. Combining the above bounds, we conclude that 
\begin{equation*}
 \E_{f} (\hat f(x_0) - f(x_0))^2 \leq  c_\psi \log (N)  2^{-2L\nu}.
\end{equation*}
\end{proof}
\subsubsection{Auxiliary Lemmas \ref{lemma:proof-of-bound-on-error-due-to-privacy-pointwise-risk}, \ref{lemma:proof-of-bounding-non-private-part-1-pointwise-risk}, \ref{lemma:proof-of-bounding-non-private-part-2-pointwise-risk}}\label{sec:proof-of-bound-on-error-due-to-privacy-pointwise-risk}

\begin{lemma}\label{lemma:proof-of-bound-on-error-due-to-privacy-pointwise-risk}
	It holds that
	\[
	\E\left(\sum_{j=1}^m u_j W^{(j)}\right)^2 \leq  \frac{2(c''_\psi)^2 \tau^2 2^{2L}}{\sum_j n^2_j\varepsilon^2_j \wedge n_j2^L},
	\]
	where $W^{(j)}\overset{i.i.d}{\sim} \mathrm{Lap}\left(0,\frac{c''_\psi \tau 2^L}{n_j \varepsilon_j}\right) $ and $u_j$ is as defined in~\eqref{eq:estimator_weights_global}.
\end{lemma}
\begin{proof}
Using the independence of the noise generating mechanisms across the servers and the fact that the $W^{(j)}$'s are centered, we have:
\begin{align*}
	\E\left[\left(\sum_{j=1}^m  u_jW^{(j)}\right)^2\right] &=\sum_{i=1}^m u^2_j\E\left[(W^{(j)})^2 \right]\\
	&\leq 2 \sum_{j=1}^m u_j^2 \frac{(c''_\psi)^2 \tau^2 2^{2L}}{n_j^2\varepsilon_j^2}.
\end{align*}
Next we use that fact that $v_j \leq n_j^2\varepsilon_j^2$ to obtain 
\begin{align*}
	\E\left[\left(\sum_{j=1}^m  u_jW^{(j)}\right)^2\right] 
	&\leq \frac{2}{\sum_j n^2_j\varepsilon^2_j \wedge n_j2^L}\sum_{j=1}^m u_j v_j\frac{(c''_\psi)^2 \tau^2 2^{2L}}{n_j^2\varepsilon_j^2}\\
	&\leq \frac{2}{\sum_j n^2_j\varepsilon^2_j \wedge n_j2^{L}}\sum_{j=1}^m u_j(c''_\psi)^2 \tau^2 2^{2L}\\
	&\leq \frac{2(c''_\psi)^2 \tau^2 2^{2L}}{\sum_j n^2_j\varepsilon^2_j \wedge n_j2^L}.
\end{align*}
\end{proof}


\begin{lemma}\label{lemma:proof-of-bounding-non-private-part-1-pointwise-risk}
	It holds that
	\[
	\E\left( \sum_{j=1}^mu_j \sum_{l=l_0}^L \sum_{k=0}^{2^l -1} \left(\frac{1}{n_j} \sum_{i=1}^{n_j} f(X^{(j)}_i)\psi_{lk}(X_i^{(j)})- f_{lk}\right) \psi_{lk}(x_0)\right)^2 \leq  c_\psi R^2 2^L\sum_{j=1}^m \frac{u_j^2}{n_j}
	\]
	where $f_{lk} = \langle f,\psi_{lk}\rangle$  are the wavelet basis coefficents of the function $f \in B^\alpha_{p,q}(R)$  and $u_j$ is as defined in~\eqref{eq:estimator_weights_global}.
\end{lemma}

\begin{proof}
Using the fact that we only need to consider those wavelet basis for which $\psi_{lk}(x_0)$ is non-zero we have that
	\begin{align*}
	\sum_{j=1}^m \sum_{l=l_0}^L \sum_{k=0}^{2^l -1} \sum_{i=1}^{n_j} \psi^2_{lk}(x_0)\frac{u_j^2}{n_j^2} \mathrm{Var}\left(f(U)\psi_{lk}(U)\right)
	&\leq \sum_{j=1}^m \sum_{l=l_0}^L \sum_{k\in K_l(x_0)} \sum_{i=1}^{n_j} \psi^2_{lk}(x_0)\frac{u_j^2}{n_j^2} \mathrm{Var}\left(f(U)\psi_{lk}(U)\right)
\end{align*}
Next using the fact proved in Lemma \ref{lemma:proof-of-bounding-non-private-part-1} we have that 
\[
\mathrm{Var}\left(f(U)\psi_{lk}(U)\right) \leq C_{R,\alpha}^2
\]
for a universal constant $C_{R,\alpha}>0$ and that $\psi_{lk} \leq 2^{l/2} \|\psi\|_\infty$ we have that the quantity of interest is bounded above by 
\begin{align*}
	 c^{(1)}_\psi R^2  \sum_{j=1}^m \sum_{l=l_0}^L  \sum_{i=1}^{n_j} \frac{u_j^2}{n_j^2} 2^l
			\leq c_\psi R^2 2^L\sum_{j=1}^m \frac{u_j^2}{n_j}.
\end{align*}
\end{proof}

\begin{lemma}\label{lemma:proof-of-bounding-non-private-part-2-pointwise-risk}
	It holds that
	\[
		\sum_{(i,j,l,k)\neq (i',j',l',k')} \frac{u_j}{n_j}\frac{u_{j'}}{n_{j'}}\psi_{lk}(x_0)\psi_{l'k'}(x_0) \E \eta^{(j)}_i \eta^{(j')}_{i'}\psi_{lk}(X_i^{(j)}) \psi_{l'k'}(X_{i'}^{(j')})
	\leq cR^2 2^{-2\nu L}
	\]
	where $\eta^{(j)}_i := \left[ \xi_i^{(j)}\right]^{\tau-f(X^{(j)}_i)}_{-\tau-f(X^{(j)}_i)}$, $\tau = C_{\nu,R} + \sqrt{2(2\nu + 1)L}$ and $u_j$ is as defined in~\eqref{eq:estimator_weights_global}.
\end{lemma}
\begin{proof}
By the same argument as in the proof of Lemma \ref{lemma:proof-of-bounding-non-private-part-3}, we have that
\begin{align*}
	 \E \eta^{(j)}_i \eta^{(j')}_{i'}\psi_{lk}(X_i^{(j)}) \psi_{l'k'}(X_{i'}^{(j')}) \leq 2^{l/2}2^{l'/2} \|\psi\|_\infty^2 R^2 e^{-(\tau-C_{\nu,R})^2} \leq 2^L \|\psi\|_\infty^2 R^2 e^{-(\tau-C_{\nu,R})^2},
\end{align*}
whenever $(i,j) \neq (i',j')$. Assume now that $(i,j) = (i',j')$, but $(l,k) \neq (l',k')$.

Define 
\begin{equation*}
	A = \left\{ - \tau - f(x) \leq \xi_i^{(j)} \leq \tau - f(x) \; \forall x \in [0,1]\right\}.
\end{equation*}
We have 
\begin{align*}
	 \E \mathbbm{1}_A (\eta^{(j)}_i)^2 \psi_{lk}(X_i^{(j)}) \psi_{l'k'}(X_{i'}^{(j')}) &= \E^\xi \E^{X} \mathbbm{1}_A (\xi^{(j)}_i)^2 \E^{X} \psi_{lk}(X_i^{(j)}) \psi_{l'k'}(X_{i'}^{(j')}) = 0,
\end{align*}
since $\psi_{lk} \perp \psi_{l'k'}$ for $(l,k) \neq (l',k')$. Furthermore, using Lemma \ref{lem:folklore_clipping_reduces_variance} and Cauchy-Schwarz, we have that
\begin{align*}
	\E \mathbbm{1}_{A^c} (\eta^{(j)}_i)^2 \psi_{lk}(X_i^{(j)}) \psi_{l'k'}(X_{i'}^{(j')}) &\leq \sqrt{ \E (\xi^{(j)}_i)^4  \E \mathbbm{1}_{A^c}  \psi^2_{lk}(X_i^{(j)}) \psi^2_{l'k'}(X_{i'}^{(j')})}.
\end{align*}
By the independence of $X_i^{(j)}$ and $\eta^{(j)}_i$ and by similar arguments as in the proof of Lemma \ref{lemma:proof-of-bounding-non-private-part-3}, we have that
\begin{align*}
	\E \mathbbm{1}_{A^c}  \psi^2_{lk}(X_i^{(j)}) \psi^2_{l'k'}(X_{i'}^{(j')}) \leq 2^{2L} \|\psi\|_\infty^2 e^{-(\tau-C_{\nu,R})^2/2}.
\end{align*}

Consequently, we have
\begin{align*}
	\sum_{(i,j,l,k)\neq (i',j',l',k')}& \frac{u_j}{n_j}\frac{u_{j'}}{n_{j'}}\psi_{lk}(x_0)\psi_{l'k'}(x_0) \E \eta^{(j)}_i \eta^{(j')}_{i'}\psi_{lk}(X_i^{(j)}) \psi_{l'k'}(X_{i'}^{(j')})\\
	&\leq c 2^L R^2 e^{-(\tau-C_{\nu,R})^2/2} \left(\sum_{j\neq j'} \sum_{i\neq i'} \frac{u_j}{n_j}\frac{u_{j'}}{n_{j'}}\right)\\
	&\leq c 2^L R^2 e^{-(\tau-C_{\nu,R})^2/2} \left(\sum_{j} \sum_{i} \frac{u_j}{n_j}\right)^2\\
	&= c2^LR^2 e^{-(\tau-C_{\nu,R})^2/2}
	\leq cR^2 2^{-2\nu L},
\end{align*}
where on the last line we have used our prescribed choice of $\tau$.
\end{proof}
	
\subsection{Lower bound proofs for the pointwise risk}\label{sec:pointwise_lb_proofs}

Before providing the proof of Theorem \ref{thm:lower_bound-pointwise-risk} at the end of this, we set the stage by deriving a few general results for $(\varepsilon,\delta)$-DP random variables and derive a few required auxiliary lemmas.

Consider a probability space $(\Omega, \cA, \P)$, all the random variables we discuss later are defined on this same space.
Let us first define the support of a random variable $X$ as
\[
\text{Supp}(X) := \{S \in \cA\,:\,\P(X \in S) >0\}.
\]
Let us define the \emph{max-divergence between $Y$ and $Z$} as 
\[
D_\infty (Y\| Z) := \sup_{S \subset \mathrm{Supp}(Y)} \log\left[\frac{\P(Y\in S)}{\P(Z\in S)}\right],
\]
where $c/0 := \infty$ for $c > 0$. The \emph{$\delta$-approximate-max-divergence between $Y$ and $Z$} is defined to be:
\[
D^\delta_\infty (Y\| Z) := \sup_{S \subset \mathrm{Supp}(Y):\P(Y\in S)>\delta} \log\left[\frac{\P(Y\in S)-\delta}{\P(Z\in S)}\right].
\]

The following lemma follows a well known construction (see e.g. \cite{dwork2014algorithmic}) for discrete sample spaces, which we adapt to accomodate the non-discrete case also.

\begin{lemma} \label{lemma:delta-approx-max-to-max-divergence}
	Let $Y$ and $Z$ be such that $\P^Y \ll \P^Z$, $D_\infty^\delta(Y\|Z) \leq \varepsilon$ and $D_\infty^\delta(Z\|Y) \leq \varepsilon$. Then, there exists random variables $Y',Z'$ such that
	\begin{equation}\label{eq:to_show_approx-max-to-max-divergence}
	D_{\mathrm{TV}}(Y,Y') \leq \delta, D_{\mathrm{TV}}	(Z,Z') \leq \delta \text{ and } D_\infty(Y'\|Z') \leq \varepsilon \text{, } D_\infty(Z'\|Y') \leq \varepsilon.
	\end{equation}
\end{lemma}
\begin{proof}
	Define the set $S$ as 
	\[
	S := \left\{ y \, : \, \frac{d\P^Y}{d\P^Z}(y) > e^\varepsilon \right\}.
	\]
	We define the measures $\P^{Y'}$ and $\P^{Z'}$  as follows. On the set $S$ define 
	\[
	\frac{d\P^{Z'}}{d\P^Z}(y) = \frac{1}{1 + e^\varepsilon} \left(\frac{d\P^Y}{d\P^Z}(y) + 1\right),
	\]
	and $\frac{d\P^{Y'}}{d\P^{Z'}}(y) = e^\varepsilon$. Similarly, on the set $S':= \left\{ y \, : \, \frac{d\P^Y}{d\P^Z}(y) \leq e^{-\varepsilon} \right\}$, define 
	\[
	\frac{d\P^{Z'}}{d\P^Z}(y) = \frac{e^\varepsilon}{1 + e^\varepsilon} \left(\frac{d\P^Y}{d\P^Z}(y) + 1\right),
	\]
	and $\frac{d\P^{Y'}}{d\P^{Z'}}(y) = e^{-\varepsilon}$. For $y \in (S \cup S')^c$, set $\frac{d\P^{Z'}}{d\P^Z}(y) = 1$ and $\frac{d\P^{Y'}}{d\P^{Y}}(y) = 1$. As the notation suggests, we then define
\begin{equation*}
\P^{Z'}(A) = \int_A \frac{d\P^{Z'}}{d\P^Z} d\P^Z,
\end{equation*}
such that $\frac{d\P^{Z'}}{d\P^Z}$ is the Radon-Nikodym derivative for the corresponding measure $\P^{Z'}$. Similarly, $\frac{d\P^{Y'}}{d\P^{Z'}}$ defines the Radon-Nikodym derivative with respect to $\P^{Z'}$ for the corresponding measure $\P^{Y'}$. 
	
Roughly speaking, this construction creates a ``version'' $\P^{Y'}$ of $\P^Y$, with altered mass on (measurable subsets of) $S$ and $S'$, where we have decreased the mass of $\P^Y$ on $S$ by $\varrho := \P^Y(S) - \P^{Y'}(S)$ and increased it by $\varrho' := \P^{Y'}(S') - \P^{Y}(S')$ on $S'$. We have that
\begin{align}
	\varrho &= \P^Y(S) - \P^{Y'}(S) \nonumber \\
	&= \P^Y(S) - \int_S \frac{d\P^{Y'}}{d \P^{Z}} d \P^{Z} \nonumber\\
	&= \P^Y(S) - \int_S \frac{e^\varepsilon}{1 + e^\varepsilon}\left(\frac{d\P^{Y}}{d \P^{Z}} + 1\right) d \P^{Z}\nonumber\\
	&= \frac{\P^Y(S) - e^\varepsilon\P^Z(S)}{1 + e^\varepsilon} \label{eq:rho_is_equal_to_frac},
\end{align} 
Next, observe that
\begin{align*}
\P^{Z'}(S) - \P^{Z}(S) &=  \int_S \left(\frac{d\P^{Z'}}{d \P^Z} - 1\right)d\P^Z\\
&=\int_S \left(\frac{d\P^Y}{d\P^Z}\frac{1}{1+e^\varepsilon} - \frac{e^\varepsilon}{1+e^\varepsilon}\right)d\P^Z\\
&= \frac{1}{1 + e^\varepsilon} \left(\int_S d\P^Y - e^\varepsilon d\P^Z\right)\\
&= \varrho,
\end{align*}
where the last line follows from the display above. Similarly, it holds that
\begin{align*}
\P^{Z'}(S') - \P^{Z}(S') = \varrho'.  
\end{align*}
If $\varrho=\varrho'$, it implies that $\P^{Z'}$ and $\P^{Y'}$ are probability measures and one can take $Z'$ and $Y'$ to be the corresponding random variables in the statement of the lemma, as they satisfy \eqref{eq:to_show_approx-max-to-max-divergence}. Here, $D_{\mathrm{TV}}(Y,Y') \leq \delta$ and $D_{\mathrm{TV}}	(Z,Z') \leq \delta$ follows from 
\begin{align*}
	D_{\mathrm{TV}}(Z,Z') &= \int \left| \frac{d\P^{Z'}}{d\P^{Z}} - 1 \right| d\P^Z \\
	&= \int_{S^c} \frac{e^\varepsilon}{1+e^\varepsilon} \left( e^{-\varepsilon} - \frac{d\P^Y}{d\P^Z} \right)d\P^Z + \int_S \frac{1}{1+e^\varepsilon} \left(\frac{d\P^Y}{d\P^Z} - e^\varepsilon\right)d\P^Z\\ 
	&\leq \delta,
\end{align*}
and $D_{\mathrm{TV}}	(Y,Y') \leq \delta$ follows similarly. Suppose $\varrho \neq \varrho'$, say $\varrho > \varrho'$, $\P^{Y'}$ and $\P^{Z'}$ are not probability measures. In what follows next, loosely speaking, we will construct probability measures $\P^{Z''}$ and $\P^{Y''}$ by appropriately increasing (respectively decreasing) the mass of $\P^{Y'}$ and $\P^{Z'}$ on $S$ and $S'$, by a total of $\beta := \varrho - \varrho'$.

	Define the set
	\[
	R := \left\{ y \, : \, \frac{d\P^{Y'}}{d\P^{Z'}}(y) < 1 \right\}.
	\]
	For $y \in R$, define
	\[
	\frac{d\P^{Y''}}{d\P^{Z'}}(y) = \lambda_1 + (1-\lambda_1) \frac{d\P^{Y'}}{d\P^{Z'}}(y),
	\]
	and
	\[
	\frac{d\P^{Z''}}{d\P^{Z'}}(y) = \lambda_2 + (1-\lambda_2) \frac{d\P^{Y'}}{d\P^{Z'}}(y),
	\]
	where $\lambda_1 = \beta (\P^{Z'}(R) - \P^{Y'}(R))^{-1} $ and $\lambda_2  = 1 - \lambda_1$. For $y \in R^c$, $\frac{d\P^{Y''}}{d\P^{Z'}}(y) =\frac{d\P^{Y'}}{d\P^{Z'}}(y)$ and $\frac{d\P^{Z''}}{d\P^{Z'}}(y) = 1$. 
	Observe that
	\[
	\P^{Z'}(R) - \P^{Y'}(R) \geq \P^{Z'}(\Omega) - \P^{Y'}(\Omega) = 2\beta > 0,
	\]
	which means that $\lambda_1$ is well defined and $\lambda_2 \geq \frac{1}{2} \geq\lambda_1$. $\P^{Y''}$ and $\P^{Z''}$ are probability measures:
	\begin{align*}
	\P^{Y''}(\Omega) &=\P^{Y''}(R) + \P^{Y''}(R^c)\\
	&= \int_R \left(\lambda_1 + (1-\lambda_1) \frac{d\P^{Y'}}{d\P^{Z'}}(y)\right)d\P^{Z'} + \P^{Y'}(R^c)\\
	&= \P^{Y'}(R) + \lambda_1(\P^{Z'}(R) - \P^{Y'}(R)) +\P^{Y'}(R^c)\\
	&= \beta + \P^{Y'}(\Omega)  = \P^Y(\Omega) = 1.
	\end{align*}
	Similarly, it can be shown that
	\begin{align*}
	\P^{Z''}(\Omega) = 1.
	\end{align*}
	 Let $Z''$ and $Y''$ be the random variables corresponding to $\P^{Z''}$ and $\P^{Y''}$. By construction, $\frac{d\P^{Y'}}{d\P^{Z'}}(y) \in [e^{-\varepsilon},e^{\varepsilon}]$ for $y \in S\cup S'$. Next, in order to obtain $Y''$ and $Z''$, we have modified the density on the set $R$. Hence on $R^c$ we have that $\frac{d\P^{Y''}}{d\P^{Z''}} = \frac{d\P^{Y'}}{d\P^{Z'}} \in [e^{-\varepsilon},e^\varepsilon]$. On $R$,
	 \begin{align*}
	 	\frac{d\P^{Y''}}{d\P^{Z''}} = \frac{\lambda_1 + (1-\lambda_1) \frac{d\P^{Y'}}{d\P^{Z'}}}{\lambda_2 + (1-\lambda_2) \frac{d\P^{Y'}}{d\P^{Z'}}} \leq 1,
	 \end{align*}
	 where the last inequality follows from the fact that $\lambda_2 > \lambda_1$ and $\frac{d\P^{Y'}}{d\P^{Z'}} < 1$ on $R$.
	 This also implies $\lambda_2 + (1-\lambda_2) \frac{d\P^{Y'}}{d\P^{Z'}} (y) \leq 1$ and
	 \begin{align*}
	 	\frac{d\P^{Y''}}{d\P^{Z''}}(y) = \frac{\lambda_1 + (1-\lambda_1) \frac{d\P^{Y'}}{d\P^{Z'}}(y)}{\lambda_2 + (1-\lambda_2) \frac{d\P^{Y'}}{d\P^{Z'}}(y)} \geq \lambda_1 + (1-\lambda_1) \frac{d\P^{Y'}}{d\P^{Z'}(y)} \geq  \frac{d\P^{Y'}}{d\P^{Z'}(y)} \geq e^{-\varepsilon}
	 \end{align*}
	 for $y \in R$.

	 Hence it holds that $e^{-\varepsilon} \leq \frac{dP^{Y''}}{d\P^{Z''}} \leq  e^\varepsilon$, which implies $D_\infty(Y''\|Z'') \leq \varepsilon$ and $D_\infty(Z''\|Y'') \leq \varepsilon$.
	
	It is left to show that $ D_{\mathrm{TV}}(Y,Y'') \leq \delta, D_{\mathrm{TV}} (Z,Z'')\leq \delta $. We have that
	\begin{align*}
		D_{\mathrm{TV}}(Z,Z'') &= \frac{1}{2} \int \left| \frac{d\P^{Z''}}{d\P^{Z}} - 1 \right| d\P^Z \\
		&\leq  \frac{1}{2} \int \left| \frac{d\P^{Z'}}{d\P^{Z}} - 1 \right| d\P^Z + \frac{1}{2} \int \left| \frac{d\P^{Z''}}{d\P^{Z}} - \frac{d\P^{Z'}}{d\P^{Z}} \right| d\P^{Z} \\
		&= \frac{1}{2} \int \left| \frac{d\P^{Z'}}{d\P^{Z}} - 1 \right| d\P^Z + \frac{1}{2} \int \left| \frac{d\P^{Z''}}{d\P^{Z'}} - 1 \right| d\P^{Z'}.
	\end{align*}
	By definition of $\frac{d\P^{Z'}}{d\P^{Z}}$, $ \frac{d\P^{Z''}}{d\P^{Z'}}$, $S$, $S^c$ and $R$, the latter display equals
	\begin{align*}
		 \frac{1}{2} \int_{S^c} \frac{e^\varepsilon}{1+e^\varepsilon} \left( e^{-\varepsilon} - \frac{d\P^Y}{d\P^Z} \right)d\P^Z + \frac{1}{2}\int_S \frac{1}{1+e^\varepsilon} \left(\frac{d\P^Y}{d\P^Z} - e^\varepsilon\right)d\P^Z \quad& \\ 
		 + \frac{1}{2} \int_{R} (1-\lambda_2)\left(1-\frac{d\P^{Y'}}{d\P^{Z'}} \right) d\P^{Z'} &=\\
		 \frac{1}{2(1+e^\varepsilon)} (\P^Z(S^c) - e^\varepsilon\P^Y(S^c))+ \frac{1}{2(1+e^\varepsilon)}(\P^Y(S) - e^\varepsilon \P^Z(S))& \\ + \frac{1}{2} \lambda_1 (\P^{Z'}(R) - \P^{Y'}(R)).\quad &
	\end{align*}
	By the fact that $D_\infty^\delta(Y\|Z) \leq \varepsilon$ and $D_\infty^\delta(Z\|Y) \leq \varepsilon$, the first two terms are each bounded above by $\delta/4$. The third term equals $\beta/2$. We obtain that the above display is bounded by $\frac{1}{2}(\delta + \beta) \leq \delta$, where we have used the fact that $\beta \leq \varrho \leq \delta/2$, which is implied by the fact that $\varrho \leq \frac{\delta}{1 + e^\varepsilon}$ following \eqref{eq:rho_is_equal_to_frac}. By the same steps, it follows also that $D_{\mathrm{TV}}(Y,Y'') \leq \delta$.
\end{proof}

Next lemma is well known; it relates the max-divergence to the KL-divergence.

\begin{lemma}\label{lemma:max-divergence-bounded-by-kl-divergence}
	If $D_\infty(Y,Z) \leq \varepsilon$ then $D_{\mathrm{KL}}(Y,Z) \leq \varepsilon(e^\varepsilon-1)$.
\end{lemma} 
\begin{proof}
	We know that for any $Y$ and $Z$ it is the case that $D_{\mathrm{KL}}(Y,Z) \geq 0$ and so it suffices to bound $D_{\mathrm{KL}}(Y,Z) + D_{\mathrm{KL}}(Z,Y)$.
	\begin{align*}
		D_{\mathrm{KL}}(Y,Z) &\leq D_{\mathrm{KL}}(Y,Z) + D_{\mathrm{KL}}(Z,Y)\\
		&= \int \log \left( \frac{d \P^Y}{d \P^Z}\right) d\P^Z + \int \log \left( \frac{d \P^Z}{d \P^Y}\right) d\P^Y\\
		&= \int \left[ \log \left( \frac{d \P^Y}{d \P^Z}\right)  + \log \left(\frac{d \P^Z}{d \P^Y}\right) \right]d\P^Z + \int \log \left( \frac{d \P^Z}{d \P^Y}\right) d\P^Y -  \int \log \left( \frac{d \P^Z}{d \P^Y}\right) d\P^Z\\
		&= \int \log \left( \frac{d \P^Z}{d \P^Y}\right) d\P^Y -  \int \log \left( \frac{d \P^Z}{d \P^Y}\right) d\P^Z\\
		&= \int \log \left( \frac{d\P^Z}{d\P^Y} \right) \left( \frac{d\P^Y}{d\P^Z} - 1\right) d\P^Z \leq \int \left|\log \left( \frac{d\P^Z}{d\P^Y} \right) \left( \frac{d\P^Y}{d\P^Z} - 1\right)\right| d\P^Z.
	\end{align*}
By the definition of 
it follows immediately that $\varepsilon \E \left| \frac{d\P^Y}{d\P^Z} - 1\right|$ and $\varepsilon(e^\varepsilon - 1)$, where we have used the fact that $\frac{d \P^Y}{d \P^Z}$ and $\frac{d \P^Y}{d \P^Z}$ are both less than $e^\varepsilon$ which follows from the fact that  $D_\infty(Y,Z) \leq \varepsilon$. 
\end{proof}

We recall the distributed differential privacy setup as laid out in the introduction. Consider $N = \sum_{j=1}^m n_j$ observations sampled i.i.d. from $P_0$ or $P_1$. For $j \in [m]$ (where $m$ is the number of servers) we have $n_j$ observations on server $j$. Each server generates an $(\varepsilon_j,\delta_j)$-DP transcript $T^j$ for $j \in[m]$, the vector of private transcripts is given by $T = (T^1, \ldots,T^m)$. Let us denote the two joint probability distributions, for both the data and transcripts as $\P_0$ and $\P_1$, whenever the data is generated by i.i.d. draws of $P_0$ or $P_1$, respectively. For $ i =0,1$ and $j \in [m]$ let $\P^{T^j}_i$ and $\P^{T}_i$ be the push forward of the transcript $T^j$ and $T$ respectively when data is generated under $\P_i$ . Since given the data generating distribution we have independent protocols, we have that $\P^T_i = \P^{T^1}_i \times \ldots \times \P^{T^m}_i$ for $ i =0,1$. 

We are interested in bounding the total variation distance between $\P^T_0 $ and $\P^T_1$. In that direction we first bound the $\delta$-approximate-max-divergence between $\P^{T^j}_0$ and $\P^{T^j}_1$ for $j \in [m]$ in Lemma \ref{lemma:impure-DP-to-approx-max-divergence} which is a direct corollary of Lemma 6.1 in \cite{karwa2017finite}.

\begin{lemma}\label{lemma:impure-DP-to-approx-max-divergence}
	We have that for any $j \in [m]$
	\[
	D_\infty^{\bar{\delta}_j}(\P^{T^j}_0,\P^{T^j}_1) \leq \bar{\varepsilon}_j,
	\]
	where $\bar{\varepsilon}_j= 6n_j\varepsilon_j D_{\mathrm{TV}}(P_0,P_1)$ and $\bar{\delta}_j= e^{\bar{\varepsilon}_j} n_j\delta_j D_{\mathrm{TV}}(P_0,P_1)$.
\end{lemma}

Now we are ready to prove Lemma \ref{lemma:total-variation-two-point} below, which gives Lemma \ref{lemma:total-variation-two-point_maintext} as a corollary. 

\begin{lemma}\label{lemma:total-variation-two-point}
	Fix a set $S \subseteq [m]$ we have that
	\begin{equation*}
		D_{\mathrm{TV}}(\P^T_0, \P^T_1) \leq \sqrt 2\sqrt{\sum_{j \in S}\bar{\varepsilon}_j (e^{\bar{\varepsilon}_j} -1) + \sum_{j \in S^c}n_jD_{\mathrm{KL}}(P_0,P_1)} + 4 \sum_{j\in S} \bar{\delta}_j ,
	\end{equation*}
	where $\bar{\varepsilon}_j = 6n_j\varepsilon_j D_{\mathrm{TV}}(P_0,P_1)$ and $\bar{\delta}_j = e^{\bar{\varepsilon}_j} n_j\delta_j D_{\mathrm{TV}}(P_0,P_1)$.
\end{lemma}

\begin{proof}[Proof of Lemma \ref{lemma:total-variation-two-point}]
	Using the fact that the transcript from the $j$-th server is $(\varepsilon_j,\delta_j)$-DP, Lemma \ref{lemma:impure-DP-to-approx-max-divergence} yields that
	\[
	D_\infty^{\bar{\delta}_j}(\P^{T^j}_0,\P^{T^j}_1) \leq \bar{\varepsilon}_j \text{ and } D_\infty^{\bar{\delta}_j}(\P^{T^j}_1,\P^{T^j}_0) \leq \bar{\varepsilon}_j.
	\]
	Using Lemma \ref{lemma:delta-approx-max-to-max-divergence} there exists $\tilde T^j_0$ and $\tilde T^j_1$ such that $D_{TV}(\P^{\widetilde T^j}_0,\P^{ \widetilde T^j}_0) \leq 2\bar{\delta}_j$ and $D_\infty(\P^{\tilde T^j}_0, \P^{\tilde T^j}_1) \leq \bar{\varepsilon}_j$ and $D_\infty(\P^{\tilde T^j}_1, \P^{\tilde T^j}_0) \leq \bar{\varepsilon}_j$.
	Let us define $\P^{\bar T}_i := \P^{\bar T^1}_i \times \ldots \times \P^{\bar T^m}_i$ where for $j \in[m]$ $\bar T^j_i := \widetilde T^j_i$ if $j\in S$ and $\bar T^j_i :=  T^j_i$ otherwise. By the triangle inequality,
	$$
	D_{\mathrm{TV}}(\P^T_0,\P^T_1) \leq D_{\mathrm{TV}}(\P^{\bar T}_0, \P^{\bar T}_1) + D_{\mathrm{TV}}(\P^T_0,\P^{ \bar T}_0) + D_{\mathrm{TV}}(\P^{\bar T}_1, \P^T_1) .
	$$
	For the second term, we have that
	\begin{align*}
		D_{\mathrm{TV}}(\P^{\bar T}_0, \P^{T}_0) &\leq \sum_{j \in [m]} D_{\mathrm{TV}}(\P^{\bar T^j}_0, \P^{T^j}_0)\\
		&= \sum_{j \in S} D_{\mathrm{TV}}(\P^{\widetilde T^j}_0, \P^{T^j}_0) \leq 2 \sum_{j \in S} \bar{\delta}_j .
	\end{align*}
	Similarly we have that $D_{\mathrm{TV}}(\P^{\bar T}_1,\P^{T}_1) \leq 2 \sum_{j \in S} \bar{\delta}_j$. We now tend to the term $D_{\mathrm{TV}}(\P^{\bar T}_0, \P^{\bar T}_1)$. By Pinskers inequality, the independence of the transcripts given the data generating process and the tensorization of the KL-divergence we have that
	\begin{align*}
		D_{\mathrm{TV}}(\P^{\bar T}_0,\P^{\bar T}_1) &\leq \sqrt{2 D_{\mathrm{KL}}(\P^{\bar T}_0,\P^{ \bar T}_1)}\\
		&= \sqrt{2 \sum_j D_{\mathrm{KL}}(\P^{\bar T^j}_0, \P^{\bar T^j}_1)}\\
		&= \sqrt{2 \sum_{j\in S} D_{\mathrm{KL}}(\P^{\widetilde T^j}_0, \P^{\widetilde T^j}_1) + \sum_{j\in S^c} D_{\mathrm{KL}}(\P^{ T^j}_0, \P^{ T^j}_1)}.
	\end{align*}
	 By Lemma \ref{lemma:max-divergence-bounded-by-kl-divergence} we have that $D_{\mathrm{KL}}(\P^{\widetilde T^j}_1,\P^{ \widetilde T^j}_0) \leq \bar{\varepsilon}_j(e^{\bar{\varepsilon}_j} - 1)$. Let us denote $P^j_i = P_i \times\ldots\times P_i$ ($n_j$ times), by a standard data processing inequality (see e.g. \cite{raginsky2016strong}), we have that 
	 \[
	 D_{\mathrm{KL}}( \P^{T^j}_0, \P^{T^j}_1) \leq  D_{\mathrm{KL}}( P^j_0,  P^j_1) = n_j  D_{\mathrm{KL}}( P_0,  P_1)
	 \]
	 which implies
	 \[
	 D_{\mathrm{TV}}(\P^{\bar T}_0, \P^{\bar T}_1) \leq \sqrt{2 \sum_{j\in S} \bar{\varepsilon}_j(e^{\bar{\varepsilon}_j} - 1) + \sum_{j\in S^c} n_j  D_{\mathrm{KL}}( P_0,  P_1)}.
	 \]
\end{proof}

\begin{proof}[Proof of Lemma \ref{lemma:total-variation-two-point_maintext}]
For two probability measures $P$ and $Q$ with $X \sim P$ and $\tilde{X} \sim Q$ and any coupling $\P^{X,\tilde{X}}$ of $(X,\tilde{X})$, it holds that
\begin{equation}\label{eq:coupling_bound_TV_general}
\| P - Q \|_{TV} \leq 2 \P^{X,\tilde{X}} \left( X \neq \tilde{X} \right),
\end{equation}
see e.g. Section 8.3 in \cite{thorisson2000coupling}. Hence, we can conclude that $D_{TV}(P_{f},P_{\tilde f}) \leq 2\rho$, where $\rho$ is defined in \eqref{eq:coupling_probability_bound}. 

Using the above fact, the proof now follows by a direct application of Lemma \ref{lemma:total-variation-two-point}, with $\P_0 = \P_{f}$ and $\P_1 = \P_{\tilde f}$. 
\end{proof}

\begin{proof}[Proof of Theorem \ref{thm:lower_bound-pointwise-risk}]
	We will be using the Le Cam two point method of \cite{Yu1997} to show the lower bound, i.e. Lemma 1 in the aforementioned paper. To that extend, we aim to show that there exists $f,\tilde{f} \in \cB^\alpha_{p,q}(R)$ such that the push-forward distributions $\P_f^T$ and $\P_{\tilde{f}}^T$ are close in total variation distance, for any $(\bm \varepsilon, \bm \delta )$-DP protocol $T$. 

	Consider $f$ such that $\|f\|_{\cB^\alpha_{p,q}} = R' < R$ and let $g:\R \to \R$ be any compactly supported function with $g(0) >0$, $\|g\|_1 >0$ and $g \in \cB^\alpha_{p,q}(R-R')$. Consider the two sequences of non-negative real numbers $\gamma_D := c_0^{-1} D^\nu$ and $\beta_n$ such that $\beta_n^\nu \gamma_D^{-1} = 1$ where we recall that $\nu = \alpha - \frac{1}{p}$ and define
	\[
	\tilde{f}(t) := \gamma_D^{-1} g(\beta_n(t-x_0)) + f(t).
	\]
	By Lemma 1 in \cite{cai2003rates}, it holds that
	\[
	\|\tilde{f}\|_{\cB^\alpha_{p,q}} \leq \gamma_D^{-1} \beta_n^\nu \|g\|_{\cB^\alpha_{p,q}} + \|f\|_{\cB^\alpha_{p,q}} \leq R.
	\]
	Next, we will aim to bound $\|\P_{\tilde{f}}^T - \P_{f}^T\|_{\mathrm{TV}}$. To that extent, let us denote by $(Y_i,X_i)$ and $(\tilde{Y}_i,\tilde{X}_i)$ random variables with marginals $P_{\tilde{f}}$ and $P_{f}$ respectively. We start by constructing a coupling as follows between these pairs of random variables as follows. 

	Let $U \sim \mathrm{Unif}(0,1)$ and set $X_i=\tilde{X}_i=U$, next for a given $U=u$ we find total variation coupling between $N(\tilde{f}(u),\sigma^2)$ and $N(f(u),\sigma^2)$. That is we have random variables $(Z_{i}(u),\tilde{Z}_{i}(u))$ such that $Z_{i}(u)\sim N({f}(u),\sigma^2)$ and $\tilde{Z}_{i}(u)\sim N(\tilde{f}(u),\sigma^2)$ 

	Set $Y_i := Z_i(U)$ and $\tilde{Y}_i := \tilde{Z}_i(U)$ respectively. Let $P_{f,\tilde{f}}$ the probability measure corresponding to coupling between $(Y_i,X_i)$ and $(\tilde{Y}_i,\tilde{X}_i)$ defined as above. We have that
	\begin{align*}
		P_{f,\tilde{f}}\left( \left(Y_i^{(j)},X_i^{(j)}\right) \neq \left(\tilde{Y}_i^{(j)},\tilde{X}_i^{(j)}\right) \right) &\leq \E\left[\mathbbm{1}\{({Z}_i(U),U) \neq (\tilde{Z}_i(U),U)\}\right]\\
		&= \E \left[\E\left(\mathbbm{1}\{{Z}_i(U) \neq \tilde{Z}_i(U)\} \mid U\right)\right]\\
		&\leq \E  \left[c \frac{|\tilde{f}(U)-f(U)|}{\sigma}\right]\\
		&= \frac{c}{\sigma} \|\tilde{f} -f\|_1.
	\end{align*}
Observe that the $L_1$-distance between $\tilde{f}$ and $f$ is given by $\gamma_D^{-1}\beta_n^{-1}\|g\|_1$. Next, we set  
$
\gamma_D := c_0^{-1} 2^{L\nu},
$
where $L$ is such that
\[
2^{L(2\nu+2)} = \sum_{j=1}^m (n_j^2\varepsilon_j^2) \wedge (n_j2^{L}).
\]
This choice automatically gives us a choice of $\beta_n$ using our construction that $\beta_n^\nu \gamma_D^{-1} = 1$. Next, we use Lemma \ref{lem:coupling_lemma} and Lemma \ref{lemma:total-variation-two-point_maintext}. First let us choose the set $S$ which we define as 
\[
S = \left\{j \in [n]: \varepsilon_j \leq \sqrt \frac{2^L}{n_j}\right\}.
\]
For $\bar{\varepsilon}_j$ with $j \in S$ we have
\begin{align*}
	\bar{\varepsilon}_j &= 6 n_j \varepsilon_j D_{\mathrm{TV}}(\P_{\tilde{f}}^T, \P_{f}^T)\\
	&=6c \sigma^{-1} n_j \varepsilon_j \gamma_D^{-1} \beta_n^{-1}\|g\|_1\\
	&= 6c \sigma^{-1} n_j \varepsilon_j c_0^{1 + \frac{1}{\nu}} 2^{-L(1+\nu)}\|g\|_1\\
	&= 6c \sigma^{-1} c_0^{1 + \frac{1}{\nu}} \frac{n_j \varepsilon_j }{\sqrt{\sum_{j=1}^m (n_j^2\varepsilon_j^2) \wedge (n_j2^{L})}}\|g\|_1\\
	&\leq 6c \sigma^{-1}\|g\|_1 c_0^{1 + \frac{1}{\nu}} \frac{n_j \varepsilon_j }{\sqrt{ (n_j^2\varepsilon_j^2) \wedge (n_j2^{L})}}\\
	&= 6c \sigma^{-1}\|g\|_1 c_0^{1 + \frac{1}{\nu}} \quad (\because j \in S).
\end{align*}
Now by suitably choosing $c_0$ in terms of $c,\sigma$, $\nu$ and $\|g\|_1$ we have that $\bar{\varepsilon}_j \leq 1$ and using the fact that $e^x - 1 \leq 2x$ for $x \in[0,1]$ we have that
\begin{align*}
	2 \sum_{j\in S} \bar{\varepsilon}_j(e^{\bar{\varepsilon}_j} - 1) &\leq 4 \sum_{j \in S} (\bar{\varepsilon}_j)^2\\
	&\leq 4 .(6c\sigma^{-1}c_0^{1+\frac{1}{\nu}}\|g\|_1)^2 \frac{\sum_{j\in S} n_j^2\varepsilon_j^2 }{\sum_{j=1}^m (n_j^2\varepsilon_j^2) \wedge (n_j2^{L})}\\
	&= c_0'\frac{\sum_{j\in S} n_j^2\varepsilon_j^2  \wedge n_j2^{L} }{\sum_{j=1}^m (n_j^2\varepsilon_j^2) \wedge (n_j2^{L})},
\end{align*}
where $c_0' = 4 .(6c\sigma^{-1}c_0^{1+\frac{1}{\nu}}\|g\|_1)^2$.
It is easy to show that via direct computation that $D_{\mathrm{KL}}(P_{\tilde{f}},P_{f}) \leq c_2 \sigma^{-2}\|\tilde{f}-f\|_2^2 $. Hence, we have that
\begin{align*}
	D_{\mathrm{KL}}(P_{\tilde{f}},P_{f}) &\leq c_2 \sigma^{-2} \gamma_D^{-2} \beta^{-1}_n \|g\|^2_2\\
	&=c_2 \sigma^{-2}\|g\|^2_2 c_0^{\frac{2+\nu}{\nu}} 2^{-L(2\nu + 1)}.
\end{align*}
For a small enough choice of $c_0$,
\begin{align*}
	\sum_{j \in S^c} n_j D_{\mathrm{KL}}(P_{\tilde{f}},P_{f}) &\leq  	\sum_{j \in S^c} n_j c_2 \sigma^{-2}\|g\|^2_2 c_0^{\frac{2+\nu}{\nu}} 2^{-L(2\nu + 1)}\\
	&\leq c_0''	\sum_{j \in S^c} n_j 2^L2^{-L(2\nu + 2)}\\
	&= c_0''\frac{\sum_{j\in S^c} n_j2^{L} }{\sum_{j=1}^m n_j^2\varepsilon_j^2 \wedge n_j2^{L}}\\
	&= c_0''\frac{\sum_{j\in S^c} n_j^2\varepsilon_j^2  \wedge n_j2^{L} }{\sum_{j=1}^m n_j^2\varepsilon_j^2 \wedge n_j2^{L}},
\end{align*}
where $c_0'' = c_2\sigma^{-2} \|g\|_2^2 c_0^{1 + 2/\nu}$.
Combining the above bounds with the condition on $\bm{\delta}$, we have that $\|\P_{\tilde{f}}^T - \P_{f}^T\|_{\mathrm{TV}}$ is bounded by
\begin{align*}
\sqrt{2}\sqrt{ \sum_{j\in S} \bar{\varepsilon}_j(e^{\bar{\varepsilon}_j} - 1) + \sum_{j \in S^c} n_j D_{\mathrm{KL}}(P_{\tilde{f}},P_{f})} + 4 \underset{j \in S}{\overset{}{\sum}} e^{\bar{\varepsilon}_j} n_j\delta_j \rho, &\leq \sqrt{2(c_0''\wedge c'_0)} + 4 e \sum n_j\delta_j\\
&\leq 2\sqrt{2(c_0''\wedge c'_0)},
\end{align*}
where we have used the fact that $\rho \leq 1$, $\bar\varepsilon_j \leq 1$ and $\sum n_j\delta_j = o(1)$. The RHS above can be made arbitrarily small by choosing $c_0$ small enough.

Hence we have by Lemma 1 in \cite{Yu1997} that the minimax rate is lower bounded by
\[
c'(\tilde{f}(x_0) - f(x_0))^2 = c'\gamma_D^{-2}g^2(0) \gtrsim 2^{-2L\nu},
\]
for a universal constant $c' >0$, which finishes the proof.
\end{proof}

\end{document}